\documentclass[reqno,10pt]{amsart}
\usepackage{mathrsfs}  
\usepackage{cancel} 
\usepackage[dvipsnames]{xcolor}
\usepackage{enumitem}
\usepackage{mathtools}
\usepackage{bm}
\usepackage{subcaption}
\usepackage{comment}
\usepackage{xargs}
\usepackage{soul}
\usepackage{subcaption}
\usepackage{amssymb}
\usepackage{tikz, pgf}
\usetikzlibrary{arrows.meta}
\usetikzlibrary{decorations.shapes}
\usetikzlibrary{decorations.pathreplacing}
\usetikzlibrary{decorations.pathmorphing}

\usepackage[textsize=tiny,textwidth=2.05cm]{todonotes}
\newcommandx{\task}[2][1=]{\todo[linecolor=blue,backgroundcolor=blue!30,#1]{#2}}
\newcommandx{\note}[2][1=]{\todo[linecolor=green,backgroundcolor=green!30,#1]{#2}}
\newcommandx{\error}[2][1=]{\todo[linecolor=red,backgroundcolor=red!30,#1]{#2}}
\usepackage{marginnote}
\usepackage[dvipsnames]{xcolor}
\usepackage{mathrsfs}
\usepackage{mathtools}
\usepackage{esint}
\usepackage{stackengine}
\usepackage{hyperref}
\usepackage{bbm}
\usepackage[titletoc]{appendix}

\newtheorem{theorem}{Theorem}[section]
\newtheorem{lemma}[theorem]{Lemma}
\newtheorem{proposition}[theorem]{Proposition}
\newtheorem{corollary}[theorem]{Corollary}

\theoremstyle{definition}
  \newtheorem{definition}[theorem]{Definition}
  \newtheorem{example}[theorem]{Example}

\theoremstyle{remark}
  \newtheorem{remark}[theorem]{Remark}
  
\newcommand{\dt}{\di t}

\newcommand{\dx}{\di x}

\newcommand{\eps}{\varepsilon}
\newcommand{\Hd}{\mathcal{H}^{d{-}1}}
\newcommand{\N}{\mathbb{N}}
\newcommand{\Z}{\mathbb{Z}}
\newcommand{\R}{\mathbb{R}}

\newcommand{\weaklystar}{\stackrel{*}{\rightharpoonup}}
\newcommand{\defas}{\coloneqq}

\usepackage{stackengine}

\newcommand*{\di}{\mathop{}\!\mathrm{d}}

\newcommand{\etamac}[1]{\eta^{#1}_{p,j}}
\newcommand{\sigmamac}[1]{\sigma^{#1}_{p,j}}

\DeclareMathOperator*{\argmin}{argmin}

\newcommand{\sym}{{\rm sym}}

\newcommand{\BBB}{\color{black}}

\newcommand{\EEE}{\color{black}}

\newcommand{\GGG}{\color{black}}

\newcommand{\RRR}{\color{red}}

\newcommand{\viola}[1]{\textcolor{blue!30!red}{#1}}

\newcommand{\verde}[1]{\textcolor{green!50!black}{#1}}

\newcommand{\ZZZZ}{\color{black}}

\renewcommand{\H}{\mathcal{H}}

\newcommand{\wto}{\rightharpoonup}

\newcommand{\mres}{\mathbin{\vrule height 1.6ex depth 0pt width 0.13ex\vrule height 0.13ex depth 0pt width 1.3ex}}

\newcommand{\M}{\mathfrak{M}}

\newcommand{\ol}{\overline}
\newcommand{\sm}{\setminus}

\newcommand{\tsubset}{\subset}

\newcommand{\Sn}{\mathbb{S}^{d{-}1}}
 \newcommand{\dhn}{\mathrm{d}\mathcal{H}^{d{-}1}}
 \newcommand{\hn}{\mathcal{H}^{d{-}1}}

 \newcommand{\SBV}{\mathrm{SBV}}
 \newcommand{\GSBV}{\mathrm{GSBV}}
 \newcommand{\BV}{\mathrm{BV}}
 \renewcommand{\M}{\mathcal{M}}
 \newcommand{\PC}{\mathrm{PC}}
 
 \newcommand{\Gnm}{\mathcal{G}_n^-}
 \newcommand{\A}{\mathcal{A}}
 \newcommand{\E}{\mathcal{E}}
 \newcommand{\F}{\mathcal{F}}
 \newcommand{\G}{\mathcal{G}}

\newcommand{\cosi}{\xrightharpoonup{\sigma_{\rm cf}}}
\newcommand{\Ldue}{2}
\newcommand{\muim}{\mu^i_-}

\newcommand{\Ti}{{t}}
\newcommand{\n}{{m }}
\newcommand{\kt}{{k_t}}

\DeclarePairedDelimiterX\setof[1]\{\}{#1}
\DeclarePairedDelimiterX\abs[1]\lvert\rvert{#1}
\DeclarePairedDelimiterX\norm[1]\lVert\rVert{#1}
\DeclarePairedDelimiterX\sprod[2]\langle\rangle{#1, #2}

\usepackage{enumitem}
\numberwithin{equation}{section}

\setlist[enumerate,1]{font=\normalfont}
\setlist[itemize,1]{font=\normalfont}
\newlist{thmlist}{enumerate}{1}
\setlist[thmlist]{label=(\roman{thmlisti}),
	ref=(\roman{thmlisti}),font=\normalfont,
	noitemsep}

\title[Quasistatic evolution of cohesive-type fracture]{Quasistatic evolution of cohesive-type fracture}

\subjclass[2020]{49J45, 49Q20,  74A45}
\keywords{Cohesive \EEE fracture,   quasistatic crack propagation,  irreversibility condition, maximal opening memory variable, free-discontinuity problems,  $\Gamma$-convergence.}

\author[V.~Crismale]{Vito Crismale}
\address[Vito Crismale]{
  Dipartimento di Matematica “G. Castelnuovo”, Sapienza Universita` di Roma, Piazzale A. Moro 2, I-00185, Rome, Italy}
\email{crismale@mat.uniroma1.it}

\author[M.~Friedrich]{Manuel Friedrich} 
\address[Manuel Friedrich]{Department of Mathematics, Johannes Kepler Universität Linz. Altenbergerstrasse 69,
4040 Linz, Austria}
\email{manuel.friedrich@jku.at}

\begin{document}
\begin{abstract} We prove the existence of globally stable quasistatic evolutions for a cohesive fracture model with unprescribed crack path and without any topological restriction, in arbitrary dimension.
The surface energy density is assumed to be concave and to exhibit an activation threshold, 
modeling depinning effects and fracture process zones in quasi-brittle materials. We devise a new notion of convergence for memory variables supported on evolving crack sets, inspired by $\sigma$-convergence in brittle fracture, guaranteeing compactness and lower semicontinuity properties.
In contrast to the brittle case, global stability is not preserved under passage to the limit because of oscillation and branching phenomena in the approximating cracks. To overcome this difficulty, we deviate from the classical   scheme for  proving \EEE energetic solutions by first proving the energy balance and convergence of the surface energies, and only afterwards recovering the   global stability condition.
    \end{abstract}
    \maketitle
    
    \section{Introduction}
  Many effects in materials are characterized by the fact that the force depends on the direction of the rate, but essentially not on its magnitude. Prominent examples include dry friction, damage, phase transformations in shape memory alloys, hysteretic behavior in magnetic, magnetostrictive, and ferroelectric materials, elastoplasticity, and fracture. Such phenomena can be formulated as so-called  {rate-independent systems}, meaning that a monotone temporal rescaling of the input leads to the same correspondingly rescaled output. Over the last decades, there has been tremendous progress in the formulation of variational mechanical problems in this setting, along with the derivation of existence results, see \cite{Mielke, Mielke2} and the references therein for a broad overview.

  \textbf{The broader context -- energetic solutions for rate-independent systems.}
Energetic   solutions \cite{MieThe99, Mielke2}  have emerged as a natural solution concept, with the significant advantage of relying only on energies and dissipation distances, rather than on their derivatives. Moreover, the formulation is sufficiently flexible to incorporate parameter-dependent problems and to derive effective theories in terms of (evolutionary) $\Gamma$-convergence \cite{MRS}. Energetic solutions are characterized by the \emph{static stability condition} (S) and the \emph{total energy balance} (E). The existence of solutions is obtained via a well-established program (see \cite[Table~2.1]{Mielke}), which is based on approximation by  time-incremental \EEE minimization problems. Central to the existence proof are two compatibility conditions between the energy and the dissipation, namely continuity of the power of the external forces and closedness of the stability set. Their combination allows one to prove  first \EEE (S) and then \EEE (E) and, a posteriori, additional convergence properties, in particular the convergence of the energies at all times in the passage from the time-discretized approximation to the time-continuous solution.

Establishing the closedness of the stability set is one of the key challenges in relevant applications. In particular, in brittle fracture, this can be reformulated as the so-called jump-transfer lemma or, equivalently, as a stability property of unilateral minimality. In more sophisticated situations such as cohesive fracture, a corresponding stability result beyond the one-dimensional setting is currently unavailable. In this work, for a specific model of cohesive-type fracture,   we verify the closedness of the stability set and 
 prove \EEE  the existence of energetic solutions.
From a methodological point of view, this is achieved by deviating from the standard  aforementioned proof \EEE steps.
Indeed, our strategy passes through a relaxed notion of (S), expressing that at each time the time-continuous limit configuration minimizes the $\Gamma$-limit of the energy functionals minimized in the incremental problems defining the discrete-time approximations. Then, exploiting a localization argument, we establish (E) and deduce 
energy convergence, \EEE which is usually obtained only a posteriori.
This additional information enables us to establish (S), which would be impossible without prior knowledge of energy convergence. We call this the   \emph{E--S approach}. We believe that this strategy may be useful in various contexts of rate-independent problems, and the present paper  illustrates \EEE
the effectiveness of this approach in the setting of cohesive fracture.

\textbf{The setting  -- variational \ZZZZ models \EEE for brittle and cohesive fracture.} Let us describe the setting in more detail. The propagation of  crack can be viewed as the result of a competition between   the minimization of bulk elastic energy  and dissipation associated with an infinitesimal increase of the cracked region. In a variational formulation, this leads to the study of energies of the form 
\begin{align}\label{mainfunctional}
\mathcal{E}( u \EEE,\Gamma) = \int_{\Omega \setminus \Gamma}    |\nabla u(x)|^2 \EEE \, {\rm d}x + \int_{ \Gamma}  g \EEE (x,[u](x),\nu_{u \EEE}(x)) \, {\rm d}\mathcal{H}^{d{-}1}(x),
\end{align}
where for simplicity we restrict ourselves to the generalized antiplane setting, i.e., $u \colon \Omega \to \R$ denotes a scalar-valued displacement field defined \EEE on a reference configuration  $\Omega \subset \mathbb{R}^d$ ($d\geq 2$). The energy features a  linearized \EEE elastic contribution depending on the displacement gradient on the  unfractured region $\Omega \setminus \Gamma$, as well as a surface term defined on a $(d{-}1)$-dimensional surface $\Gamma$, the crack set, where the density $ g \EEE$ \EEE 
depends on the crack opening  $[u]$ representing the difference of the traces of $u$ on the two sides of $\Gamma$. 

In the setting of \emph{brittle fracture}, the surface density  $ g \EEE (x, \cdot, \nu)$ \EEE is equal to  $  \kappa(x,\nu) \EEE >0$, referred to as the  (local) \EEE fracture toughness.  In many processes, the cohesive forces acting between the lips of the crack are not negligible and therefore \emph{cohesive zone models}, as proposed by Barenblatt \cite{barenblatt}, are more suitable. Neglecting for simplicity the dependence on $x$ and $\nu$,  this corresponds \EEE to a  nonnegative \EEE  increasing  density  $g$ \EEE explicitly depending on $[u]$. 
As  after the emergence of a macroscopic fracture the cohesive forces decrease for increasing opening, the function $g$ is usually assumed to be \ZZZZ (strictly) \EEE concave. 

{Existence and regularity results} for minimizers of \eqref{mainfunctional} under prescribed boundary conditions have been the subject of intensive research over the past decades, see \cite{BFM}  and \EEE \cite[Chapter~5]{Ambrosio-Fusco-Pallara:2000} \EEE for an overview  \BBB (see also \cite{FriLabSti, FriPerSol21} for the vectorial setting of linearized elasticity).   \EEE 
 In the case $g(0)=0$ and $g'(0^+) \in (0,\infty)$, a relaxation process in $\BV$ \EEE occurs, since for sufficiently large gradients it becomes energetically favorable to replace them by many infinitesimally small jumps, potentially leading to diffuse cracking of dimension higher than $d{-}1$, see \cite{BoBrBu, BCMG}. Although it has been shown that in one dimension minimizers belong to  the space of  Special functions of Bounded Variation $\SBV$ \EEE (see \cite{BCMG, BoCI}), corresponding results in higher dimensions currently seem out of reach. Else, \ZZZZ if \EEE  $g( 0^+ \EEE )>0$ or $g'(0^+) =\infty$, it is well known from \cite{Amb89} that minimizers belong to $\SBV$ (or suitable generalizations).

In the present paper,  we assume \EEE  $g( 0^+ \EEE)>0$, corresponding to an \emph{activation threshold}. 
From a mechanical point of view, this models  the additional energetic contribution that many materials display before standard cohesive propagation takes place, accounting for the finite energy required to nucleate or \emph{depin} the crack from  local heterogeneities in the underlying microstructure. Such a phenomenon naturally arises in the presence of fracture process zones (FPZs), where microcracking, bridging mechanisms, debonding, or other dissipative microscale processes precede the formation of a macroscopic crack.
 This is particularly relevant in composite media and \emph{quasi-brittle} materials (see e.g.\ \cite{Baz97, EGP02, ZZ22}), where a significant part of the dissipated energy is associated with the activation of the FPZ itself rather than with the subsequent crack opening,  and which model, for instance, \EEE
 fracture processes in fiber–matrix debonding in composites, delamination in laminated composites, or frictional cohesive fracture in geomaterials. 
 Indeed, in \cite{Bar18, BLZ, DMZe},   a cohesive zone model with activation threshold naturally emerges as the homogenization limit of a class of brittle composites exhibiting bridging mechanisms. A similar interpretation  appears in the Elastic--Damage--PlasticityDamage (E--D--PD) 
  regime discussed in \cite{alessi},    in the framework of phase-field approximation of cohesive fracture through elastoplastic-damage models (see e.g.\ \cite{AMV,     DMOT,    VHCD26} and references therein):    in this setting, the competition of damage -- initially predominant -- with plastic effects leads to a cohesive response including an additional depinning-type energy contribution.  We also  mention  some phase-field approximations of cohesive crack \ZZZZ energies \EEE  through elasticity-damage models with an activation threshold for the limit energy \cite{CC1,  FI,   Iu}. 
(See also  \cite{ACF1, ACF2, CFI1,  CFI24, CFI25} for the case  $g( 0^+ \EEE)=0$ in the limit energy density,  and \cite{Chambolle-Conti-Francfort-phase} for non-interpenetration.)

\textbf{State-of-the-art for quasistatic crack evolutions.} Within the framework of rate-independent systems,  the seminal paper by \EEE {\sc Francfort} and {\sc Marigo} \cite{frma98} introduced the notion of an \EEE \emph{irreversible quasistatic crack evolution} for the brittle case, i.e., $g \equiv \kappa$. This evolution is characterized by three key principles: monotone growth of the crack set, mechanical equilibrium at each instant of time, and an energy balance reflecting the absence of dissipation during the evolution process; the last two  correspond \EEE to (S) and (E) for energetic solutions.  In contrast to the theory previously developed starting from \cite{griffith},  the key feature of \cite{frma98}  lies in the fact that \EEE  the crack trajectory 
 is not \EEE
prescribed a priori.

The first rigorous existence result for the model proposed in \cite{frma98} was obtained in \cite{DM-Toa} for a two-dimensional antiplane setting under rather restrictive assumptions on the topology of the crack  (see \cite{Cha03} for a generalization to the vectorial setting). \EEE  These topological constraints were later removed
  by {\sc Francfort} and {\sc Larsen} in \EEE
\cite{Francfort-Larsen:2003}, where the problem was reformulated within a weak variational framework. This approach was subsequently extended to nonlinear elasticity \cite{dMasoFranToad, DalGiac}, including models incorporating non-interpenetration conditions \cite{Lazzaroni}, and, in two dimensions, to geometrically linear elasticity \cite{FriedrichSolombrino}.   Moreover, \cite{BabGia14} proved that in dimension two the weak evolution of \cite{Francfort-Larsen:2003} is strong, \ZZZZ i.e., \EEE the crack set is topologically closed for any time, extending the static regularity result \cite{DeGCarLea}.   Without aiming at exhaustiveness, we also refer to \cite{ALL20, DM-Toa2, Larsen} for related approaches based on local minimization principles, and mention approximation results for fracture evolutions in phase-field models \cite{Giacomini:2005}, eigenfracture \cite{ba duc}, and finite-element approximations \cite{seutter, GP1, GP2}. Moreover, crack evolutions have been identified as effective variational limits for atomistic models \cite{FriedrichSeutter}, homogenization \cite{GiacPonsi}, and linearization \cite{steinke}.

The derivation of corresponding existence results in the cohesive case is significantly more challenging 
since \EEE any meaningful irreversibility condition must take into account not only the history of the crack geometry, but also  that of \EEE the crack amplitudes through an internal variable. This leads to serious difficulties in establishing the global stability condition, and available results are currently restricted to dimension one or to models with prescribed crack  path. \EEE We briefly review the relevant literature focusing on the different choices of the internal \emph{memory variable} encoding irreversibility, referring to \cite{BFM} for an overview. 
In \cite{DMZ} and \cite{CaToa} the memory 
variable is the maximal crack opening: in \cite{DMZ} energy dissipation occurs only 
when the crack opening overcomes the maximal opening up to the current time, without any energy \ZZZZ recovery, \EEE  whereas \EEE 
in \cite{CaToa} some amount of energy is recovered while the crack opening decreases. In \cite{CriLaOr}, 
 the energy is dissipated also during the unloading phase, so the memory variable is the variation in time of the crack opening and \emph{fatigue} may produce cracks. \EEE In these  works, \EEE the existence of globally stable \EEE quasistatic evolutions
is proven.  Further existence results for quasistatic evolutions of cohesive-type fracture based on viscous approximations (in the abstract language of \cite{MRS16}) are obtained in \cite{Al, ACFS, Ca, NeSca} (see also related models \cite{TZ, BCFR}    in the framework of  delamination \cite{KMR}).  All the previous models assume prescribed crack path. 

{\sc Bonacini}, {\sc Conti}, and {\sc Iurlano} \cite{BoCI} \ZZZZ proved \EEE existence of globally stable quasistatic evolutions for a 1d model with unprescribed (0-dimensional) crack, with the maximal opening as memory variable and energy recovering in the unloading phase,  with an activation threshold for the dissipation. 
%
 Concerning variational evolutions beyond the one-dimensional setting and allowing for a free-discontinuity set, the most far-reaching result appears to be \EEE 
\cite{Giacomini:2005b}, where a cohesive-to-brittle transition for quasistatic fracture evolutions is established in the large-domain limit. Nevertheless, \ZZZZ also \EEE in this work no existence result for a genuinely cohesive continuous-in-time fracture evolution is obtained.

\textbf{The present paper.} The goal of our work is to establish an existence result for a quasistatic fracture evolution associated with a cohesive-type energy, without imposing any topological restrictions on the crack geometry.  This is performed in arbitrary dimensions and in the weak  setting of $\SBV$-functions. \EEE 
We view our paper as a proof of concept showing that the theory of energetic solutions is sufficiently flexible to also encompass free-discontinuity evolutions in cohesive fracture, provided that the proof strategy is suitably adapted, as outlined below. For this reason, we consider \BBB a scalar model with \EEE isotropic bulk and surface energy densities and uniformly bounded boundary data, even if more complicated settings could be treated at the price of  technical complications. In the spirit of the first existence result on cohesive fracture evolution in the setting of prescribed crack path \cite{DMZ},  we consider the maximal opening criterion.

\textbf{Proof strategy.} We employ the classical strategy based on approximating solutions by suitable time-discrete incremental minimization schemes, where the maximal opening enters as a memory variable, see \eqref{minimizing-scheme}. 
  An approximate energy   inequality \EEE
  is obtained by standard arguments,    providing uniform-in-time \EEE energy bounds  for the discrete-time approximations, \ZZZZ denoted by \EEE $(u_n(t), K_n(t), \psi_n(t))$. 
In turn, these bounds guarantee the precompactness of the displacements $u_n(t)$ at every fixed time.
For the discrete-time memory variables $\psi_n(t)$, supported on \ZZZZ the \EEE crack sets $K_n(t)$, one needs a notion of convergence preserving the three main properties of irreversibility, energy inequality  (i.e., \EEE lower semicontinuity of the dissipation is needed), and global stability  in the time-continuous limit $n \to \infty$. \EEE 

To this aim, we devise a new notion of convergence for functions defined on variable rectifiable sets of codimension one, denoted by $\sigma_{\rm cf}$-convergence  (`cf' standing for `cohesive fracture'). \EEE This draws inspiration from $\sigma$-convergence introduced by {\sc Giacomini} and {\sc Ponsiglione} \cite{GiacPonsi} (see also the related notions of $\sigma^2$-convergence \cite{dMasoFranToad} and $\sigma^2_{\rm \sym}$-convergence \cite{CriFri20}), developed 
for brittle fracture and applied to sequences of `discrete-time precrack' sets $K_n(t)$. In that setting, \ZZZZ one first proves \EEE \EEE the key property that the discrete-time functionals, where jump on $K_n(t)$ has no contribution,    $\Gamma$-converge \EEE to a functional still admitting an integral representation on $\SBV$  with a surface density $h^-$ depending on the material point $x$ and the normal $\nu$. \ZZZZ Then,  \EEE  the $\sigma$-limit $K(t)$, mechanically interpreted as the `limiting precrack', is defined as the maximal rectifiable set   where $h^-$  vanishes. \EEE

In our cohesive setting, we first prove the analogous integral representation result for the  $\Gamma$-limit \EEE  of the discrete-time functionals, where the surface density $h^-(x,\xi,\nu)$ now also depends on the crack opening $\xi$.   
(The arguments are essentially available \EEE in the literature, see  \cite{Sto lavoro GSBV, FM}. However, we give a self-contained proof  in a general setting, see  Appendix \ref{sec:App}.)  We then define the $\sigma_{\rm cf}$-limit $(K(t), \psi(t))$ by requiring that $K(t)$ is the union of the maximal rectifiable 0-sets  of $h^-$ \EEE over all nonzero jump openings, and $\psi(t)$ is the maximal Borel function supported on $K(t)$ such that $h^-(x,\xi, \nu_{K(t)}(x))=0$ for all $\xi$ with $g(|\xi|)\leq \psi(t,x)$, for $x \in K(t)$.
This notion of convergence guarantees irreversibility and lower semicontinuity. However, unlike in the brittle case, global stability is not preserved under passage to the limit,  as we explain now. 

A more direct approach to global stability in the brittle setting, basically equivalent to the $\Gamma$-convergence perspective, \EEE  is the celebrated \EEE  jump-transfer construction from \cite{Francfort-Larsen:2003}, which allows to transfer the jumps of a competitor onto the jump \ZZZZ sets \EEE of the approximating sequence. In the cohesive framework,  the strategy  of transferring jump \EEE may fail for two mechanisms,  illustrated in \EEE Figure~\ref{fig1} and Example~\ref{exAB}:
(A) if the length of the approximate jump is strictly larger than that of the limit jump;
(B) if the approximate jump consists of several flat connected components accumulating onto the limit \ZZZZ crack. \EEE   

\EEE
\begin{figure}[ht]
\begin{subfigure}{0.48\textwidth}
\begin{tikzpicture}[scale=0.7, every node/.style={font=\footnotesize},  baseline={(0,-0.7)}]
\def\amp{0.18}      
\def\xL{0}
\def\xR{4}
\draw[gray!40,dashed] (\xL,0)--(\xR,0);
\draw[blue!30!red, very thick]
  (\xL,0)
  -- (\xL+0.25,\amp)  -- (\xL+0.5,0)
  -- (\xL+0.75,-\amp) -- (\xL+1,0)
  -- (\xL+1.25,\amp)  -- (\xL+1.5,0)
  -- (\xL+1.75,-\amp) -- (\xL+2,0)
  -- (\xL+2.25,\amp)  -- (\xL+2.5,0)
  -- (\xL+2.75,-\amp) -- (\xL+3,0)
  -- (\xL+3.25,\amp)  -- (\xL+3.5,0)
  -- (\xL+3.75,-\amp) -- (\xR,0);
\draw[gray] (\xL,-0.05)--(\xL,0.05);
\draw[gray] (\xR,-0.05)--(\xR,0.05);
\node[gray] at (\xL,-0.3) {\scriptsize $0$};
\node[gray] at (\xR,-0.3) {\scriptsize $1$};
\node[blue!30!red, anchor=south west] at (\xL+1.3,\amp-0.25) {$\viola{K_n}$};
\node[green!50!black] at (\xL+2,-0.5) {$\verde{\psi_n}\equiv g(\theta)$};
\draw[decorate,decoration={brace,amplitude=4pt,raise=2pt}, gray!70]
  (\xL,\amp+0.15) -- (\xR,\amp+0.15);
\node[gray!70!black] at (\xL+2,\amp+0.6) {\scriptsize $\mathcal{H}^1(\viola{K_n})\to \kappa>1$};

\draw[-{Latex[length=2.5mm]}, thick] (\xR+0.4,0)--(\xR+1.5,0)
  node[midway, above] {\scriptsize $\sigma_{\rm cf}$};
\node[below] at (\xR+0.95,-0.05) {\scriptsize $n\to\infty$};

\def\xLR{\xR+1.8}
\def\xRR{\xR+5.8}
\draw[blue!30!red, very thick] (\xLR,0)--(\xRR,0);
\draw[gray] (\xLR,-0.05)--(\xLR,0.05);
\draw[gray] (\xRR,-0.05)--(\xRR,0.05);
\node[gray] at (\xLR,-0.3) {\scriptsize $0$};
\node[gray] at (\xRR,-0.3) {\scriptsize $1$};
\node[blue!30!red, anchor=south] at (\xLR+2,0.05) {$\viola{K}=(0,1)\times\{0\}$};
\node[green!50!black] at (\xLR+2,-0.35) {$\verde{\psi} \equiv g(\theta)  $};
\end{tikzpicture}
\caption{Oscillation}\label{fig1A}
\end{subfigure}
\begin{subfigure}{0.48\textwidth}
\begin{tikzpicture}[scale=0.7, every node/.style={font=\footnotesize}, baseline={(0,0)} ]
\def\sep{0.25}
\draw[gray!40,dashed] (0,0)--(4,0);
\draw[blue!30!red, very thick] (0,\sep)--(4,\sep);
\node[blue!30!red, anchor=south west] at (3.1,\sep) {$\viola{K_n^+}$};
\node[green!50!black, anchor=west] at (0.55,\sep+0.3) {$\verde{\psi_n}=\verde{g}(\theta_+)$};
\draw[blue!30!red, very thick] (0,-\sep)--(4,-\sep);
\node[blue!30!red, anchor=north west] at (3.1,-\sep) {$\viola{K_n^-}$};
\node[green!50!black, anchor=west] at (0.55,-\sep-0.3) {$\verde{\psi_n}=\verde{g}(\theta_-)$};
\node[gray] at (0,-0.5) {\scriptsize $0$};
\node[gray] at (4,-0.5) {\scriptsize $1$};
\draw[<->, gray!70] (-0.2,-\sep)--(-0.2,\sep);
\node[gray!70!black] at (-0.5,0) {\scriptsize $\tfrac{2}{n}$};

\draw[-{Latex[length=2.5mm]}, thick] (4.3,0)--(5.4,0);
\node[below] at (4.85,-0.05) {\scriptsize $n\to\infty$};
\node[above] at (4.85,0.05) {\scriptsize $\sigma_{\rm cf}$};

\def\xLR{5.8}
\def\xRR{9.8}
\draw[gray] (\xLR,-0.05)--(\xLR,0.05);
\draw[gray] (\xRR,-0.05)--(\xRR,0.05);
\draw[blue!30!red, very thick] (\xLR,0)--(\xRR,0);
\node[blue!30!red, anchor=south] at (\xLR+2,0.05) {$\viola{K}=(0,1)\times\{0\}$};
\node[green!50!black] at (\xLR+2,-0.35) {$\verde{\psi} \equiv  \verde{g}(\theta_+{+}\theta_-)$};
\node[gray] at (\xLR,-0.3) {\scriptsize $0$};
\node[gray] at (\xRR,-0.3) {\scriptsize $1$};
\end{tikzpicture}
\caption{Cumulation}\label{fig1B}
\end{subfigure}
\caption{The two mechanisms preventing jump transfer: in (A) the functions $\psi_n$ are constant, but the sets $K_n$ are highly oscillating; in (B) two branches  $K^-_n$, $K^+_n$ \EEE corresponding to cumulated jump  $0<\theta_-<\theta_+$ converge to a single branch  with accumulated jump $\ZZZZ \theta := \EEE \theta_+ + \theta_-$. A combination of (A) and (B) is also possible. \newline
($\alpha$) In both cases, the   energy cost in the limit for a competitor with crack $K$ and jump height $\theta+ \bar{\theta}$ is given by  $g(\theta+ \bar{\theta}) - g(\theta)$. In (A), transferring the new jump  on $K_n$ yields asymptotically the  energy cost  $\kappa (g(\theta+ \bar{\theta}) - g(\theta))$. In (B), transferring the new jump on $K_n^+$ (corresponding to the `mechanical precrack') yields the energy cost  $g(\theta^++ \bar{\theta}) - g(\theta^+)$. Using the strict concavity of $g$, in both cases this is bigger than $g(\theta+ \bar{\theta}) - g(\theta)$, whence the jump-transfer argument fails. \newline
($\beta$) Note that in both cases the surface energies do not converge: in (A), we have  $\mathcal{H}^1(K_n) \to \kappa > 1 =\mathcal{H}^1(K)$ while, in (B),  $g(\theta_+) + g(\theta_-) > g(\theta_+ + \theta_-)$.  
%
%
%
}\label{fig1}
\end{figure}
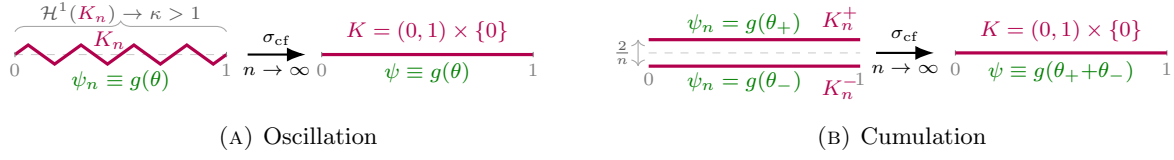

 In both cases, \EEE the transfer procedure necessarily produces an excess of energy,  see ($\alpha$) in the caption of Figure \ref{fig1}. \EEE
As a consequence,  when interpreted within the perspective of $\Gamma$-convergence, the stability induced by the $\Gamma$-limit of the  discrete-time functionals   is strictly weaker than the one required by the global stability property entering the definition of energetic solutions. We therefore distinguish between the `weak form' and the `strong form' of (S), cf.\ Definition \ref{def: unilat min}.
\EEE Our central observation is that both  mechanisms \EEE (A)--(B), which produce the gap between the weak and strong form of (S), correspond precisely to situations in which the surface energies of the approximating solutions fail to converge to the surface energy of the limit configuration,  see again Figure \ref{fig1}, in particular ($\beta$) in the caption. \EEE  In fact,   as a central property of $\sigma_{\rm cf}$-convergence, we  prove that the \EEE   strong form of (S) is preserved under $\sigma_{\rm cf}$-convergence provided one additionally assumes surface energy convergence (or at least convergence of the measure of the crack sets), \ZZZZ cf.\ Corollary \ref{cor. unilat}. \EEE Such convergence is naturally expected and has indeed been established in all relevant existence results for brittle fracture. However, it is usually obtained \emph{a posteriori}, after establishing (S) and (E), rather than \emph{a priori}, as required in the present setting. Consequently, our strategy departs from the standard existence scheme for energetic \ZZZZ solutions \EEE by following the E--S approach: we first establish (E)  together with \EEE energy convergence, and only afterwards recover the strong form of (S).

 We now describe the main idea of the argument. \EEE
Following the classical strategy for proving the reverse energy inequality, given $t\in [0,T]$  and \EEE a fixed error $\eps$, we partition the interval $[0,t]$ into nodes $(s_i)_{i=0}^k$ so that the power of the external loading is approximated by the corresponding Riemann sum on $(s_i)_{i=0}^k$ up to $\eps$.  We combine this \EEE with a suitable localization procedure on small cubes, which allows us to assume that the limit crack is locally almost flat within each cube. Given a cube $Q$, we identify the \emph{starting time} $s_p(Q)$ at which the crack `appears for the first time'. At this stage, the problem is not yet constrained by a preexisting crack, and it is therefore energetically favorable for a flat limit crack to be approximated by a single flat crack line  (up to an error of order $\eps$), excluding mechanisms (A) and (B). Since minimality is a \emph{global} property, whereas the absence of a preexisting crack at a given time $s_i$ can only be expected in certain cubes (namely, those associated with the corresponding starting time), a rigorous implementation of this strategy requires the weak form of (S) derived via $\Gamma$-convergence, which is precisely the notion of global minimality available to us at this stage of the proof. \EEE

More precisely, we test the weak form of (S) at each time $s_{i{-}1}$ against the configuration at time $s_i$, corrected by the backward increment of the boundary datum. This yields a weak version of the energy balance involving the dissipation associated with the weak form of (S).  Although this does not, in principle, yield the correct energy balance at time $s_i$, it allows us to deduce that, in all cubes $Q$ with starting time $s_p(Q)=s_i$, surface energy convergence holds inside $Q$, as desired. While this already captures a key aspect of the proof, substantial difficulties arise from the need to guarantee energy convergence also at later times $s_j$ with $s_j>s_p(Q)$. To achieve this, we proceed inductively in time to control the additional oscillation and branching that may arise at each subsequent time $s_j$, crucially exploiting the threshold condition $g( 0^+ \EEE)>0$. By an induction procedure, \EEE this  allows us to establish (E). Then, by the lower semicontinuity properties of $\sigma_{\rm cf}$-convergence, we deduce the convergence of the surface energies from the time-discrete approximations to the time-continuous evolution. \BBB This eventually allows us to deduce the strong form of \EEE  (S).  

The methods presented here are naturally tailored to this specific \ZZZZ setting of \EEE cohesive fracture. Nevertheless, we believe that the underlying strategy may prove effective not only for a broad class of cohesive fracture models, but also for the analysis of other dissipative phenomena, such as delamination and elastoplasticity.

 \textbf{Outline of the paper.}  In Section \ref{sec: evo} we present the setting and the main results. Section \ref{sec:prel} is devoted to some preliminaries, in particular the formulation of an integral representation result for energies with precrack. In Section \ref{subsec:sigmaconv} we introduce the novel notion of set-function convergence, called $\sigma_{\rm cf}$-convergence. Here, we also discuss the concepts of weak and strong unilateral minimality. Section \ref{sec: main proof} then contains the proof of our main result on the existence of quasistatic crack evolutions. Finally, in Appendix \ref{sec:App} we provide a self-contained proof of the integral representation result presented in Section \ref{sec:prel}.  \EEE

  \textbf{Notation.} We close the introduction with some general notation. \EEE For the general notions on \ZZZZ $\SBV$-functions  \EEE and their fine properties we refer to \cite{Ambrosio-Fusco-Pallara:2000}. For $u\in \ZZZZ \SBV \EEE $,  $\mathrm{D}u$ \ZZZZ  denotes \EEE the distributional derivative of $u$ which decomposes as $\mathrm{D}u= \nabla u \mathcal{L}^d + \ZZZZ [u] \otimes \nu_u \mathcal{H}^{d{-}1} \mres_{J_u}   \EEE $ into absolutely continuous and singular part with respect to the
Lebesgue measure, $\nabla u$ \ZZZZ being \EEE the approximate gradient of $u$.
Moreover, $J_u$ denotes the jump set of $u$, which \ZZZZ coincides \EEE up to \ZZZZ an \EEE $\hn$-negligible set with the set of approximate discontinuity points of $u$, and $\nu_u$ \ZZZZ is \EEE the measure theoretic normal to $J_u$.
The symbols $u^\pm$ denote the one-sided approximate limits of $u$ at a point of $J_u$, from the side of $\pm \nu_u$, \ZZZZ and we set $[u] := u^+-u^-$. \EEE We recall that, for $p>1$, $\SBV^p$ is the space of $\SBV$ functions with $\nabla u$ in $L^p$ and $\hn(J_u)$ finite.   We \BBB say that $u_n$ converges weakly to $u$ in $\SBV^p$, and \EEE  write $u_n \wto u$ in $\SBV^p$,  if $(u_n)_n \subset \SBV^p$ is such that $\|u_n\|_1 + \|\nabla u_n\|_p + \hn(J_{u_n})$ is bounded and $u_n \to u$ in $L^1$. (In particular,  this implies $u\in \SBV^p$ and $\nabla u_n \wto \nabla u$ in $L^p$.) \ZZZZ We denote by $\M_b^+(\Omega')$ the class of positive bounded Radon measures on $\Omega'$. \EEE

\BBB 
Given two countably $(\Hd, d{-}1)$ rectifiable (in the sequel we will say just \emph{rectifiable}) \EEE sets $\Psi_1, \Psi_2$ with $\mathcal{H}^{d{-}1}(\Psi_i) <+\infty $ for $i=1,2$, we write $\Psi_1 \subset \Psi_2$ if $\mathcal{H}^{d{-}1}(\Psi_1 \setminus \Psi_2) = 0$. Moreover, given $\Psi$ and  two Borel functions $f_i\colon \Psi \to \R$, we write $f_1 \le f_2$ if $\mathcal{H}^{d{-}1}(\lbrace f_1 > f_2\rbrace) = 0$. \ZZZZ We also write $\Psi_1 = \Psi_2$ if $\Psi_1 \subset \Psi_2$ and $\Psi_1 \supset \Psi_2$, and analogously $f_1 =f_2$. \EEE   A measure theoretic unit normal vector of  $\Psi$ is denoted by $\nu_\Psi$.  \EEE   Pointwise evaluations of Borel functions are  always \EEE  meant in the sense of precise representatives.    
We denote the scalar product on $L^2(  \Omega' \EEE )$ by $\langle\cdot , \cdot  \rangle$. We set $a  \vee \EEE b:=\max\{a,b\}$ and $a^+:=a\vee 0$.
For any $x \in \R^d$, $\varrho>0$,  $\nu \in \Sn$,   we denote by  $Q_\varrho^\nu(x)$  the cube centered in $x$, with sidelength $\varrho$ and two faces normal to $\nu$.  Moreover, we let  $Q_\varrho^{\nu,+}(x) := \lbrace y\in Q^\nu_\varrho(x) \colon (y- x)\cdot \nu > 0\rbrace$ and,  for any $\xi \geq 0$, 
$\ol u_{x,\xi,\nu}:= \xi \chi_{Q_\varrho^{\nu,+}(x)}$. \EEE

\section{Setting and main result}\label{sec: evo}

This section is devoted to the formulation of our main result. Let $\Omega \subset \R^d$ \BBB be \EEE a bounded Lipschitz domain. In order to \BBB incorporate \EEE boundary conditions on a part $\partial_D \Omega \subset \partial \Omega$ of the boundary, we \BBB follow the standard procedure of imposing \EEE boundary conditions in a \emph{neighborhood} of the boundary. \BBB To this end, \EEE we suppose that there exists another Lipschitz set ${\Omega'} \supset \Omega$ with  $\partial_D \Omega = \partial \Omega \cap {\Omega'}$ such that also  $\Omega' \setminus \overline{\Omega}$ is Lipschitz. \BBB We \EEE  consider the  cohesive-type energy  
\begin{align}\label{energy-static-new}
E(u)\defas  \int_{\Omega'}|\nabla u |^2 \,    {\rm d}x  + \int_{J_u} g(|[u]|) \, {\rm d}\Hd
\end{align}
\BBB for \EEE $u\in SBV^2(\Omega')$, \BBB where \EEE the surface density $g \colon [0, \infty) \to [0,\infty)$ is continuous \ZZZZ and \EEE  increasing \ZZZZ on $(0,\infty)$, \EEE and  satisfies
\begin{align}\label{gproperty}
g(0) = 0, \quad \BBB   g(0^+)  := \EEE  \lim_{t \searrow 0}\EEE g(t) =  c_g, \quad\quad  g(a) + g(b) \ge g(a+ b) + c_g \ \text{ for all } a,b >0 
\end{align}
as well as
\begin{align}\label{gproperty2}
g(a + t) \le  (g(a) + C_g t) \wedge  C_g \EEE \ \  \text{ for all } a,t >0 
\end{align}
for some constants $C_g,\,c_g  >0$.   In particular, as $g$ is continuous, increasing, and subadditive \ZZZZ on $(0,\infty)$, \EEE   $E$ is lower semicontinuous with respect to the $L^1(\Omega')$-topology,  see \cite[Section~5.4, Example~5.23]{Ambrosio-Fusco-Pallara:2000}.  

\BBB We will assume boundary conditions $u =w$ on $\Omega' \setminus \overline{\Omega}$ for some  boundary datum $w\in H^1( \Omega') \cap L^\infty(\Omega')$. In this case, we \EEE  note that it is actually irrelevant whether the elastic energy is taken over $\Omega$ or $\Omega'$ as the difference is merely a constant, namely $\int_{\Omega' \setminus \overline{\Omega}}|\nabla w |^2 \,    {\rm d}x$.  The evolution on a time horizon $[0,T]$ will be driven by a time-dependent boundary condition $w\in W^{1,1}(0,T;H^1( \Omega'))$ which satisfies, \BBB for some $M>0$, \EEE
\begin{align}\label{uniformnound}
\Vert w(t) \Vert_\infty \le  M \EEE \quad \text{for all }t \in [0,T]. 
\end{align}
\BBB Next,  we introduce \emph{pairs of set-function} 
\begin{align}\label{mathcalcdef} 
\mathcal{C} := \big\{ \text{$(K,\psi) \colon $    $ K\subset  \Omega \cup  \partial_D \Omega \EEE $ rectifiable with $\mathcal{H}^{d{-}1}(K) < + \infty$,  \ \   $\psi \colon K \to [0,\infty)$ Borel} \big\} 
\end{align}
and, given a pair  $(K,\psi)  \in \mathcal{C}$, \EEE  we  define  the  history-dependent  \emph{cohesive-type fracture energy}\EEE
\begin{align}\label{eq: lim-en}
    \mathcal{E}(u,K,\psi):=
       \int_{\Omega'} |\nabla u|^2 \, {\rm d}x +  \int_{K \cap J_u}   \big(g(|[u]|)   \vee \EEE \psi\big)\EEE \, {\rm d}\Hd +   \int_{ J_u \sm K}     g(|[u]|) \, {\rm d}\Hd
  \end{align}
    for each $u \in \SBV^2 (\Omega')$. \BBB Here, \EEE $K$ represents the cracked surface \BBB and $\int_K \psi \, \ZZZZ {\rm d}\mathcal{H}^{d{-}1}\EEE$ the dissipated surface energy up to some time in the evolution, i.e., the \emph{history}, see also \eqref{17050743} below. In particular, \emph{admissible functions} \ZZZZ $u$ \EEE for \ZZZZ a \EEE given boundary datum $w$ and $(K,\psi)  \in \mathcal{C}$ are required to  satisfy   \EEE
    \begin{equation}\label{def:ADgH}
  \ZZZZ u \EEE =w \text{ on } \Omega' \setminus \overline{\Omega}.
  \end{equation} 
and 
\begin{align}\label{constraints}  
 J_u  \subset  K, \quad \quad  g(|[u]|) \le \psi \text{ on }  J_u. 
  \end{align}
  We denote the class of functions $\ZZZZ u \EEE \in \SBV^2(\Omega')$ satisfying \eqref{def:ADgH}--\eqref{constraints}  by  $AD(w,K,\psi)$. \EEE           Note that for $K= \emptyset$ the energy \eqref{eq: lim-en} coincides with the one in \eqref{energy-static-new}.  Moreover, if  $u \in AD(w,K,\psi)$, then \eqref{eq: lim-en} simplifies to 
 \begin{align}\label{eq: lim-en-an} 
  \mathcal{E}(u,K,\psi)=     \int_{\Omega'} |\nabla u|^2 \, {\rm d}x +   \int_{K} \psi \, {\rm d}\Hd .
 \end{align}
 We finally   highlight that  the crack sets $K$ may intersect the Dirichlet boundary $\partial_D \Omega$,  but $K \cap (\Omega' \setminus \overline{\Omega}) = \emptyset$.  \EEE

 We now formulate the main result of this paper.

 \begin{theorem}[Quasistatic crack \BBB  evolution of cohesive-type fracture\EEE]\label{main def}
 Let  $w  \in W^{1,1}(0,T;H^1( \Omega')) $  \BBB satisfy \EEE \eqref{uniformnound}. \BBB Then, \EEE there exists an \emph{irreversible quasistatic crack evolution} with respect to the boundary condition  $w$, i.e., there exists a mapping  $t\to (u(t),K(t),\psi(t))$ with \BBB $(K(t),\psi(t)) \in \mathcal{C}$ and \EEE  $u(t) \in AD(w(t),K(t),\psi(t))$ for all $t \in [0,T]$ such that the following  four   conditions hold:
  
  \begin{itemize}
   \item[(a)] \emph{Initial condition}: $u(0)$ minimizes the energy $E$ given in \eqref{energy-static-new}   among all $v \in \SBV^2(\Omega')$ with
$v = w(0)$ on $\Omega' \setminus \overline{\Omega}$.  
  \item[(b)] \emph{Irreversibility}:   $K(t_1) \subset  K(t_2)$ and $\psi(t_1) \le \psi(t_2)$ on  $K(t_1)$   for all $0\leq t_1\leq t_2\leq T$.
  \item[(c)] \emph{Global stability}:    For every $t \in (0,T]$,   for  every  $(H,\zeta) \BBB \in \mathcal{C} \EEE $ with $K(t) \subset H $ and  $\psi(t) \le \zeta$ on $K(t)$,  and  for   every $v\in AD(w(t),H,\zeta)$    it holds that 
      \begin{equation}\label{finalstability}
       \mathcal{E}(u(t),K(t),\psi(t))  \leq \mathcal{E}(v,H,\zeta)\,.
      \end{equation}
      \item[(d)]   \emph{Energy balance}:    The function $t\mapsto \mathcal{E}(u(t),K(t),\psi(t))$ is absolutely continuous and it holds that
      \begin{equation*}\label{energybalance}
          \frac{\rm d}{{\rm d}t} \mathcal{E}(u(t),K(t),\psi(t)) =  \ZZZZ 2 \EEE \int_{\Omega'} \EEE \nabla u(t) \cdot   \partial_{t} \nabla w(t)\, {\rm d}x \quad \text{for a.e.\ $t \in [0,T]$}\,,  
      \end{equation*} 
      where by $\partial_t$ we denote the time derivative of $w$.
  \end{itemize}
\end{theorem}   
\begin{remark}
The  evolution $t\to (u(t),K(t),\psi(t))$ constructed in \BBB the \EEE proof  satisfies
  \begin{equation}\label{17050743}
  K(t)=\bigcup_{ s \in [0,t]} J_{u(s)}, \quad \psi(t)=\mathrm{ess\,sup}_{s \in [0,t]} \,  g(|[u(s)]|).
  \end{equation} 
\end{remark}

%
%
%
%
%
%
%
%

 \EEE

\section{Preliminaries}\label{sec:prel}
 
 In this section, we present an integral representation  result \EEE for energies with precrack and collect some basic slicing properties.

\subsection{Integral representation for energies with precrack}

As before, let  \EEE  $\Omega \subset \R^d$ be a bounded Lipschitz domain and let  $\Omega' \supset \Omega$ be another bounded Lipschitz domain such that also $\Omega' \sm \ol \Omega$ has Lipschitz boundary.  Define the Dirichlet boundary \EEE $\partial_D \Omega:= \Omega' \cap \partial \Omega$.  Let $1 < p  <  \infty$. \EEE 
 By $\mathcal{A}(\Omega')$ 
 we denote the family of open 
 subsets of $\Omega'$, 
 and we let 
\begin{equation*}
{\rm PC}(A):=\big\{v \in \SBV(A) \colon \nabla v=0 \text{ a.e.\ in }A,\, \hn(J_v)<+\infty\big\} \quad\text{for }A \in \A(\Omega'),
\end{equation*}  
\begin{equation*}
\SBV^p_{{w}}(\Omega'):=\big\{v \in \SBV^p(\Omega') \colon v={{w}} \text{ on } \Omega' \sm \ol \Omega\big\} \quad \text{for }\BBB {{w}} \EEE  \in W^{1,p}(\Omega').
\end{equation*}
We consider \BBB a sequence $(K_n,\psi_n)_n \subset \mathcal{C}$ (see \eqref{mathcalcdef}) \EEE  such that  $\Hd(K_n) \le C_0$ and $\int_{K_n} \psi_n \, {\rm d}\Hd \le C_0$ for some $C_0 \ge 0$. 
\BBB We recall that this pair represents the \emph{precrack} which is already present in the material and the corresponding density of  dissipated surface energy. \EEE Let $g \colon [ 0 ,+\infty) \to [0,+\infty)$  be \EEE continuous \ZZZZ and \EEE  increasing \ZZZZ on $(0,\infty)$, \EEE and satisfying \eqref{gproperty}--\eqref{gproperty2}. In particular, note that $g$ and $g-c_g$ are subadditive.

\BBB 
Consider  $w \in W^{1,p}(\Omega')$ and $(w_n)_n \subset W^{1,p}(\Omega')$ with $\|w\|_{\infty}\leq M$ and $ \|w_n\|_{\infty}\EEE\leq M$ for all $n \in \N$ \EEE such that $w_n \to w$ in $W^{1,p}(\Omega')$.  \BBB Our goal is to study the asymptotic behavior of the functionals 
$$\int_{\Omega'} |\nabla u|^p \dx + \int_{J_u}   \big(  g(|[u]|) - \psi_n \big)^+ \EEE \, \dhn$$
for $u \in \SBV^p_{w_n}(\Omega')$, where \EEE by convention,  for $A \in \mathcal{A}(\Omega')$, we set 
 \begin{align}\label{convention}
\int_{A \cap J_u}  \big(  g(|[u]|) - \psi_n \big)^+ \EEE \, \dhn \equiv \int_{A \cap (J_u \cap K_n)}   \big(   g(|[u]|) - \psi_n \big)^+    \EEE \, \dhn + \BBB  \int_{A \cap (J_u \setminus K_n)} \EEE g(|[u]|) \, \dhn .
\end{align}
\BBB (This corresponds to extending $\psi_n \equiv 0$ outside $K_n$.) As we employ the localization technique of $\Gamma$-convergence, following the approach in \cite{GiacPonsi}, it is convenient to consider a version of this energy where jumps on $\Omega' \setminus \overline{\Omega}$ can occur but are highly penalized. We introduce   
$ \E_n^- \colon L^1(\Omega') \times \A(\Omega') \to [0,+\infty]$ by
\begin{equation}\label{samenotation}
\E_n^-(u,A):= \int_A |\nabla u|^p \dx + \int_{A \cap J_u\cap \overline{\Omega}}   \big(  g(|[u]|) - \psi_n \big)^+ \EEE \, \dhn \BBB +  ( \EEE C_g  +1 )  \hn\big(\EEE(A \cap J_u ) \sm \ol \Omega\big) \EEE  \quad \text{\BBB if \EEE $\BBB u|_A \EEE \in \SBV^p(A)$}  
\end{equation}
and $+\infty$ otherwise. \BBB Note that jumps   in $\Omega' \setminus \overline{\Omega}$ are possible, but they are more expensive compared to the ones  in $\ZZZZ \ol\Omega \EEE$ as $\sup g \le C_g  < C_g + 1$, see \eqref{gproperty2}. \EEE Similarly, we introduce $\H_n^-  \colon L^1(\Omega') \times \A(\Omega') \to [0,+\infty]$ as the surface part of this energy on piecewise constant  functions, \EEE namely 
\begin{equation}\label{hnm}
\H_n^-(u, A):= \int_{  A \cap J_u \EEE \cap \ol \Omega}  \big( g(|[u]|) - \psi_n\big)^+   \, \dhn +  ( \EEE C_g  +1 )  \hn\big(\EEE(A \cap J_u ) \sm \ol \Omega\big) \EEE \quad \BBB \text{if } u|_A \EEE \in \PC(A) 
\end{equation}
and $+\infty$ otherwise.  We further let  
\begin{equation*}\label{0812242009}
\mathbf{m}^{X}_{\F}(\ol v, A):=\inf_{v \in X(A)}\big\{\F(v, A)\colon v=\ol v \text{ in a neighborhood of } \partial A\big\} \ \text{ for }\ol v \in X(A), \, A \in \mathcal{A}(\Omega'),
\end{equation*}
where $X$ stands for a suitable functional space (typical choices are $X=W^{1,p}, \, \PC$) and  $\F$ for a suitable functional.
Given $F_n,\, F \colon L^1(\Omega') \times \A(\Omega') \to [0,+\infty]$,  following \EEE \cite[Definition~16.2 and Proposition~16.4]{DMLibro}, we say that $F_n \colon L^1(\Omega') \times \A(\Omega') \to [0,+\infty]$ $\ol \Gamma$-converge 
to $F$ if $F_n(\cdot, A)$ $\Gamma$-converge to $F(\cdot, A)$ in the strong \BBB $L^1(A)$-topology \EEE for any $A \in \A(\Omega')$. \BBB We also recall the notation $\ol u_{x,\xi,\nu}:= \xi \chi_{Q_\varrho^{\nu,+}(x)}$ for any $x \in \Omega'  $, $\xi \in   \R $, $\nu \in \Sn$, \ZZZZ and $\varrho >0$. \EEE  We come to the main statement of this section.

\begin{theorem}\label{thm:convpreliminaries}
Under the above assumptions, there exist two functionals $\H^-$, $\E^- \colon L^1(\Omega') \times \A(\Omega') \to [0,+\infty]$ such that, up to a (not relabeled) subsequence, $\H_n^- $ and $\E_n^- $ $\ol \Gamma$-converge  to $\H^- $ and $\E^- $, with
\begin{equation}\label{Hnu}
\H^-(u,A)= \int_{A\cap J_u} h^-(\cdot, [u],\nu_u) \,\dhn \quad\text{for }u \in \PC(\Omega'),
\end{equation}
\begin{equation}\label{Enu}
\E^-(u,A)= \int_A |\nabla u|^p\dx+\int_{A\cap J_u} h^-(\cdot, [u],\nu_u) \,\dhn \quad\text{for } \BBB u \in \SBV^p(\Omega'), \EEE
\end{equation}
  where for any $x \in \BBB \Omega' \EEE $, $\xi \in  \R \EEE$, \ZZZZ and \EEE $\nu \in \Sn$ \BBB we have \EEE
\begin{equation}\label{1004261901}
\begin{split}
h^-(x, \xi, \nu) \BBB = \EEE \limsup_{\varrho \to 0^+} \frac{\mathbf{m}_{\H^-}^{\PC}(\ol u_{x,\xi,\nu}, Q^\nu_\varrho(x))}{\varrho^{d{-}1}}.
\end{split}
\end{equation}
Moreover,
\begin{equation}\label{12040854}
h^-(x, \xi, \nu) =\limsup_{\varrho \to 0^+}  \liminf_{n\to \infty} \frac{\mathbf{m}_{\H^-_n}^{\PC}(\ol u_{x,\xi,\nu}, Q^\nu_\varrho(x))}{\varrho^{d{-}1}}=\limsup_{\varrho \to 0^+}  \limsup_{n\to \infty} \frac{\mathbf{m}_{\H^-_n}^{\PC}(\ol u_{x,\xi,\nu}, Q^\nu_\varrho(x))}{\varrho^{d{-}1}}.
\end{equation}
\end{theorem}

\BBB
\begin{corollary}\label{cor: bc}
In the setting of Theorem \ref{thm:convpreliminaries}, given $w \in W^{1,p}(\Omega')$ and $(w_n)_n \subset W^{1,p}(\Omega')$ with  $w_n \to w$ in $W^{1,p}(\Omega')$, we define  $\E_n^{\partial, -}, \E^{\partial, -}  \colon L^1(\Omega') \times \A(\Omega') \to [0,+\infty]$ by
$$\E_n^{\partial, -}(u) := \begin{cases} \E_n^{  -}(u,\Omega')  & \text{if } u \in \SBV^p_{w_n}(\Omega'), \\ + \infty& else, \end{cases} \quad \quad 
\E^{\partial, -}(u) := \begin{cases} \E^{  -}(u,\Omega')  & \text{if } u = w \text{ on } \Omega' \setminus \overline{\Omega}, \\ + \infty& else. \end{cases}
$$
Then  $\E_n^{\partial, -}$  $\Gamma$-converge  to $\E^{\partial, -}$ in the   strong  \ZZZZ  $L^1(\Omega')$-topology. \EEE 
\end{corollary}
We omit the proof of Corollary \ref{cor: bc} as it follows from arguments in \cite{GiacPonsi}. \EEE  Indeed,  with the  presence of boundary conditions, the $\Gamma$-liminf for $\E_n^{\partial, -}$ is immediate and for the $\Gamma$-limsup  we can  argue as in \cite[Lemma~7.1]{GiacPonsi} to show that, due to the fact that $g \le C_g$ (see \eqref{gproperty2}) and the prefactor $(C_g +1)$ in \eqref{samenotation},  recovery sequences have asymptotically vanishing jump in $\Omega' \setminus \overline{\Omega}$ and can  thus be altered by a cut-off construction to attain the boundary condition $w_n$ on $\Omega' \setminus \overline{\Omega}$. \EEE 

\BBB Let us now come to the proof of Theorem \ref{thm:convpreliminaries}. \EEE The proof makes use of two general $\Gamma$-convergence  and integral representation results for free-discontinuity problems, namely the representation of energies defined on $\mathrm{(G)SBV}^p$-functions \cite{Sto lavoro GSBV} and piecewise rigid functions \cite{FM}, along with a suitable handling of the precrack (see \cite{GiacPonsi}). In our setting, only simplified versions of the results in \cite{Sto lavoro GSBV, FM} are needed, for which  we include  self-contained statements and proofs in Appendix~\ref{sec:App}.

 As a preparation, we need the following two lemmas which will also be proved in Appendix \ref{sec:App}. The first is an integral representation result  and the second one \BBB is \EEE a variant of a `classical' fundamental estimate.   \BBB Both statements take \EEE the presence of $(K_n, \psi_n)$   \BBB   into account whence a perturbation argument is necessary to obtain uniform positive lower bounds for the surface energy densities. \EEE 
 
 \begin{lemma}[Integral representation]\label{extra-lemma}
Suppose that $\H_n^-$  $\ol \Gamma$-converge   to $\H^- \colon L^1(\Omega') \times \A(\Omega') \to [0,+\infty]$. Moreover, \BBB consider a countable set $I \subset (0,1)$ with $0 \in \overline{I}$ such that  there exists a subsequence (not relabeled) for which  for every $\eps \in I$  the functionals defined by $\H_n^{\varepsilon,-} \colon L^1(\Omega') \times \A(\Omega') \to [0,+\infty]$ as
\begin{align}\label{HHNN}
\H_n^{\varepsilon,-}(u, A):=\H_n^-(u, A)+ \varepsilon \,\hn(A \cap J_u \cap K_n) \quad \text{if } u|_A \in \PC(A)
\end{align}
and $+\infty$ otherwise, \EEE $\ol \Gamma$-converge to $\H^{\varepsilon,-} \colon L^1(\Omega') \times \A(\Omega') \to [0,+\infty]$, where  
 $${ \H^{\varepsilon,-}(u,A)= \int_{A\cap J_u} h^{\varepsilon,-}(\cdot, [u],\nu_u) \,\dhn \quad \text{ for }  u \in \PC(\Omega') \EEE  }$$
 with $h^{\varepsilon,-}(x, \xi,\nu)=\limsup_{\varrho \to 0^+}  \varrho^{-(d{-}1)} \mathbf{m}_{\H^{\varepsilon,-}}^{\PC}(\ol u_{x,\xi,\nu}, Q^\nu_\varrho(x))$ for all $x \in \Omega'$, $\xi \in \R$, and $\nu \in \mathbb{S}^{d{-}1}$.
 Then, $\H^-$ is of the form  \eqref{Hnu} with   $h^-$  \ZZZZ given as in \eqref{1004261901}. Moreover, it holds that 
 $h^-(x,\xi,\nu) =   \lim_{\varepsilon\to 0} h^{\varepsilon,-}(x,\xi,\nu)$
 for all $x \in \Omega'$, $\xi \in \R$, and $\nu \in \mathbb{S}^{d{-}1}$.
 \end{lemma}

 \BBB In the proof of  Theorem \ref{thm:convpreliminaries} below, the $\ol \Gamma$-convergence of the sequences $(\H_n^{\varepsilon,-})_n$ defined in \eqref{HHNN} will be guaranteed by an application of \cite{FM}. In this context, we also need a fundamental estimate for the functionals   $(\H_n^{\varepsilon,-})_n$. \EEE
 
\begin{lemma}[Fundamental estimate]\label{le:fundestHnm}
\BBB Let $\eps \in I$, \EEE let $\eta, \xi >0$, let  $A'$, $A$, $B \in \mathcal{A}(\Omega')$ with $A' \subset \subset A$, and let  $S:= (A\sm A') \cap B$. \BBB Then, there exists \EEE  $\Lambda>0$ \BBB and a continuous function $\omega_g^\eps \colon [0,\infty) \to [0,\infty)$ with $\omega_g^\eps (0) = 0$ \EEE such that for every $n \in \N$,  for every $\sigma \in (0,    1  )$ \BBB with $\omega_g^\eps (\sigma) \le \frac{1}{2}$, \EEE and $u \in \PC(A;[0,\xi])$, $v \in \PC(B;\lbrace 0, \xi\rbrace)$    there exist $\varphi \in C_c^\infty(A; [0,1])$ with $\varphi=1$ in a neighborhood of $\ol{ A'}$  and $w \in \PC(\BBB A' \EEE \cup B; [0,\xi])$  such that \EEE
\begin{itemize}
\item[(i)] $w=v$ in $B \sm A$,
\item[(ii)] $\hn(J_{w} \sm (J_u \cup J_v \cup S))=0$,
\item[(iii)]  $\|\varphi u+(1-\varphi)v-w\|_{L^\infty(\BBB A' \EEE \cup B)}  \le \EEE \sigma$, and
\item[(iv)] $\BBB \H^{\eps,-}_n \EEE(w, A'\cup B) \leq (1 + \BBB \omega_g^\eps(\sigma) \EEE )(1+\eta) \big( \BBB \H^{\eps,-}_n \EEE(u,A) + \BBB \H^{\eps,-}_n \EEE(v,B)\big) + \frac{\Lambda}{\sigma}\|u-v\|_{L^1(S)}$.
\end{itemize}
\end{lemma}

  In (iii), $u$ and $v$ are considered to be extended arbitrarily outside of $A$ and $B$, respectively. \EEE   We now come to the proof of Theorem \ref{thm:convpreliminaries}. 
\EEE

\begin{proof}[Proof of Theorem \ref{thm:convpreliminaries}]

By \cite[Theorem 16.9]{DMLibro}, there exist $\H^- \colon L^1(\Omega') \times \A(\Omega') \to [0,+\infty]$ and $\BBB \E^- \EEE \colon L^1(\Omega') \times \A(\Omega') \to [0,+\infty]$ such that, up to a (not relabeled) subsequence, $\H_n^-$ and $\BBB \E_n^- \EEE $ $\ol \Gamma$-converge   to $\H^-$ and $\BBB \E^- \EEE $, \BBB respectively. \EEE We are left to prove the 
integral representations of $\H^-$ and $\mathcal{E}^-$ and the characterization of $h^-$ in the statement.
Given $I \subset (0,1)$ countable with $0 \in \overline{I}$, for \EEE
any $\varepsilon \BBB \in I \EEE $ we define $\H_n^{\varepsilon,-}$ 
 as in \eqref{HHNN} and  \EEE    $\E_n^{\varepsilon,-} \colon L^1(\Omega') \times \A(\Omega') \to [0,+\infty]$ by 
\begin{equation*}
 \E_n^{\varepsilon,-}(u, A):=\BBB \E_n^- \EEE (u, A)+ \varepsilon\, \hn(A \cap \BBB  J_u \cap \EEE K_n) \quad \BBB \text{if } u|_A \in \SBV^p(A)
\end{equation*}
and $+ \infty$ otherwise.  Observe that the surface densities of the functionals $\H_n^{\varepsilon,-}$ and $ \E_n^{\varepsilon,-}$ are bounded uniformly from below by $\eps$. This allows us to apply  \cite[Theorem~3.9]{Sto lavoro GSBV}    and \cite[Theorem~2.3]{FM} (by taking $L=\{0\}$ therein) 
 to infer that \EEE there   exist $\H^{\varepsilon,-}$,  $\E^{\varepsilon,-}\EEE \colon L^1(\Omega') \times \A(\Omega') \to [0,+\infty]$ 
such that $\H_n^{\varepsilon,-} $, $\E_n^{\varepsilon,-} $ $\ol \Gamma$-converge  (up to a subsequence) \EEE to $\H^{\varepsilon,-} $, $\E^{\varepsilon,-} $, with
\begin{equation}\label{1004261943}
\H^{\varepsilon,-}(u,A)= \int_{A\cap J_u} h^{\varepsilon,-}(\cdot, [u],\nu_u) \,\dhn \quad\text{for }u \in \PC(\Omega'), 
\end{equation}
\begin{equation}\label{16040759XXX}
\E^{\varepsilon,-}(u,A)= \int_A |\nabla u|^p\dx+\int_{A\cap J_u} h^{\varepsilon,-}(\cdot, [u],\nu_u) \,\dhn\quad\text{for } u \in \SBV^p(\Omega'),
\end{equation}
and $h^{\varepsilon,-}(x, \xi,\nu)=\limsup_{\varrho \to 0^+} \varrho^{-(d{-}1)}\mathbf{m}_{\H^{\varepsilon,-}}^{\PC}(\ol u_{x,\xi,\nu}, Q^\nu_\varrho(x))$  \BBB for any $x \in  \Omega'  $, $\xi \in  \R $, and  $\nu \in \Sn$. Note that by a diagonal argument the subsequence can be chosen independently of $\eps \in I$. More precisely, to guarantee that the $\Gamma$-limit $\H^{\varepsilon,-}$ satisfies the integral representation  \eqref{1004261943}, one needs to check the property 
  \begin{equation}\label{12040901-eps}
\limsup_{n\to \infty} \mathbf{m}_{\H^{\eps,-}_n}^{\PC}(\ol u_{x,\xi,\nu}, Q^\nu_\varrho(x)) \leq \mathbf{m}_{\H^{\eps,-}}^{\PC}(\ol u_{x,\xi,\nu}, Q^\nu_\varrho(x)) \leq \sup_{0<\varrho'< \varrho} \liminf_{n\to \infty}\mathbf{m}_{\H^{\eps,-}_n}^{\PC}(\ol u_{x,\xi,\nu}, Q^\nu_{\varrho'}(x)),
\end{equation}
  corresponding to  \cite[condition (2.8)]{FM} (with cubes in place of balls). \ZZZZ Moreover, the fact that the surface densities in \eqref{1004261943} and \eqref{16040759XXX} coincide follows from \eqref{12040901-eps} and \cite[(3.10)--(3.11)]{Sto lavoro GSBV}.   We defer the verification of  \eqref{12040901-eps} \EEE  to the end of the proof.  \EEE We note that alternatively, instead of employing the general results in  \cite[Theorem~3.9]{Sto lavoro GSBV}    and \cite[Theorem~2.3]{FM}, we could  resort \EEE to Proposition \ref{prop:Gammaconvgnminus} and  Theorem \ref{generalgammacon} below. \ZZZZ (Choose $f_n = \ZZZZ |\cdot|^p \EEE $, $g_n = g$ on $\overline{\Omega}$ and $g_n = C_g+1$ on $\Omega' \setminus \overline{\Omega}$, as well as $\varphi_n \colon K_n \to [0,\infty)$ such that  \EEE  $\psi_n = g(\cdot,\varphi_n,\nu_{K_n})$.)

  Now,  the representation of $\H^-$ in \eqref{Hnu} with the density $h^-$ given in \eqref{1004261901} holds by Lemma \ref{extra-lemma}.   Moreover, using 
 $\hn(K_n)\leq C_0$, we get $\E^-(u,A) \leq \E^{\varepsilon,-}(u,A) \leq \E^-(u,A) + C_0\varepsilon$ by general properties of $\Gamma$-convergence, i.e.,    $\E^{\varepsilon,-}$ converges   pointwise to   $\E^-$ as given \ZZZZ  in \EEE \eqref{Enu}, where we \ZZZZ  use 
" \EEE  that   
$h^-(x,\xi,\nu) =   \lim_{\varepsilon\to 0} h^{\varepsilon,-}(x,\xi,\nu)$
for all $x \in \Omega'$, $\xi \in \R$, and $\nu \in \mathbb{S}^{d{-}1}$. \BBB In a similar fashion,  \eqref{12040854} follows by passing to the limit $\eps \to 0$ ($\eps \in I$) in \eqref{12040901-eps}, \ZZZZ and letting $\varrho \to 0$. \EEE This concludes the proof.

It remains to prove \eqref{12040901-eps}. For notational convenience, we drop the index $\eps$ and write $\H^-$ and $\H^-_n$. As a preparation, we fix \ZZZZ $\varrho>0$ \EEE and define the neighborhood $N_{\varrho'}:=  Q^\nu_\varrho(x) \sm Q^\nu_{\varrho'}(x)$ for $\varrho' \in (0,\varrho)$. We observe that by \eqref{gproperty2} and \eqref{HHNN} we have
\begin{align}\label{ttouselater}
\H^-_n(\ol u_{x,\xi,\nu}, N_{\varrho'} ) \leq ( C_g + \eps)(\varrho^{d{-}1}-(\varrho')^{d{-}1}).
\end{align}   
\emph{First inequality in \eqref{12040901-eps}:}  \BBB For any $\delta>0$, choose \EEE $v \in \PC(Q^\nu_\varrho(x))$ with $\H^-(v, Q^\nu_\varrho(x))\leq \mathbf{m}_{\H^-}^{\PC}(\ol u_{x,\xi,\nu}, Q^\nu_\varrho(x)) + \delta$ and  $v=\ol u_{x,\xi,\nu}$ on \BBB $N_{\varrho'}$ for some \EEE $\varrho' \in (0,\varrho)$. \BBB By \EEE $\Gamma$-convergence of $\H_n^-(\cdot, Q^\nu_\varrho(x))$ to $\H^-(\cdot, Q^\nu_\varrho(x))$, \BBB  we find \EEE $(v_n)_n \subset \PC(Q^\nu_\varrho(x))$  such that  $v_n \to v $ in $L^1(Q^\nu_\varrho(x))$ and $\lim_{n\to \infty} \H_n^-(v_n, Q^\nu_\varrho(x))= \H^-(v, Q^\nu_\varrho(x))$. \ZZZZ By the monotonicity of $g$ the energies $\H_n^-$ and $\H^-$ are decreasing by truncation. Therefore, up to a truncation   it is not restrictive to assume that $v$ and  $v_n$  have  values in  $[0,\xi]$.   Fixing  
\BBB $\varrho''$ with $\varrho'<\varrho'' < \varrho$, \EEE
for any $n$ we apply Lemma~\ref{le:fundestHnm}  to $v_n$   (in place of $u$) \EEE and $\ol u_{x,\xi,\nu}$  (in place of $v$), with $A'=Q^\nu_{\varrho'}(x)$, $A=Q^\nu_{\varrho''}(x)$, $B=\BBB N_{\varrho'}\EEE$, $\sigma= \sqrt{\tau_n}$ for $\tau_n:= \|v_n- \ol u_{x,\xi,\nu}\|_{L^1(\ZZZZ N_{\varrho'} \EEE )}$, and   $\eta = \eta_n$ \EEE with $\eta_n\to 0$ slowly enough such that  the corresponding $\Lambda_n$ satisfies \EEE   $\Lambda_n\,\sqrt{\tau_n}\to 0$ as $n\to \infty$:  this provides a modification $w_n \in \PC(Q^\nu_\varrho(x))$ such that $w_n=\ol u_{x,\xi,\nu}$ on $Q^\nu_\varrho(x) \sm A$ 
(neighborhood of $\partial Q^\nu_\varrho(x)$), $w_n \to v$ in $L^1(Q^\nu_\varrho(x))$, and $\BBB \limsup_{n\to \infty} \EEE \H_n^-(w_n, Q^\nu_\varrho(x)) \BBB \le  \H^-(v, Q^\nu_\varrho(x))+ \BBB \limsup_{n \to \infty} \H^-_n(\ol u_{x,\xi,\nu}, N_{\varrho'} )  $. \BBB Using \eqref{ttouselater}, this \EEE implies the inequality by the arbitrariness of $\delta>0$ \BBB and $\varrho' <\varrho$. \EEE

\emph{Second inequality in \eqref{12040901-eps}:} For any 
$\varrho' \in (0,\varrho)$, \ZZZZ we choose a subsequence (not relabeled) that realizes the $\liminf$ of $\mathbf{m}_{\H^-_n}^{\PC}(\ol u_{x,\xi,\nu}, Q^\nu_{\varrho'}(x))$. We \EEE  consider $u_n \in \PC(Q^\nu_{\varrho'}(x))$ such that $u_n= \ol u_{x,\xi,\nu}$ near $\partial Q^\nu_{\varrho'}(x)$ and $\H^-_n(u_n, Q^\nu_{\varrho'}(x)) \leq \mathbf{m}_{\H^-_n}^{\PC}(\ol u_{x,\xi,\nu}, Q^\nu_{\varrho'}(x))+1/n$. By truncation we may assume that any $u_n$ maps into  $[0,\xi]$. Defining   the extension 
$v_n :=u_n \chi_{Q^\nu_{\varrho'}(x)} + \ol u_{x,\xi,\nu}\chi_{N_{\varrho'}}$, 
it holds that $\H^-_n(v_n, Q^\nu_{\varrho}(x))= \H^-_n(u_n, Q^\nu_{\varrho'}(x))+ \H^-_n(\ol u_{x,\xi,\nu}, \ZZZZ {N_{\varrho'} \EEE })$. \BBB  Since \EEE 
$v_n$ maps into $[0,\xi]$ and  $\hn(J_{v_n})$ is  \BBB uniformly \EEE bounded (by $g\geq c_g$, see \eqref{gproperty}, and $\hn(K_n)\leq C_0$), we \BBB can apply a compactness result for piecewise constant functions (see \cite[Theorem~4.25]{Ambrosio-Fusco-Pallara:2000})  \EEE to find $v \in \PC(Q^\nu_{\varrho}(x))$ with $v= \ol u_{x,\xi,\nu}$ \BBB in $N_{\varrho'}$ \EEE   such that   $v_n\to v$ in $L^1(Q^\nu_{\varrho}(x))$, up to a (not relabeled) subsequence.  \BBB In view of  \eqref{ttouselater}, we find  
\[
 \H^-(v,Q^\nu_{\varrho}(x))\leq \liminf_{n\to \infty}  \H^-_n(v_n, Q^\nu_{\varrho}(x))\leq  \liminf_{n\to \infty}  \H^-_n(u_n, Q^\nu_{\varrho'}(x)) + \BBB (C_g + \eps) \EEE (\varrho^{d{-}1}-(\varrho')^{d{-}1}).
\]
As $\mathbf{m}_{\H^-}^{\PC}(\ol u_{x,\xi,\nu}, Q^\nu_\varrho(x))\leq \H^-(v,Q^\nu_{\varrho}(x))$ and $\H^-_n(u_n, Q^\nu_{\varrho'}(x)) \leq \mathbf{m}_{\H^-_n}^{\PC}(\ol u_{x,\xi,\nu}, Q^\nu_{\varrho'}(x))+1/n$, this concludes  \BBB the proof of the second inequality. \EEE
\end{proof}
\begin{remark}\label{rem:72GiaPon}
Arguing as in \cite[Remark~8.2]{GiacPonsi}, it can be shown that the values of $h^-(\cdot, \xi, \nu)$ on $\partial_D \Omega$ depend only on $g$ and $(K_n, \psi_n)$, and  do not depend neither on \EEE the choice of $\Omega'$ with the above properties, nor on  the  exact value $\widetilde{C_g}>C_g$  \EEE such that $\H_n^-(u,A)=\widetilde{C_g} \hn( A \cap J_u \EEE)$ for $A \in \mathcal{A}(\Omega')$ with $A \subset \Omega'\sm \ol \Omega$. 
\end{remark}

\begin{remark}\label{rem:tothm31}
\BBB Arguing \EEE as in \cite[Proposition~5.3]{GiacPonsi}, by Theorem~\ref{thm:convpreliminaries} we deduce the separate lower semicontinuity
$$ \int_A |\nabla u|^p \, {\rm d}x  \le \liminf_{n \to \infty}   \int_A |\nabla u_n|^p \, {\rm d}x,    \quad \int_{ A \cap J_u \EEE} h^-(\cdot,[u],\nu_u) \, {\rm d}\Hd \le \liminf_{n\to \infty} \int_{ A \cap J_{u_n} \EEE} \big( g(|\BBB [u_n] \EEE |)-\psi_n\big)^+ \, {\rm d}\Hd $$
for any $A\in \A(\Omega')$ and \BBB any sequence $(u_n)_n$ with \EEE  $u_n  \in \EEE \SBV^p_{w_n}(\Omega')$ and  \EEE  $u_n \rightharpoonup u$ 
in $\SBV^p(\Omega').$  
\end{remark}

\BBB We close this section with a consequence of the fundamental estimate. 

\begin{lemma}[Boundary conditions in recovery sequences]\label{recovboundra}
Suppose that $\H_n^-$  $\ol \Gamma$-converge   to $\H^-$. Then, there exists a subsequence (not relabeled) such that, given $v \in \PC(Q^\nu_\varrho(x))$ with   $v=\ol u_{x,\xi,\nu}$ in a neighborhood of $\partial Q^\nu_\varrho(x)$, we can choose  
$(v_n)_n \subset \PC(Q^\nu_\varrho(x))$  satisfying  $v_n \to v $ in $L^1(Q^\nu_\varrho(x))$, $v_n=\ol u_{x,\xi,\nu}$ in a neighborhood of $\partial Q^\nu_\varrho(x)$, and $\lim_{n\to \infty} \H_n^-(v_n, Q^\nu_\varrho(x))= \H^-(v, Q^\nu_\varrho(x))$. 
\end{lemma}
\begin{proof}
As above, by a diagonal argument, we can choose a subsequence such that $\H_n^{\eps,-}$  $\ol \Gamma$-converge   to $\H^{\eps, -}$ for all $\eps \in I$. The statement for $\H_n^{\eps,-}$ and $\H^{\eps,-}$  in place of  $\H_n^{-}$ and $\H^{-}$  follows as in the proof of the first inequality in \eqref{12040901-eps}, by using a diagonal argument. Then the statement for $\eps = 0$ follows by a further diagonal argument as $\eps \to 0$, taking \eqref{HHNN} into account. \EEE
\end{proof}

\subsection{Slicing}\label{sec:slicing}
Let \EEE us introduce some slicing notation. For $y \in \R^d$ and $\nu \in \mathbb{S}^{d{-}1}$, given a set $G \subset \R^d$ and $v \colon G \to \R$, we define $G^\nu_y := \lbrace t \in \R \colon  y + t\nu \in G\rbrace$, and for $t \in G_y^\nu$ we set $v^\nu_y(t) :=  v(y + t\nu)$. We also define  the \ZZZZ orthogonal \EEE projection  $ P^{\nu}(G) := \lbrace y \in \R^d  \colon y \cdot \nu = 0, \,  G^\nu_y \neq \emptyset \rbrace$.   We formulate the following   technical \EEE  lemma \BBB that we will use several times. \EEE

\begin{lemma}\label{le:projection}
Let $Q \subset \R^d$ be a cube with two sides orthogonal to $\nu \in \mathbb{S}^{d{-}1}$. Let $\BBB \Psi \EEE \subset Q$ be a Borel set such that  $\hn( \BBB \Psi  \EEE )\leq \ZZZZ \mathcal{H}^{d-1}( P^{\nu}(\Psi)) \EEE +  \BBB \delta \EEE $ \BBB for some $\delta>0$ \EEE and let $A\subset Q$ be a Borel  set satisfying $A=(P^\nu(A) + \nu \R)\cap Q$. Then,
\[
\hn(\Psi \cap A) \leq \hn(P^\nu(\Psi \cap A)) + \delta.
\]
\end{lemma}
\begin{proof}
Since $\hn(\Psi   \sm A) - \hn(P^\nu(\Psi  \sm A  ))\geq 0$ \BBB by the Area Formula (see e.g.\ \cite[(2.72)]{Ambrosio-Fusco-Pallara:2000}), \EEE  $\hn(\Psi)= \hn(\Psi \cap A)+\hn(\Psi  \sm A \EEE )$, and $\ZZZZ \mathcal{H}^{d-1}(P^{\nu}(\Psi)) \EEE = \hn(P^\nu(\Psi \cap A))+\hn(P^\nu( \Psi \sm A))$, it holds that
\[
\hn(\Psi\cap A) - \hn(P^\nu(\Psi \cap A))\leq \hn(\Psi )-  \ZZZZ \mathcal{H}^{d-1} (P^{\nu}(\Psi )) \EEE \leq \delta.
\]
This concludes the proof.
\end{proof}  
 
\section{A novel convergence of sets and functions}\label{subsec:sigmaconv}

This section is devoted to  a \EEE  new notion of convergence, inspired by $\sigma$-convergence  \cite{GiacPonsi},  which  involves \EEE   general couples $( {K}_n, {\psi}_n)$ with rectifiable sets  $ {K}_n$  and Borel functions  ${\psi}_n \colon K_n \to [0,\infty)$.  \BBB We always extend $\psi_n \equiv 0$ outside $K_n$. \EEE For convenience, in the definition we directly use the (larger)   Lipschitz domain $\Omega'$ as this allows for a convenient adaptation of the notion to  a setting with boundary conditions. \ZZZZ Recall the shorthand notation in \eqref{convention}.  \EEE

\begin{definition}\label{def:sigmaconv}
 Let  $\widehat g \colon \Omega' \times [0,\infty) \to [0,\infty)$ be a  Borel function. \EEE
A sequence of  sets and functions $(K_n,\psi_n)_n$, where $K_n$ are rectifiable sets in  $\Omega'$  and   $\psi_n \colon K_n \to [0,\infty)$ are Borel functions, \BBB is said to \EEE \emph{$\sigma_{\rm cf}$-converge in $\Omega'$ to $(K,\psi)$}  if \EEE the functionals $\mathcal{H}_n^-\colon \BBB L^1(\Omega') \EEE {\times} \mathcal{A}(\Omega') \to [0,+\infty)$, defined by
 \begin{equation}\label{1708240952}
\mathcal{H}_n^-(u, A):= \int_{  A \cap J_u \EEE}  \big(    \widehat{g} \EEE (\cdot, |[u]|) - \psi_n \big)^+  \, \dhn \quad \BBB \text{if } u|_A \in \PC(A) \EEE
\end{equation}  
\BBB and $+\infty$ otherwise, 
 $\ol \Gamma$-converge  to  $\mathcal{H}^-\colon \BBB L^1(\Omega') \EEE {\times} \mathcal{A}(\Omega') \to [0,+\infty)$    with \EEE
\begin{equation}\label{1708240957}
\mathcal{H}^-(u, A) = \EEE \int_{ A \cap J_u \EEE} h^-(\cdot, [u], \nu_u) \EEE \,  \mathrm{d}\mathcal{H}^{d{-}1} \quad  \BBB \text{for }  u \in \PC(\Omega'), \EEE
\end{equation}
and the following two properties hold:
\begin{itemize}
\item $K$ is the maximal (with respect to  $\subset$) rectifiable set in $\Omega'$ such that
\begin{equation}\label{maxK}
\lim_{\xi \to 0}  h^-(x, \xi, \nu_K(x))=0 \quad\text{for } \mathcal{H}^{d{-}1}\text{-a.e.\ }x \in K,
\end{equation}
i.e., for every rectifiable $H \subset \Omega'$ we have
\begin{align}\label{maxK2}
 \lim_{\xi \to 0}  h^-(x, \xi, \nu_H(x))=0 \quad\text{for } \mathcal{H}^{d{-}1}\text{-a.e.\ }x \in H \ \Rightarrow \ H \subset K.   
 \end{align}
\item $\psi \colon K \to [0,\infty)$ is the maximal (with respect to  $\le$) Borel function  such that  
\begin{equation}\label{maxpsi}
h^-(x,\xi, \nu_K(x))=0 \quad\text{for all $\xi \in \R$ with   $g(|\xi|) \EEE \le  \psi(x) $ for } \mathcal{H}^{d{-}1}\text{-a.e.\ }x \in K,
\end{equation}
i.e., for every Borel function $\zeta \colon K \to [0,\infty)$ we have    \ZZZZ 
\begin{align}\label{maxpsi2}
h^-(x,\xi, \nu_K(x))=0 \quad\text{for all $\xi \in \R$ with   $g(|\xi|)   \le  \zeta(x) $ for } \mathcal{H}^{d{-}1}\text{-a.e.\ }x \in K \ \Rightarrow \   \zeta  \le \psi. 
\end{align}
\EEE
\end{itemize} 
\end{definition}
 We  adapt the notion to a setting with boundary conditions. As in Section \ref{sec: evo}, we consider the sets $\Omega \subset \Omega'$ with $ \partial_D \Omega = \partial \Omega \cap \Omega'$, and let $g$ be a function as  in Section \ref{sec: evo}, in particular satisfying  \eqref{gproperty}--\eqref{gproperty2}.  \BBB Recall the definition of $\mathcal{C}$ in \eqref{mathcalcdef}. \EEE

\begin{definition}\label{def:sigmaconv2} 
A sequence of  sets and functions $(K_n,\psi_n)_n$,  \BBB with $(K_n,\psi_n) \in  \mathcal{C} \EEE $,  \EEE  \emph{$\sigma^\partial_{\rm cf}$-converges in $\Omega'$ to $(K,\psi)$}  if $(K_n,\psi_n)_n$  $\sigma_{\rm cf}$-converges in the sense of Definition \ref{def:sigmaconv} for the function $\widehat g$ given by
\begin{align}\label{widehtatg}
\widehat g(x,t) = \begin{cases}   g(t) & \text{for } x \in \Omega' \cap \overline{\Omega}, \,  t \ge 0, \\   C_g+1 & \text{for } x \in \Omega' \sm \overline{\Omega}, \,  t \ge 0. \end{cases} 
\end{align}
\end{definition}
In view of Theorem \ref{thm:convpreliminaries} \BBB and Corollary \ref{cor: bc}, \EEE and the definition of $\widehat g$, we see that this notion indeed relates to boundary conditions imposed on $\Omega' \setminus \overline{\Omega}$, even if the notion does not directly involve any boundary conditions. 

 From now on, we always adopt the setting of Definition \ref{def:sigmaconv2}, and \EEE study the main properties for  $g$ from Section~\ref{sec: evo} and for  sequences $(K_n,\psi_n)_n \BBB \subset \mathcal{C} \EEE $ satisfying  $\hn(K_n),\, \int_{K_n} \psi_n \,\dhn \leq C_0$. (We will not repeat these general assumptions all the time.)  \EEE  
\begin{remark}\label{rem: sigma1} We  proceed with \EEE some direct consequences and remarks. 
\begin{itemize}
\item[(i)] By \BBB definition of  $\sigma^\partial_{\rm cf}$-convergence, the sequence  \EEE $(\H^-_n)_n$ corresponding to $(K_n,\psi_n)_n$ $\ol \Gamma$-converges, and \BBB by Theorem~\ref{thm:convpreliminaries} \EEE  the density $h^-$ of the limit in  \eqref{1708240957} is characterized by \eqref{1004261901}--\eqref{12040854}.   In particular,   we get  $h^-(x, \xi, \nu) \le   g(|\xi|) $ \BBB for all $x \in  \Omega  $, \ZZZZ $\xi \in \R$, \EEE   and $\nu \in \mathbb{S}^{d{-}1}$  by taking  $\ol u_{x,\xi,\nu}$ as competitor. For all differentiability points of $x \in \partial_D \Omega$ (and thus for $\mathcal{H}^{d{-}1}$-a.e.\ $x \in \partial_D \Omega$) we also have  $h^-(x, \xi, \nu) \le   g(|\xi|) $  by using \EEE   $\xi\chi_{\Omega}$ as competitor. \EEE Moreover, if $(K_n,\psi_n)$  $\sigma^\partial_{\rm cf}$-converges in $\Omega'$ to $(K,\psi)$, then $K \subset \Omega \cup \partial_D \Omega$  and $(K,\psi) \in \mathcal{C}$. \EEE

\item[ (ii)] By Remark~\ref{rem:propglimit} and continuity of $g$ \ZZZZ on $(0,\infty)$, \EEE it follows that the function $\xi \mapsto h^-(x,\xi,\nu)$ is continuous for all $x \in    \Omega  \cup \partial_D \Omega \EEE$  and $\nu \in \mathbb{S}^{d{-}1}$ (with same modulus of continuity of $g$), that $h^-(x,\xi,\nu) = h^-(x,-\xi,\nu) $, and that $\xi \mapsto h^-(x,\xi,\nu)$ is nondecreasing on $[0,\infty)$. 

\item[(iii)]  If  \EEE $(K_n,\psi_n)_n$, $(\tilde{K}_n,\tilde{\psi}_n)_n$ are such that $\tilde{K}_n  \subset   K_n$, $\tilde{\psi}_n \le  \psi_n$ and $\sigma^\partial_{\rm cf}$-converge in $\Omega'$ to $(K,\psi)$ and $(\tilde{K},\tilde{\psi})$, respectively, then $\tilde{K}  \subset   K$ and $\tilde{\psi}  \le   \psi$.   Indeed, denoting by \EEE $h^-$, $\tilde{h}^-$ the corresponding limit densities in Definition \ref{def:sigmaconv}, this follows since $h^-\leq \tilde{h}^-$ by monotonicity of the densities in the $\Gamma$-convergence result Theorem~\ref{thm:convpreliminaries}.

\item[(iv)]
By the results in   Appendix \ref{sec: ppi3},  we could \EEE 
 consider more general \BBB situations \EEE of the form 
\begin{equation*}\label{1708240952'}
\mathcal{H}_n^-(u, A):= \int_{  A \cap J_u \EEE} \big(g_n( \cdot, \EEE [u], \nu_u) - \psi_n \big)^+   \, {\rm d}  \mathcal{H}^{d{-}1}, 
\end{equation*} 
with $(g_n)_n$ satisfying (g1)--(g4),   and possibly take $g_n(x, \BBB \varphi_n, \EEE \nu_u)$ in place of $\psi_n$,  for some sequence of Borel functions   $(\varphi_n)_n$.

\RRR 
\end{itemize}
\end{remark}

 For convenience and with a slight abuse of notation, we will from now \BBB on \EEE always speak of $\sigma_{\rm cf}$-convergence although $\sigma^\partial_{\rm cf}$-convergence is intended. \EEE    We proceed with fundamental properties of this type of convergence. 

\begin{theorem}[Compactness and lower semicontinuity]\label{sigma-comp} 
  \ZZZZ Let    $(K_n,\psi_n)_n \subset \mathcal{C} $ with   $\hn(K_n) \le C_0$ and $\int_{K_n} \psi_n \,\dhn \leq C_0$. \EEE Then, 
there exists a subsequence (not relabeled), a rectifiable set $K$ and a Borel function $\psi\colon K \to   [c_g, \EEE \infty)$ such that $(K_n,\psi_n)_n$ $\sigma_{\rm cf}$-converges in $\Omega'$ to $(K,\psi)$.  Moreover, for all $A \in \mathcal{A}(\Omega')$ we have 
\begin{equation}\label{lsc-neu}
\liminf_{n\to \infty} \hn(K_n \cap A) \geq \hn(K\cap A),
\end{equation}  
\begin{align}\label{lsc}
 \liminf_{n\to \infty} \int_{K_n \cap A} \psi_n \, {\rm d}\mathcal{H}^{d{-}1} \ge \int_{K\cap A} \psi \, {\rm d}\mathcal{H}^{d{-}1} ,
\end{align}
\begin{align}\label{lscminus}
 \liminf_{n\to \infty} \int_{K_n \cap A}  (\psi_n-c_g) \EEE \, {\rm d}\mathcal{H}^{d{-}1} \ge \int_{K\cap A} (\psi -c_g) \, {\rm d}\mathcal{H}^{d{-}1}. 
\end{align}
\end{theorem}

  \begin{proof}
By Theorem~\ref{thm:convpreliminaries}, up to a subsequence (not relabeled) we get \BBB that \EEE  $\mathcal{H}_n^-$ $\ol \Gamma$-converge to $\mathcal{H}^-$, represented as in \eqref{1708240957},  where $h^-$ is given as in \EEE \eqref{1004261901}--\eqref{12040854}.

\noindent \emph{Step 1: The limit set $K$.} 
We define the class
$${\mathcal{K} := \Big\{ H \subset   \Omega \cup \partial_D \Omega \EEE  \colon \text{$H$ is rectifiable and $\lim_{\xi \to 0}  h^-(x, \xi, \nu_H(x))=0$ for $\mathcal{H}^{d{-}1}\text{-a.e.\ }x \in H$}} \Big\}, $$
and the measures $\widetilde{\mu}_n:= \hn\mres_{K_n}$. Since $\hn(K_n)\leq C_0$, there exists $\widetilde{\mu} \in \M_b^+(\Omega')$ such that, up to a subsequence, $\BBB \widetilde{\mu}_n \EEE \weaklystar \widetilde{\mu}$ in $\M_b^+(\Omega')$  as $\BBB n \EEE  \to \infty$. \EEE
Given $H \in \mathcal{K}$,  we claim that for $\hn$-a.e.\ $x \in H$ \BBB the following holds: \EEE 
for any $\varepsilon>0$ there exists $\ol \varrho>0$ such that for any $\varrho \in (0 ,\ol \varrho)$ with \ZZZZ $B_\varrho(x) \subset \Omega'$ and \EEE $\widetilde{\mu}(\partial B_\varrho(x))=0$   one has \EEE
\begin{equation}\label{14041056}
\Hd(H \cap B_\varrho(x)) - \varepsilon \varrho^{d{-}1} \leq \widetilde{\mu}(B_\varrho(x)) .
\end{equation}
  Once  this \EEE is shown,  by the rectifiability of $H$ and the Besicovitch Derivation Theorem (see \cite[Theorem~2.22]{Ambrosio-Fusco-Pallara:2000}) applied to $\widetilde{\mu}$ and $\hn\mres_H$ it follows that 
\begin{equation}\label{14041106}
\hn\mres_H \leq \widetilde{\mu} \quad \text{ in } \ZZZZ\M^+_b \EEE (\Omega').
\end{equation}
In particular, by \eqref{14041106}, we get $\mathcal{H}^{d{-}1}(H) \le \liminf_{n \to \infty} \mathcal{H}^{d{-}1}(K_n) \le C_0$ for each $H \in \mathcal{K} $. 
Therefore, 
$$M_{\mathcal{K}}:=\sup \lbrace \mathcal{H}^{d{-}1}(H) \colon \, H \in \mathcal{K}    \rbrace \leq C_0,  $$
and we can consider a   maximizing sequence $(H_k)_k$ of this maximization problem. We define $K:= \bigcup_{k \in \N} H_k \in \mathcal{K}$ and observe that $M_{\mathcal{K}} = \mathcal{H}^{d{-}1}(K)$. By definition of $\mathcal{K}$   we observe that $K$ satisfies \eqref{maxK}. Since  $M_{\mathcal{K}} = \mathcal{H}^{d{-}1}(K)$, we also see that $K$ is maximal (with respect to $\subset$) with this property, i.e., \eqref{maxK2} holds. 
Moreover, as $K \in \mathcal{K}$, by \eqref{14041106}  also \eqref{lsc-neu} follows.

 We close this step of the proof by proving \EEE  \eqref{14041056}.
By rectifiability of $H$, for $\hn$-a.e.\ $x \in H$ there exists $S \subset  \Omega \cup \partial_D \Omega$  being \EEE a $(d{-}1)$-dimensional graph of class $C^1$ such that $x \in S$, with normal $\BBB \nu_H(x) \EEE$ at $x$, $\lim_{\varrho\to 0^+}   \varrho^{-{d{-}1}} \hn({ B_\varrho(x)} \cap (H\triangle S))=0$  
and for $\varrho$ small $S$ separates $B_\varrho(x)$ in two connected components $B^{\pm,S}$.
Let us fix $\varepsilon>0$ and let $\ol \varrho>0$  be \EEE such that, for any $\varrho<\ol \varrho$,  
\begin{equation}\label{14040946}
\hn\big(  {B_{\varrho}(x)} \EEE \cap (H\triangle S)\big)  \le \frac{\varepsilon}{2}  \EEE \,\varrho^{d{-}1}.
\end{equation}
For $\delta>0$, we set $v:=\delta \chi_{B^{+,S}}$.
We notice that, since
$h^-(x,\xi,\nu) \le g(|\xi|)$   for all \BBB $\mathcal{H}^{d{-}1}$-a.e.\ \EEE  $x \in \Omega \cup \partial_D \Omega $ \EEE \ZZZZ and all $\xi \in \R$, $\nu \in \mathbb{S}^{d{-}1}$ \EEE (see Remark~\ref{rem: sigma1}(i)), 
\begin{equation*}\label{14040942}
\begin{split}
\H^-(v,  {B_{\varrho}(x)} \EEE)&= \int_{J_v} h^-(\cdot, [v],\nu_v) \,\dhn  \le \EEE g(\delta) \, \hn\big((S \sm H) \cap  {B_{\varrho}(x)} \EEE\big) + \int_{H \cap S \cap  B_\varrho(x)} h^-(\cdot, [v], \ZZZZ \nu_H) \EEE \,\dhn
\\&
 \le \EEE  g(\delta)\,  \frac{\varepsilon}{2} \EEE \,\varrho^{d{-}1} + \int_{H \cap B_\varrho(x)} h^-(\cdot, \delta,\nu_H) \,\dhn=: E^\delta_{\varrho,\varepsilon},
\end{split}
\end{equation*}
where we also used that $|[v]| = \delta$ on $J_v$.
Let $(v_n)_n \subset \PC(B_\varrho(x))$ be a recovery sequence for $v$ with respect to the $\Gamma$-convergence (in the strong $L^1$-topology) of $\H^-_n(\cdot, B_\varrho(x))$ to $\H^-(\cdot, B_\varrho(x))$. By a truncation argument, it is not restrictive to assume that $v_n$ \BBB has \EEE values in  $[0,\delta]$. In view of \eqref{1708240952}, we particularly get  \BBB (recalling the convention that \EEE $\psi_n \equiv 0$ outside of $K_n$) \EEE
\begin{equation*}
 E^\delta_{\varrho,\varepsilon}   \ge \EEE \H^-(v,  {B_{\varrho}(x)} \EEE) = \lim_{n\to \infty}\H^-_n(v_n,   {B_{\varrho}(x)} \EEE) \geq \limsup_{n\to \infty} \int_{J_{v_n}\sm K_n} g ( |[v_n]|) \,\dhn,
\end{equation*}
 where we used that  $\widehat{g}$ as given in \eqref{widehtatg} is bigger or equal than $g$. \EEE 
Then, for $\varrho \BBB < \bar{\varrho}\EEE$ such that  \ZZZZ  $B_\varrho(x) \subset \Omega'$ and  \EEE $\widetilde{\mu}(\partial B_\varrho(x))=0$,   \BBB  by the fact that $|[v_n]| \le  \delta$ on $J_{v_n}$, the monotonicity of $g$, and \EEE $K_n \subset \Omega \cup \partial_D \Omega$, we find  \EEE
\begin{equation*}
 E^\delta_{\varrho,\varepsilon}+g(\delta) \widetilde{\mu}(B_\varrho(x)) =   E^\delta_{\varrho,\varepsilon}+g(\delta) \liminf_{n\to \infty} \mathcal{H}^{d{-}1}(K_n \cap B_\varrho(x))   \ge \liminf_{n \to \infty }\int_{ J_{v_n}}   {g}(  |[v_n]|)  \,  \dhn.
\end{equation*}
\BBB Then, \EEE  by lower semicontinuity of $v\in \PC(B_\varrho(x))\mapsto\int_{J_v}  {g}( |[v]|) \,\dhn$ in the strong $L^1$-topology (by   the  fact that \EEE $g$ is  nonnegative, continuous, increasing, and subadditive \ZZZZ on $(0,\infty)$, \EEE see \cite[Section 5.4, Example~5.23]{Ambrosio-Fusco-Pallara:2000})  and  by the fact that \EEE $|[v]| = \delta$ on $J_v$, we get 
\begin{equation*}
\begin{split}
 E^\delta_{\varrho,\varepsilon}+g(\delta) \widetilde{\mu}(B_\varrho(x))  \ge \int_{ J_{v}}   {g}( |[v]|)  \, \dhn =  g(\delta) \Hd(  S \EEE  \cap B_\varrho(x)).
\end{split}
\end{equation*}
We conclude \eqref{14041056} by the arbitrariness of $\delta >0$,   using that $\lim_{\delta\to 0}g(\delta)=c_g>0$ (see \eqref{gproperty}),  \eqref{14040946}, \EEE and the fact that, as $H \in \mathcal{K}$, $\lim_{\delta\to 0}\int_{H \cap B_\varrho(x)} h^-(x, \delta,\nu_H) \,\dhn=0$  by the  Monotone Convergence Theorem. (Recall from Remark~\ref{rem: sigma1}(ii) that $h^-(x,\cdot,\nu)$ is nondecreasing.)

\vspace{1em}

\noindent \emph{Step 2: The limit function $\psi$.}
 Since $\int_{K_n} \psi_n \,\dhn \leq C_0$, there exists $\mu\in  \M_b^+(\Omega')$ \EEE such that, up to a subsequence,
\begin{equation}\label{1910251228}
\mu_n:=\psi_n \hn\mres_{K_n} \weaklystar \mu \quad\text{in } \ZZZZ \M^+_b \EEE (\Omega'). \EEE
\end{equation}
We consider the class
$$ {\Psi} := \big\{ \varphi \colon K \to [0,\infty) \text{ Borel} \colon \, h^-(x, \varphi(x), \nu_K(x))=0 \quad\text{for } \mathcal{H}^{d{-}1}\text{-a.e.\ }x \in K \big\}. $$
We claim that for every $\varphi\in \Psi$ and  $x \in K$  \ZZZZ the following holds:  for any $\varepsilon>0$ there exists $\ol \varrho>0$ such that for any $\varrho \in (0 ,\ol \varrho)$ with \EEE $ Q^{\nu_K(x)}_\varrho(x)\EEE\subset \Omega'$ and $\mu (\partial Q^{\nu_K(x)}_\varrho(x))\EEE=0$, it holds that
\begin{equation}\label{1910251204}
g(\varphi(x))\leq \frac{\mu(Q^{\nu_K(x)}_{\varrho}(x))}{\varrho^{d{-}1}} \ZZZZ + \eps. \EEE
\end{equation}
 We defer the proof to the end of this step. \EEE By   \eqref{1910251204} and the Besicovitch Derivation Theorem (see \cite[Theorem~2.22]{Ambrosio-Fusco-Pallara:2000}) applied to $\mu$ and $g(\varphi) \hn\mres_K$ it follows that for every $\varphi\in \Psi$ 
\begin{equation}\label{0105260938}
g(\varphi) \hn \mres_K \leq \mu \quad \text{in } \ZZZZ \M^+_b \EEE (\Omega'),
\end{equation} 
which gives in particular \EEE
$\int_K g(\varphi)\, \dhn \leq \mu(\Omega')\leq\liminf_{n\to \infty} \mu_n(\Omega')  \leq C_0$ \BBB by \eqref{1910251228}. \EEE
Therefore,  $M_{\Psi}:=\sup_{\varphi \in \Psi} \int_K g(\varphi)\,\dhn \leq C_0 $ and \EEE we set $\psi:=\sup_{k\in \N}g(\varphi_k)$, for $(\varphi_k)_k$ an increasing maximizing sequence for the problem
 $\sup_{\varphi \in \Psi} \int_K g(\varphi) \,\dhn$. \EEE
By the definition of $\Psi$ and the continuity and monotonicity of $\xi \mapsto h^-(x,\xi,\nu)$ (see Remark \ref{rem: sigma1}(ii)) we get that $\psi$ satisfies \eqref{maxpsi}, and that $\psi\geq c_g$ by \eqref{gproperty}.    Moreover, as $M_{\Psi} = \int_K  \psi \, {\rm d} \mathcal{H}^{d{-}1}$  by the Monotone Convergence Theorem, the function $\psi$ is maximal (with respect to $  \le  $) with this property, i.e., \eqref{maxpsi2} holds.  By  \eqref{0105260938} \EEE we get \eqref{lsc}.

 Let us now come to the proof of  \eqref{1910251204}. \EEE 
\ZZZZ Fix \EEE $\varphi\in \Psi$, $x \in K$, \ZZZZ and  $\varepsilon>0$. 
 For brevity, let  $\nu:= \nu_K(x)$ and $\xi:= \varphi(x)$. By \eqref{1004261901} \ZZZZ we can choose  $\ol \varrho>0$   such that, for any $\varrho<\ol \varrho$, we have $ \ZZZZ Q^{\nu}_\varrho \EEE (x)\EEE\subset \Omega'$  and   there exists $v\in \PC(Q^\nu_\varrho(x))$ with   
$v=\ol u_{x,\xi,\nu} \text{ in a neighborhood of }\partial Q^\nu_\varrho(x)\text{ and }\mathcal{H}^-(v, Q^\nu_\varrho(x)) \le \varepsilon \ZZZZ \varrho^{d-1} \EEE$. \ZZZZ Let us fix $0 < \varrho <\ol \varrho$ such that also \EEE  $\mu (\partial  \ZZZZ Q^{\nu}_\varrho \EEE (x)\EEE)=0$ \ZZZZ holds. \EEE 

Let $(v_n)_n \subset \PC(Q^\nu_\varrho(x))$ be a recovery sequence for $v$ with respect to  $\Gamma$-convergence  of $\mathcal{H}_n^-(\cdot, Q^\nu_\varrho(x))$ to $\mathcal{H}^-(\cdot, Q^\nu_\varrho(x))$. 
Then $\limsup_{n \to \infty} \mathcal{H}_n^-(v_n, Q^\nu_\varrho(x))  \le \EEE \varepsilon \ZZZZ \varrho^{d{-}1} \EEE$.  \BBB By \EEE $\mu(\partial Q^\nu_\varrho(x))=0$ and  the lower semicontinuity of $v\in \PC( Q^\nu_\varrho(x)\EEE)\mapsto\int_{J_v} g(|[v]|) \,  \dhn$ in the strong $L^1$-topology, recalling \eqref{hnm} and \eqref{1910251228} it follows that  
\begin{equation}\label{1910251311}
\int_{J_v} g(|[v]|)\,\dhn \leq    \liminf_{n \to \infty}    \int_{J_{v_n}}  {g} \EEE (  |[v_n]|)\,\dhn \leq \liminf_{n\to \infty}\mu_n(Q^\nu_\varrho(x)) + \varepsilon  \ZZZZ \varrho^{d{-}1} \EEE   =   \mu(Q^\nu_\varrho(x))+\varepsilon \ZZZZ \varrho^{d{-}1}, \EEE
\end{equation}
 where we again used that  $g \le \widehat{g}$.
\BBB 
Here we used \EEE the  following elementary formula that we  will \EEE employ several times in the sequel:  
\begin{equation}\label{proprliminfsup}
 \liminf_{n\to \infty} a_n    - \liminf_{n\to \infty} b_n  \leq \limsup_{n\to \infty} \,  (a_n-b_n)^+ \text{ for sequences } (a_n)_n,\,(b_n)_n \subset [0,+\infty).
\end{equation}  
Since 
  $v=\ol u_{x,\xi,\nu}$ in a neighborhood of $\partial Q^\nu_\varrho(x)$, \BBB by a slicing argument  along the direction $\nu$ we get \EEE
\begin{equation}\label{1910251316}
g(\xi) \,\varrho^{d{-}1} \leq \int_{J_v} g(|[v]|)\,\dhn.
\end{equation}
In fact, \BBB recalling the notation in Subsection \ref{sec:slicing}, \EEE
 for $\hn$-a.e.\ $y \in  P^\nu \EEE (Q^\nu_\varrho(x))$ it holds 
\[
(g(\ZZZZ |v| \EEE ))^\nu_y \in \PC\big((Q^\nu_\varrho(x))^\nu_y\big) \text{ with } J_{(g(|v|))^\nu_y}=J_{v^\nu_y}=(J_v)^\nu_y=(J_{g(|v|)})^\nu_y
\]
and
\[
D^j (v^\nu_y) \big((Q^\nu_\varrho(x))^\nu_y\big)= \sum_{s \in (J_v)^\nu_y} [v](s)= \xi.
\] 
Therefore,  by the subadditivity of $g$, we get 
\[
g(\xi)= g\Big(\sum\nolimits_{s \in (J_v)^\nu_y} [v](s)\Big)\leq \sum_{s \in J_{(g( |v| \EEE ))^\nu_y}}g(|[v](s)|) = |D^j (g(|v|)^\nu_y)|\big((Q^\nu_\varrho(x))^\nu_y\big).
\]
Thus, we obtain \eqref{1910251316} by integrating over $y\in  P^\nu \EEE (Q^\nu_\varrho(x))$ and using the Area Formula. Finally, by \eqref{1910251311}, \eqref{1910251316},  \ZZZZ and \EEE  the notation $\xi = \varphi(x)$,  \EEE
 we conclude \eqref{1910251204}.

\hspace{1em}

\noindent \emph{Step 3: Proof of \eqref{lscminus}.}
Similarly to \eqref{1910251228},  there exists $\lambda \in \M_b^+(\Omega')$ such that, \EEE up to a subsequence,
 \begin{equation}\label{0304261441}
\lambda_n:= \BBB (\psi_n-c_g ) \EEE  \hn\mres_{K_n} \weaklystar \lambda \quad\text{in }  \ZZZZ \M^+_b \EEE (\Omega'). \EEE
\end{equation}
With the above notation  (in particular $\nu:= \nu_K(x)$), \EEE we claim that 
for every $\varphi\in \Psi$, $x \in K$, and  \ZZZZ  $\varepsilon>0$ there exists $\ol \varrho>0$ such that for any $\varrho \in (0 ,\ol \varrho)$ with \EEE $ \ZZZZ Q^{\nu}_\varrho \EEE (x)\EEE\subset \Omega'$ and $\ZZZZ \lambda \EEE  (\partial \ZZZZ  Q^{\nu}_\varrho \EEE (x))\EEE=0$, it holds that \EEE
\begin{equation}\label{0304261443}
g(\varphi(x))-c_g\leq \frac{\lambda(Q^\nu_\varrho(x))}{\varrho^{d{-}1}} \ZZZZ + \eps. \EEE
\end{equation}
As above, we then get  $(g(\varphi)-c_g) \hn \mres_K \leq \lambda$ in $\ZZZZ \M^+_b \EEE (\Omega')$, and then we conclude \EEE
by the Monotone Convergence Theorem, recalling the definition of $\psi$  and \eqref{0304261441}. \EEE
To prove \eqref{0304261443}, we argue as done   for \eqref{1910251204}, \EEE  replacing \eqref{1910251311} with  
\begin{equation*}\label{03041449}
\int_{J_v}( g(|[v]|)-c_g) \,  \dhn\leq \BBB \liminf_{n \to \infty} \EEE \int_{J_{v_n}} ( {g} \EEE(|[v_n]|)-c_g)\,\dhn \leq \liminf_{n\to \infty}\lambda_n(Q^\nu_\varrho(x)) + \varepsilon \ZZZZ \varrho^{d{-}1} \EEE =\lambda(B_\varrho(x))+\varepsilon \ZZZZ \varrho^{d{-}1}, \EEE
\end{equation*}
for the same $(v_n)_n$ 
(since 
$(\widehat{g}  (\cdot) \EEE -c_g) - (\psi_n-c_g)= \BBB  \widehat{g}   (\cdot) \EEE -\psi_n$).
  Here,   we used the  
 lower semicontinuity of $v\in \PC(Q^\nu_\varrho(x)\EEE)\mapsto\int_{J_v}( g(|[v]|)-c_g) \,  \dhn$ 
 in the strong $L^1$-topology,  given that \EEE 
   $g-c_g$ is  nonnegative, \EEE continuous, increasing, and subadditive \ZZZZ on $(0,\infty)$ \EEE  by \eqref{gproperty}.
\end{proof}
 
%

\EEE

%

By \eqref{lscminus} we  immediately \EEE deduce the following.
\begin{corollary}\label{cor:lsc-neu22}
If 
$\limsup_{n\to \infty} \int_{K_n \cap A} \psi_n \, {\rm d}\mathcal{H}^{d{-}1} \le \int_{K \cap A} \psi \, {\rm d}\mathcal{H}^{d{-}1} + \lambda$ for some $\lambda \ge 0$ \BBB and $A \in \mathcal{A}(\Omega')$, \EEE
then
\begin{align}\label{lsc-neu22}
 \limsup_{n\to \infty}  \mathcal{H}^{d{-}1}(K_n \cap A) \le  \mathcal{H}^{d{-}1}(K \cap A) + C\lambda
\end{align}
 for a constant $C > 0$ only depending on $c_g$. \EEE
\end{corollary}

  \begin{proposition}[\BBB Closedness of the irreversibility condition\EEE]\label{prop:stability}
Suppose that  
$(K_n,\psi_n)_n \ZZZZ \subset \EEE \mathcal{C}$ $\sigma_{\rm cf}$-converges in $\Omega'$ to $(K,\psi)$. Let $(u_n)_n$ be a sequence in $\SBV^p(\Omega')$ with $u_n \rightharpoonup u$ 
in $\SBV^p(\Omega')$ 
such that 
\begin{align}\label{follow}
\lim_{n \to \infty} \mathcal{H}_n^-(u_n, \Omega') = 0,
\end{align}
where $\mathcal{H}_n^-$ is defined in \eqref{1708240952}.  Then $J_u \subset  K$ and  $g(|[u](x)|)   \le \psi(x)$ for $\mathcal{H}^{d{-}1}$-a.e.\  $x \in J_u$. 
\end{proposition} 
 
\begin{proof}
By the  definition of $\sigma_{\rm cf}$-convergence and  by the $\Gamma$-convergence of $\E_n^-(\cdot, \Omega')$ to $\E^-(\cdot, \Omega')$ along with the separate lower semicontinuity of bulk and surface parts (see Remark~\ref{rem:tothm31}) \BBB we get  \EEE
 that 
\begin{align*}
\int_{J_u}  {h}^-(\cdot, [u],    \nu_u)  \,  \dhn & \le \liminf_{n \to \infty} \int_{ J_{u_n}} \big(  g(|[u_n]|)   -  {\psi}_n\big)^+   \,  \dhn  = 0,
\end{align*}
where the last step follows from \eqref{follow}.  This first shows ${h}^-(x, [u](x),    \nu_u(x)) =0$ for  $\mathcal{H}^{d{-}1}\text{-a.e.\ }x \in J_u$, which together with the monotonicity of $\xi \mapsto h^-(x,\xi,\nu)$ (see Remark \ref{rem: sigma1}(ii)) yields  $\lim_{\xi \to 0}  h^-(x, \xi, \nu_u(x))=0$ for $\mathcal{H}^{d{-}1}\text{-a.e.\ }x \in J_u$. 
Thus, $J_u \subset  K$  by \eqref{maxK2}. Moreover, the fact that $h^-(x, [u](x), \nu_K(x))=0$ for $\mathcal{H}^{d{-}1}\text{-a.e.\ }x \in K$ (setting $[u](x) = 0$ for $x \in K \setminus J_u$) also implies   $g(|[u]|) \EEE \le \psi $ a.e.\ on $J_u$ by \eqref{maxpsi2}.  
\end{proof}

We now study \BBB unilateral \EEE minimality under boundary conditions.  Recall the definitions of $\mathcal{C}$ \ZZZZ in \EEE \eqref{mathcalcdef} and of $AD(w,K,\psi)$ \ZZZZ in \EEE   \eqref{def:ADgH}-\eqref{constraints}, \ZZZZ as well as \EEE the convention \eqref{convention}. \EEE
%
\begin{definition}[Unilateral minimizer]\label{def: unilat min}
 Let $w \in W^{1,p}(\Omega')$. Let \EEE $h^-$ \ZZZZ  and $(K,\psi) \in \mathcal{C}$ be \EEE as given in Definition~\ref{def:sigmaconv}.  We  say that $(u,K,\psi)$ is:
 
(i) 
 a \EEE \emph{weak unilateral minimizer with respect to $h^-$}   if  $u\in AD(w, K,\psi)$ and  
\begin{align}\label{unilat-weak}
\int_{\Omega'} |\nabla u|^p \, {\rm d}x \le \int_{\Omega'} |\nabla v|^p \, {\rm d}x   +  \int_{J_v} h^-(\cdot,[v],\nu_v)   \, {\rm d}\mathcal{H}^{d{-}1} \quad \text{for any }  v \in \SBV^p_w(\Omega').
\end{align}

(ii) 
a \emph{strong unilateral minimizer} if   $u\in AD(w, K,\psi)$ and
\begin{align}\label{unilat-strong}
\int_{\Omega'} |\nabla u|^p \, {\rm d}x \le \int_{\Omega'} |\nabla v|^p \, {\rm d}x +   \int_{J_v}  \big( g(|[v]|)  - \psi  \big)^+ \, \dhn \quad \text{for any }  v \in \SBV^p_w(\Omega').
\end{align}
\end{definition}

\BBB We emphasize that (ii) corresponds to a notion of minimality which is equivalent to  \eqref{finalstability},  for $p=2$. \EEE

\begin{example}\label{exAB}
Let us explain why in general  (i) does not imply (ii) \EEE
by providing two examples representing different mechanisms. We also refer to Figure~\ref{fig1} \BBB for an illustration. \EEE

(A) Consider a sequence of strongly oscillating curves $(K_n)_n \ZZZZ \subset \R^2 \EEE $ which $\sigma$-converges \cite{GiacPonsi} to the straight segment $K = (0,1) \times \lbrace 0 \rbrace \subset \R^2$ and satisfies  $ \kappa := \EEE \liminf_{n\to \infty}  (2\rho)^{-1}\mathcal{H}^1(K_n \cap B_\rho(x)) \ge 1$ for all $\rho >0$ and $x \in K$. (For instance, let any $K_n$ be the graph of a function $f_n(x):= n^{-1} \varphi(n x)$ defined in $(0,1)$ for $\varphi$ Lipschitz and 1-periodic with $\int_0^1 \sqrt{1+|\varphi'|^2}=\kappa$, e.g., $\varphi= (\l_\kappa \cdot) \chi_{[0,1/2]} +  (   l_k \EEE     -l_k   \cdot ) \chi_{[1/2,1]}$, $l_k:= \sqrt{\kappa^2-1}$, on $[0,1]$ and extended by periodicity.)
 \EEE
 Choose $\psi_n \equiv  \tau \EEE >0$ on $K_n$. Then, \BBB by using competitors of the  form $\xi\chi_{Q^{e_2}_\rho(x) \cap D_n}$ in \eqref{12040854}, where $D_n \subset Q^{e_2}_\rho(x) $ is the region above $K_n$, \EEE it is elementary to check that 
 the $\sigma_{\rm cf}$-limit is given by $K$ and $\psi =   \tau \EEE   $ on $K$, and the function $h^-$ satisfies $h^-(x,\xi,e_2) =  \kappa \EEE (  g(|\xi|)   - \tau)^+  $ for $x \in K$ \BBB and $\xi >0$. \EEE  For $ \kappa \EEE>1$, we thus see that \eqref{unilat-weak} is weaker compared to \eqref{unilat-strong}. 

(B)  Suppose that  $g$ is strictly concave and   satisfies \EEE \eqref{gproperty}, \eqref{gproperty2}, for instance \EEE $g(x) = 1 + \frac{x}{1+x}$.
Consider the sequence of two parallel segments $(K_n)_n$ given by $K_n = K_n^- \cup K_n^+$ with $K_n^\pm =  (0,1) \times \lbrace \pm \frac{1}{n} \rbrace$, which $\sigma$-converges \cite{GiacPonsi} to the straight segment $K = (0,1) \times \lbrace 0 \rbrace \subset \R^2$. Choose $\psi_n \equiv g(\theta_\pm) >0$ on $K^\pm_n$ with $0 < \theta_- < \theta_+$.  \BBB Let $x \in K$ and $\xi >0 $, and define $ \bar{\xi} := ( \xi - (\theta_+ + \theta_-))^+$. Then,  by using competitors taking values $\theta_- + \theta_+ + \bar{\xi} $ (above $K_n^+$),  $\theta_-$ (between $K_n^-$ and $K_n^+$) and $0$ (below $K_n^-$), \EEE   it is elementary to check that   the $\sigma_{\rm cf}$-limit is given by $K$ and $\psi =  g(\theta_+ + \theta_-)   $ on $K$, and the function $h^-$ satisfies 
\[ h^-(x,\xi,e_2) =    g(\theta_+ +  \bar{\xi} \EEE )    - g(\theta_+).  \] 
 As this is bigger than $g(\theta_+ + \theta_- +  \bar{\xi} \EEE )     - g(\theta_+ + \theta_-)$  in the case  $\bar{\xi}>0$, \EEE we thus see that \eqref{unilat-weak} is weaker compared to \eqref{unilat-strong}. 

\end{example}

\begin{remark}\label{weakii}
With the notation of Definition~\ref{def: unilat min}, \eqref{unilat-weak} particularly implies (recall \eqref{def:ADgH}, \eqref{constraints})
$$\int_{\Omega'} |\nabla u|^p \, {\rm d}x \le \int_{\Omega'} |\nabla v|^p \, {\rm d}x  \quad\text{for all $v \in AD(\BBB w,\EEE K,\psi)$}. $$ 
\end{remark} 
 We now study the asymptotic properties of sequences  $(u_n, K_n, \psi_n)_n$ \EEE with 
\begin{equation}\label{seqstrongmin}
(u_n, K_n, \psi_n) \text{ strong unilateral minimizer}, \ (K_n,\psi_n) \cosi (K,\psi),  \ u_n \in \SBV^p_{w_n}(\Omega'), \ u_n \wto u \text{ in } \ZZZZ \SBV^p \EEE (\Omega'),
\end{equation} 
 where \EEE  $(K_n,\psi_n) \cosi (K,\psi)$ means that the sequence  $(K_n,\psi_n)_n$ \EEE $\sigma_{\rm cf}$-converges to $(K, \psi)$ in $\Omega'$, and  where we assume that \EEE  $w_n \to w$ in $W^{1,p}(\Omega')$.
We notice that $(u_n, K_n, \psi_n)$ is a strong unilateral minimizer with $u_n \in \SBV^p_{w_n}(\Omega')$ if and only if \ZZZZ $u_n \in AD(w_n,K_n,\psi_n)$ and 
 $u_n$ is a minimizer of  $\ZZZZ \E_n^{\partial,-}\EEE$ with   $\E_n^{\partial,-}$ defined in   Corollary \ref{cor: bc}. 
%
%
%

\begin{proposition}[Weak unilateral minimality]\label{prop. unilat}
 Assuming \eqref{seqstrongmin}, it holds that \EEE $(u,K,\psi)$ is a  weak   unilateral minimizer with respect to $h^-$   as given in \eqref{1708240957} and that $\nabla u_n \to \nabla u$ in $L^p(\Omega';\R^d)$.
\end{proposition} 

\begin{proof}
As  $(u_n,K_n,\psi_n)$ are  strong \EEE unilateral minimizers, by definition we have $J_{u_n} \subset  K_n$ and $g(|[u_n]|)  \le \psi_n$ on $J_{u_n}$. Then, by Proposition \ref{prop:stability}  we get $J_u  \subset K$ and $g(|[u](x)|) \le \psi(x)$ for $\mathcal{H}^{d{-}1}$-a.e.\  $x \in J_u$.  \ZZZZ This shows that $u \in AD(w,K,\psi)$. \EEE

 Moreover, as  observed before,  $u_n$  \EEE is a minimizer of  \BBB $\E_n^{\partial, -}$. \EEE Thus,  by  \BBB Corollary \ref{cor: bc} \EEE  and general properties of $\Gamma$-convergence, $u$ is a minimizer of  \BBB $\E^{\partial, -}$, \EEE and it holds $\mathcal{E}^-_n(u_n, \Omega') \to \mathcal{E}^-(u, \Omega')$ \EEE as $n\to \infty$.  \BBB (Notice that this holds first for a subsequence, but then for the whole sequence by Urysohn's lemma.) 
 Consequently, for each  $v \in \SBV^p_w(\Omega')$ we get
 \begin{align*}
 \int_{\Omega'} |\nabla u|^p \, {\rm d}x & = \mathcal{E}^-(u,\Omega') \BBB = \E^{\partial, -}(u)  \le  \E^{\partial, -}(v)  = \EEE \mathcal{E}^-(v,\Omega') = \int_{\Omega'} |\nabla v|^p \, {\rm d}x  + \int_{ J_v} h^-(\cdot,[v],\nu_v)  \,  {\rm d}  \mathcal{H}^{d{-}1}.
 \end{align*}
 This shows \eqref{unilat-weak}. 
 By the  separate lower semicontinuity for bulk and surface part, see Remark~\ref{rem:tothm31},  together with $\mathcal{E}^-_n(u_n, \Omega') \to \mathcal{E}^-(u, \Omega')$ we get convergence of both parts of the energy, in particular  $\nabla u_n \to \nabla u$ in $L^p(\Omega';\R^d)$.
\end{proof}

\BBB As discussed in the introduction, \EEE in order to guarantee \emph{strong} unilateral \BBB minimality, we need an additional property, namely convergence of  the measure of the crack sets. This is  \EEE   subject of the next lemma: we define the measures $\lambda_n = \mathcal{H}^{d{-}1}\mres_{K_n}$ and, up to a  not relabeled subsequence, we suppose that the weak$^*$ limit $\lambda \in \BBB \M_b^+(\Omega') \EEE$ exists. Using the  Radon–Nikod\'ym theorem we let 
\begin{align}\label{kappa}
\kappa = \frac{{\rm d}\lambda}{{\rm d}\mathcal{H}^{d{-}1}\mres_K},
\end{align}
and observe that $\kappa \in L^1(K; \mathcal{H}^{d{-}1}\mres_K)$.

\begin{lemma}[Control on $h^-$]\label{lemma: wum}
 Assume that $(K_n,\psi_n) \cosi (K,\psi)$ and \BBB let $\kappa$ as in  \eqref{kappa}. Then,  \EEE for \ZZZZ $\mathcal{H}^{d{-}1}$-a.e.\  \EEE $x \in K$ with $ \kappa(x)  =1$   and all $\xi \in \R$ with $g(|\xi|) \ge \psi(x)$ we have   
$$h^-(x,\xi,\nu_K) \le    g(|\xi|) - \psi(x),$$
\ZZZZ  where $h^-$  is  given as  \EEE  in \eqref{1708240957}.
\end{lemma}

\begin{proof}
Let $\eps >0 $. We write $\zeta = g^{-1}(\psi(x))$    and $\nu = \nu_K(x)$ for shorthand. Up to changing sign, it is not restrictive to treat only the case $\xi \BBB \ge \zeta \EEE $. In view of \eqref{1004261901} \BBB and Remark~\ref{rem: sigma1}(i), \EEE we can choose $r>0$ sufficiently small such that on the cube $Q:=Q^\nu_r(x)$ it holds that
\begin{align}\label{wum1}
h^-(x, \xi, \nu) \le \frac{\mathbf{m}^{\mathrm{PC}}_{\mathcal{H}^-}(\ol u_{x,\xi,\nu}, Q)}{r^{d{-}1}} + \eps , \quad \quad h^-(x, \zeta, \nu) \ge \frac{\mathbf{m}^{\mathrm{PC}}_{\mathcal{H}^-}(\ol u_{x,\zeta,\nu}, Q)}{r^{d{-}1}} - \eps  
\end{align}
and  
\begin{align}\label{wum2}
\text{$\lambda( \ZZZZ \overline{Q} \EEE ) \le (\kappa(x) + \eps)r^{d{-}1} = (1 + \eps)r^{d{-}1}$}
\end{align}
with $\lambda$ as introduced before \eqref{kappa}. Indeed, this can be done for \ZZZZ $\mathcal{H}^{d{-}1}$-a.e.\ \EEE $x \in K$. We choose $z \in  \PC(Q)$ \EEE such that \BBB $z = \ol u_{x,\zeta,\nu} $ in a neighborhood of $\partial Q$ and \EEE
\begin{align}\label{wum3}
  \mathcal{H}^-(z, Q)   \le  \mathbf{m}^{\mathrm{PC}}_{\mathcal{H}^-}(\ol u_{x,\zeta,\nu}, Q) + \eps r^{d{-}1}. 
\end{align}
\BBB Choose a subsequence of $(\mathcal{H}^-_n)_n$ (not relabeled) for which the statement in  Lemma \ref{recovboundra} holds. \EEE Let $(z_n)_n$ be a recovery sequence for  $z$ with respect to the $\Gamma$-convergence (in the strong $L^1$ topology) of  $\mathcal{H}_n^-(\cdot,Q)$ to $\mathcal{H}^-(\cdot,Q)$. \BBB  By Lemma~\ref{recovboundra} \EEE we may assume that $z_n = u_{x,\zeta,\nu}$ in a neighborhood $\partial Q$.
We get 
\begin{align}\label{wum4}
\limsup_{n \to \infty} \mathcal{H}^-_n(z_n, Q) \le   \mathcal{H}^-(z, Q)  \le 2\eps r^{d{-}1},
\end{align}
where we used \eqref{wum1}, \eqref{wum3}, and the fact that $h^-(x, \zeta, \nu) = 0$. Without restriction, by truncation, we can assume that $z_n \colon Q \to [0,\zeta]$.  As a competitor for the minimization problem  $\mathbf{m}^{\PC}_{\mathcal{H}^-}(\ol u_{x,\xi,\nu}, Q)$, we consider the admissible function $w = \frac{\xi}{\zeta} z$. We also define the sequence $w_n = \frac{\xi}{\zeta} z_n $. First, by $\Gamma$-convergence of $\mathcal{H}^-_n(\cdot, Q)$ to $\mathcal{H}^-(\cdot, Q)$ and \eqref{wum1} we get 
\begin{align}\label{frola}
h^-(x, \xi, \nu) \le   \frac{\mathcal{H}^-(w, Q)}{r^{d{-}1}} + \eps \le  \liminf_{n \to \infty}\frac{ \mathcal{H}^-_n(w_n, Q)}{r^{d{-}1}} + \eps.   
\end{align} 
\BBB Recalling \eqref{1708240952}, \EEE  \eqref{wum4} implies $\limsup_{n \to \infty}\int_{J_{z_n} \setminus K_n} \BBB \widehat{g} \EEE (\cdot, \EEE|[z_n]|) \, {\rm d}\Hd \le 2\eps r^{d{-}1}$ as well as 
\begin{align}\label{eq: as well as}
\limsup_{n \to \infty}\mathcal{H}^{d{-}1}(J_{z_n} \setminus K_n) \le 2\eps c_g^{-1} r^{d{-}1}
\end{align}
 by \eqref{gproperty}. Thus, using \eqref{gproperty2}, \BBB \eqref{widehtatg}, \ZZZZ  \eqref{frola}--\eqref{eq: as well as}, \EEE and the fact that $z_n$ takes values in $[0,\zeta]$ \EEE we get  
\begin{align}\label{something on G0}
h^-(x,\xi,\nu)  \le  \liminf_{n \to \infty} \frac{1  }{ r^{d{-}1}} \int_{J_{w_n} \cap K_n} \big( g(|[w_n]|) - \psi_n\big)^+ \, {\rm d} \mathcal{H}^{d{-}1}    + 2 \eps +  2\eps C_gc_g^{-1}\Big(\frac{\xi}{\zeta} -1 \Big) \BBB \zeta \EEE + \eps.   
\end{align}
From \eqref{wum2} and \eqref{eq: as well as} we get 
\begin{equation}\label{14041520}
\limsup_{n \to \infty}\mathcal{H}^{d{-}1}(J_{z_n} ) \le   r^{d{-}1} + C\eps r^{d{-}1}
\end{equation}
\BBB for some $C>0$. \EEE This along with that fact that $(z_n)_n \subset \PC(Q)$ with  $z_n = u_{x,\zeta,\nu}$ in a neighborhood of $\partial Q$ and a slicing argument 
gives that 
$G_n := \lbrace x \in J_{z_n} \colon   [z_n] = \zeta \rbrace$ satisfies
\begin{align}\label{something on G}
\mathcal{H}^{d{-}1}(J_{z_n} \setminus G_n) \le C\eps r^{d{-}1}.
\end{align} 
In fact, as $\# (J_{z_n})^\nu_y\geq 1$ for all $y \in P^\nu(Q)$ (again by $z_n = u_{x,\zeta,\nu}$ near $\partial Q$), by \eqref{14041520} \BBB and  the \EEE Area Formula \EEE we get $\hn(P^\nu(Q)\sm E_n) \le  C \varepsilon r^{d{-}1}$, \BBB where \EEE $E_n:=\{y\in  P^\nu(Q) \colon \# (J_{z_n})^\nu_y=1\}$. \BBB Let $F_n = (E_n + \nu \R)\cap Q$. \EEE  Then,  by Lemma~\ref{le:projection} applied to $\BBB \Psi \EEE =J_{z_n}$ and $A= \BBB  Q \setminus F_n$ \BBB we get $\mathcal{H}^{d{-}1}(J_{z_n} \setminus F_n) \le C\eps r^{d{-}1}$. Eventually, \eqref{something on G} follows from the fact that $G_n \ZZZZ\supset \EEE J_{z_n} \cap F_n$ since, \EEE  for $y\in  J_{z_n} \cap F_n$, $(J_{z_n})^\nu_y=\{t_y\}$ and $[(z_n)^\nu_y](t_y)=[z_n](y+ t_y \nu)=\zeta$ by $z_n = u_{x,\zeta,\nu}$ near $\partial Q$.

Thus, by \BBB \eqref{gproperty2}, \eqref{proprliminfsup}, \EEE  \eqref{wum4},  \eqref{something on G0}, \BBB and \EEE \eqref{something on G} we obtain  
\begin{align*}
h^-(x,\xi,\nu) & \le  \liminf_{n \to \infty} \frac{1}{r^{d{-}1}} \int_{J_{w_n} \cap K_n\cap \lbrace  g(|[w_n]|) > \psi_n\rbrace}  \big( g(|[w_n]|) - g(|[z_n]|) \big)\, {\rm d} \mathcal{H}^{d{-}1}  \\
&  \ \ \  + \BBB \limsup_{n \to \infty} \EEE \frac{1  }{r^{d{-}1}}   \int_{J_{w_n} \cap K_n\cap \lbrace  g(|[w_n]|) >  \psi_n\rbrace}   \big(  g(|[z_n]|) -\psi_n\big)^+ \, {\rm d} \mathcal{H}^{d{-}1} + C\eps\\
& \le  \liminf_{n \to \infty} \frac{1  }{r^{d{-}1}} \int_{J_{w_n} \cap K_n \cap G_n\cap \lbrace  g(|[w_n]|) > \psi_n\rbrace}  \big( g(|[w_n]|) - g(|[z_n]|) \big)\, {\rm d} \mathcal{H}^{d{-}1}  + C\eps,
\end{align*}
where \BBB the constant \ZZZZ $C$ \EEE  also depends on $\xi$, $\zeta$, and $c_g,C_g$. \EEE
Then, by the definition of  
$w_n$  and \eqref{wum2}   we derive 
$$h^-(x,\xi,\nu) \le  \liminf_{n \to \infty} \frac{1}{r^{d{-}1}}\mathcal{H}^{d{-}1}(K_n \ZZZZ \cap Q \EEE )  \big(g(\xi) - g(\zeta))    + C\eps \le   g(\xi) - g(\zeta)    + C\eps.$$
This concludes the proof by the arbitrariness of $\eps > 0$.  
\end{proof}

\begin{corollary}[Strong unilateral minimality]\label{cor. unilat}
Assuming \eqref{seqstrongmin} and \EEE that the function $\kappa$ defined in \eqref{kappa} satisfies
\begin{align}\label{curciiil}
\kappa =1  \ \ \  \text{ \ZZZZ $\mathcal{H}^{d{-}1}$-a.e.\ \EEE  on } K,
\end{align}
it holds that $(u,K,\psi)$ is a strong  unilateral minimizer. \BBB Moreover, \EEE $\nabla u_n \to \nabla u$  in $L^p(\Omega';\R^d)$.
\end{corollary} 
  
\begin{proof}
The statement follows immediately by Proposition \ref{prop. unilat} and the fact that for each $v \in \SBV^p_w(\Omega')$ we have 
$$h^-(x,[v],\nu_v) \le  \big( g(|[v]|) - \psi\big)^+   \text{ for  $\BBB  x \in J_v \cap K$}, \quad h^-(x,[v],\nu_v) \le  g(|[v]|)   \text{ for $\mathcal{H}^{d{-}1}$-a.e.\ $\BBB x \in \EEE J_v \setminus  K$}$$
 by  Lemma \ref{lemma: wum} and Remark \ref{rem: sigma1}(i), respectively.
\end{proof}

%
%

\section{Proof of the existence result}\label{sec: main proof}

This section is devoted to the proof of Theorem \ref{main def}. The proof will be performed by time discretization.   Given   a sequence $(\delta_n)_n \subset (0,\infty)$ with $\delta_n \to 0$, for each $\delta_n$, we consider the subdivision $0=t^0_n <\dots<t_n^{ T /\delta_n}=T$  of the interval $[0,T]$ with step size $\delta_n$. (Without restriction, we assume that $T /\delta_n \in \N$.) Given $w\in W^{1,1}(0,T;H^1( \Omega'))$ satisfying \eqref{uniformnound}, we correspondingly let $w^k_n := w(t^k_n)$ be the sequence of boundary data at different time steps $k\in \{0,\dots, T/\delta_n \}$. 
 We \EEE suppose that the \emph{initial value} $u_n^0 =u^0 \in  \SBV^2_{w(0)}(\Omega')$ is a minimum configuration in the sense that 
 \begin{equation}\label{minimizing-scheme0}
  u^{0}\in  {\rm  argmin} \big\{ E(v) \colon v \in \SBV^2_{w(0)}(\Omega') \big\},
  \end{equation}
  with $E$ given as in \eqref{energy-static-new}. We inductively define an evolution as follows: given $(u_n^j)_{0 \le j \le k}$, we define 
the  \BBB \emph{precrack} \EEE  and  the corresponding \BBB \emph{density of dissipated surface energy} \EEE at time $t_n^k$ by  
   \begin{equation}\label{Tcrack}
 K_n^k  := \bigcup_{j = 0}^{k} J_{u^j_n}, \quad \quad   \psi_n^k(x) := \max_{0\le j \le k}   g(|[u^j_n](x)|) \ \ \text{ for } x \in K_n^k,
  \end{equation}
where for definiteness we set $[u^j_n](x) = 0$ for $x \notin  \BBB J_{u^j_n} \EEE $. Then, we \BBB choose \EEE
     \begin{equation}\label{minimizing-scheme}
    u_{n}^{k{+}1}\in \argmin \big\{    \mathcal{E}(v, K_n^k,\psi_n^k) \colon \,  v\in \SBV^2_{w_n^{k{+}1}}(\Omega') \big\},
  \end{equation}  
  i.e., the minimization of the energy defined in \eqref{eq: lim-en} involves the previous time steps in terms of $(K_n^k,\psi_n^k)$. The existence of minimizers in \BBB \eqref{minimizing-scheme0} and \EEE \eqref{minimizing-scheme}   follows from the direct method and Ambrosio's theorem. Indeed, due to \eqref{uniformnound}, by a truncation argument, we can   assume that all competitors for the minimization \BBB problems \EEE  in \eqref{minimizing-scheme0} and \eqref{minimizing-scheme} satisfy this bound.  Moreover, we observe that  
$\int_{J_u}   \BBB g(|[u]|)  \, \EEE {\rm d}\Hd $ \BBB in \eqref{energy-static-new} \EEE   is    lower semicontinuous, see \cite[Example 5.23]{Ambrosio-Fusco-Pallara:2000}, \BBB as \EEE   $g$ is  nonnegative, \EEE continuous, increasing, and subadditive \ZZZZ on  $(0,\infty)$. \EEE As for \eqref{eq: lim-en}, we can apply \cite[Lemma 8.3]{Giacomini:2005b}, neglecting the Cantor part. \EEE

 We define the evolution  $u_n\colon [0,T] \to \SBV^2(\Omega')$ \EEE by
  \begin{equation}\label{definitionofun}
  u_n(t):= u^k_n  \quad \text{for}\; t\in[t^{k}_n,t^{k{+}1}_n)\,.
  \end{equation}
The \ZZZZ precrack \EEE $K_n$ and the density of dissipated surface energy \EEE $\psi_n$ are defined by
 \begin{equation}\label{definition-crackset}
    K_n(t)\defas K_n^k, \quad  \psi_n(t)  := \EEE \psi_n^k   \quad  \text{for $t\in [t^{k}_n,t^{k{+}1}_n)$}\,.
    \end{equation}  
\ZZZZ In a similar way, we define the boundary conditions $w_n(t):= w^k_n$  for $t\in[t^{k}_n,t^{k{+}1}_n).$

We observe that $(u_n(t), K_n(t), \psi_n(t))$ are strong unilateral  minimizers for every $n\in \N$ and $t \in [0,T]$ in the sense of Definition \ref{def: unilat min} as  $u_n(t)\in AD(w_n(t),K_n(t), \psi_n(t))$ by \eqref{definitionofun}--\eqref{definition-crackset}, and further 
 as
\begin{align}\label{desired equation}
\int_{\Omega'} |\nabla u_n(t)|^2 \, {\rm d}x \le \int_{\Omega'} |\nabla v|^2 \, {\rm d}x +   \int_{J_v}  \big( g(|[v]|)  - \psi_n(t)  \big)^+ \, \dhn\quad \text{for any }  v \in \SBV^2_{w_n(t)}(\Omega').
\end{align}
In fact, for $t$ fixed, we choose $k$ such that $t \in [t_n^{k+1}, t_n^{k+2})$. Then, \eqref{minimizing-scheme} implies   $\mathcal{E}(u^{k+1}_n, K_n^k,\psi_n^k) \le \mathcal{E}(v, K_n^k,\psi_n^k)$ for all $v\in \SBV^2_{w_n^{k{+}1}}(\Omega')$. We use \eqref{Tcrack} to find  $\mathcal{E}(u^{k+1}_n, K_n^k,\psi_n^k) = \mathcal{E}(u^{k+1}_n, K_n^{k+1},\psi_n^{k+1}) $ and $ \mathcal{E}(v, K_n^k,\psi_n^k) \le  \mathcal{E}(v, K_n^{k+1},\psi_n^{k+1})$. From this, we deduce $\mathcal{E}(u^{k+1}_n, K_n^{k+1},\psi_n^{k+1}) \le \mathcal{E}(v, K_n^{k+1},\psi_n^{k+1})$. Eventually, subtracting $\int_{K_n^{k+1}} \psi_n^{k+1} \, {\rm d}\mathcal{H}^{d-1}$ from both sides yields \eqref{desired equation}.  
 
 We now proceed with the proof of the main theorem.  We start by some standard estimates, in particular  a \EEE discrete energy estimate and \EEE  uniform energy bounds on $u_n(t)$  which  enable us to pass to the limit by  compactness and semicontinuity results, see Subsection \ref{subsec1}. The proof of the main result is then given in Subsection \ref{subsec2}. Here, the \BBB main \EEE  novelty lies in  the derivation of \eqref{finalstability} where we  crucially exploit the notion of convergence introduced in Section \ref{subsec:sigmaconv}.  \BBB To this end, we follow the E--S approach, i.e., we first establish the energy balance and prove the stability condition afterwards. \EEE  We defer to Subsection~\ref{subsec3}   a  fundamental part of the proof which deals with the local structure of crack sets \BBB along the evolution. \EEE Roughly speaking, in this step we confirm   assumption \eqref{curciiil} which is needed for the property of strong unilateral minimizers.

\subsection{Energy   estimate,   compactness, and lower semicontinuity}\label{subsec1}

The first  goal of this subsection is to show that the energy of the evolution is bounded uniformly in time. For this, we need to prove an energy estimate on the time-discrete level, which will be also crucial to establish the energy balance. \EEE 

We recall the notation for the time discretization $\{0=t^{0}_n< t^{1}_n<\dots< t^{T /\delta_n}_n =T\}$ of the interval $[0,T]$ with step size $\delta_n$, and introduce the following shorthand notation.   Given an arbitrary $v\in \SBV^2_{w_n^k}(\Omega')$, let
\begin{equation}\label{1908241552}
\mathcal{E}^{0}(v)=E(v)\, \quad \text{and} \quad \mathcal{E}^{k}_n(v) \defas  \mathcal{E}(v, K_n^{k{-}1},\psi_n^{k{-}1})\quad \text{for} \, { k\ge 1, \EEE }
\end{equation}
where $E$ and $\mathcal{E}$ are defined in  \eqref{energy-static-new} and \eqref{eq: lim-en}, respectively, \ZZZZ and $(K_n^{k{-}1}, \psi_n^{k{-}1})$ are given in \eqref{Tcrack}. \EEE  
We start by proving a bound on the elastic part of the energy.  
\begin{lemma}[Elastic energy, uniform bound]\label{prop-control-on-ela}
  Let $t \mapsto u_n(t)$ be the discrete evolution defined in \eqref{definitionofun}.   There exists \EEE a constant   $C_w>0$ \BBB only \EEE   depending on $w$ such that   $E(u^0) \le C_w$ and 
  \begin{align*}
   \int_{ \Omega'} |\nabla {u}_n(t)|^2  \, {\rm d} x + \Vert u_n(t) \Vert_\infty \le C_w \quad \text{ for all $t \in [0,T]$ and for all $n \in \N$.}
   \end{align*}
      
\end{lemma}
\EEE
\begin{proof}
The control on the elastic energy is standard and follows by using $w(0)$ and $w_n^k$ as competitors in  \eqref{minimizing-scheme0} and \eqref{minimizing-scheme}, respectively. As discussed  below \eqref{minimizing-scheme}, by a truncation argument  we can assume that $\Vert u_n(t) \Vert_\infty \le M$  (cf.\ \eqref{uniformnound}) \EEE for all $t \in [0,T]$. 
\end{proof}
\noindent
As an immediate consequence, we obtain the following corollary. 
\begin{corollary}\label{cor: full-ela-control}
    Let $t \mapsto u_n(t)$ be the discrete evolution defined in \eqref{definitionofun}.
    Then, there exists a constant   $C_w>0$  depending on $w$, but independent of $k$ and $n$, such that 
\begin{align}\label{vio}
\int_{0}^{t_{n}^{k}}  \langle \EEE \nabla {u}_n(\tau),   \partial_{t} \nabla w(\tau)  \rangle \EEE \,{\rm d}\tau \le   C_w \quad  \text{for all $k \in \EEE \{0,\ldots,  T/\delta_n\}$}. \EEE
\end{align} 
\end{corollary}
\begin{proof}
 By H\"older's inequality we find 
$$ \int_{0}^{t_{n}^{k}}  \langle \EEE\nabla {u}_n(\tau) ,  \partial_{t} \nabla w(\tau) \rangle \EEE\,{\rm d}\tau  \leq    \, \Vert \partial_t \nabla  w \Vert_{L^1(0,T; L^2(\Omega'))}\; \|\nabla {u}_n\|_{L^{\infty}(\ZZZZ 0, T; \EEE L^2(\Omega'))} \,.$$
 Using \EEE Proposition~\ref{prop-control-on-ela}
and   $w\in W^{1,1}(0,T; H^1(\Omega'))$ we deduce \eqref{vio}. \EEE 
\end{proof}
We continue with the following discrete energy estimate that will be fundamental to establish a uniform bound on the energy and to prove the energy-balance law.   
\begin{lemma}[Discrete energy estimate]\label{e-balance-lemma}
  Let $t \mapsto u_n(t)$ be the discrete evolution defined in \eqref{definitionofun}. Let $k \in \EEE \{0,\ldots,  T/\delta_n\}$. Then, there exists  $(\beta_n)_n$   independent \EEE of $k$ with $\beta_n \to 0$   as $n \to \infty$ \EEE such that 
    \begin{equation*}\label{eq: for en bal}
{        \mathcal{E}_n^k( u_n^k ) - \mathcal{E}^{0}(u^0) \leq 2  \int_{0}^{t_{n}^{k}}  \langle \EEE \nabla {u}_n(\tau), \partial_{t} \nabla  {w}(\tau)   \rangle \EEE \,{\rm d}\tau  +  \beta_n  .   }
    \end{equation*}
\end{lemma}

\begin{proof}
The argumentation follows a well-known strategy, see e.g.\   \cite[Section~3.2]{Francfort-Larsen:2003}, \BBB \cite[Lemma 3.2]{Giacomini:2005b}, or  \EEE \cite[Lemma 6.1]{dMasoFranToad}, and   relies on testing the problem at time $t_n^k$ with $u_n^{k{-}1} + w_n^{k} - w_n^{k{-}1} $. We omit the details.
\end{proof}
  
As a direct consequence, we obtain the following bound on the energy. 

\begin{corollary}[Energy bound]\label{cor: energy bound}
 Let $t \mapsto u_n(t)$ be the discrete evolution defined in \eqref{definitionofun}. Then, there exists a constant $C_w>0$ \EEE only depending on $w$ such that 
\[\mathcal{E}^k_n(u^k_n) \le C_w \quad \text{for all $k  \in \EEE \{0, \ldots, T/\delta_n\} $ and $n \in \N$.} \]
\end{corollary}

\begin{proof}
   The proof follows by combining $\ZZZZ \mathcal{E}^{0} \EEE \ZZZZ (u^0) \EEE \le C_w$ (see Lemma \ref{prop-control-on-ela}), \eqref{vio}, and Lemma \ref{e-balance-lemma}.
\end{proof}

Based on the energy bound in Corollary \ref{cor: energy bound}, we can pass to the limit in the crack sets and displacements by compactness arguments. We start with the crack sets.

\begin{proposition}[Convergence of crack sets\EEE]\label{limit-crack-evo}
  \GGG Let $t\mapsto (u_n(t), K_n(t),\psi_n(t))$ be the evolution defined in \eqref{definitionofun}  and \eqref{definition-crackset}. Then, there  exist  \BBB pairs $(K(t), \psi(t)) \in \mathcal{C}$ \EEE   for $t \in [0,T]$ satisfying
\begin{align}\label{irr-prop}
\text{ $K(t_1) \subset  K(t_2)$ and $\psi(t_1) \le \psi(t_2)$ on  $K(t_1)$   for all $0\leq t_1\leq t_2\leq T$,} 
\end{align}  
 and a subsequence (not relabeled) such that, for every $t\in [0,T]$,  $(K_n(t),\psi_n(t))\cosi (K(t),\psi(t))$. \EEE
 Moreover, for all $U \in \mathcal{A}(\Omega')$ it holds that 
 \begin{align}\label{lsc-app}
\int_{K(t)\cap U} \psi(t) \, {\rm d}\mathcal{H}^{d{-}1} \le \liminf_{n\to \infty} \int_{K_n(t)\cap U} \psi_n(t) \, {\rm d}\mathcal{H}^{d{-}1},
\end{align} 
 \begin{align}\label{lsc2-app}
\mathcal{H}^{d{-}1}(K(t) \cap U )   \le \liminf_{n\to \infty}  \mathcal{H}^{d{-}1}(K_n(t) \cap U).
\end{align}

\end{proposition}

\begin{proof}
The compactness at each time $t$ follows from \BBB Theorem \EEE \ref{sigma-comp}. The fact that the subsequence can be chosen independently of $t$ \ZZZZ and \EEE property \eqref{irr-prop} follow by  Remark~\ref{rem: sigma1}(iii)  and a variant of Helly's theorem.   We sketch  the proof for the reader's convenience: applying \EEE Lemma~\ref{le:Helly} \BBB below \EEE   to the sequence 
\[
f_n(t):= \int_{K_n(t)} \psi_n(t) \,\dhn
\]
and recalling the bound in Corollary \ref{cor: energy bound}, up to extracting a (not relabeled) subsequence, it holds that there exists $E\subset [0,T]$ countable such that for every $\varepsilon>0$ 
\begin{equation}\label{12052136}
\limsup_{n\to \infty} \big(f_n(t^+) - f_n(t^-)\big)   \BBB \le \EEE \varepsilon \quad \text{for all }t \in [0,T]\sm E, \, \text{ with }t\in (t^-,t^+) \text{  for suitable \EEE }t^\pm \in \mathbb{Q}.
\end{equation}
Moreover, \ZZZZ by a diagonal argument, \EEE up to extracting a (not relabeled) subsequence, \ZZZZ from \EEE Theorem~\ref{thm:convpreliminaries}  \ZZZZ we get \EEE
\begin{equation}\label{12052231}
\H^-_n(t) \ \ol \Gamma\text{-converge to $\H^-(t)$ for any $t \in [0,T] \cap \mathbb{Q}$,}
\end{equation}
for $\H^-_n(t)$ defined as in \eqref{hnm} with $\psi_n(t)$ in place of $\psi_n$ therein \BBB (recall \ZZZZ \eqref{convention}, \EEE where $K_n(t)$ now appears). In particular, we notice that for any $A \in \A(\Omega')$ and   $u$ with $u|_A \in \PC (A)$ (otherwise the value is always $+\infty$) \EEE 
   it is elementary to check that \EEE
\begin{equation}\label{12052201}
\H^-_n(s)(u, A)-\H^-_n(t)(u, A)\leq  f_n(t) - f_n(s) \quad\text{ for all }s\leq t.
\end{equation}
Let $t \in [0,T]\sm E$.  By Theorem~\ref{thm:convpreliminaries}, \BBB \eqref{Tcrack},  \eqref{definition-crackset}, \ZZZZ and Corollary \ref{cor: energy bound} \EEE  there exists a subsequence $n(t)$ such that $\H^-_{n(t)}(t)$ $\ol \Gamma$-converge to $\H^-(t)$, with $\H^-(t^+)\leq \H^-(t)\leq \H^-(t^-)$  for  $t^\pm \in \mathbb{Q}$   as in \eqref{12052136}. \EEE For any $u \in  L^1 \EEE (\Omega')$ and $A \in \A(\Omega')$,
by \eqref{12052136}--\eqref{12052201} we have that, for any $u_n\to u$ in $L^1(\Omega')$,
\begin{equation*}
\H^-(t)(u, A)\leq \H^-(t^-)(u, A) \leq \liminf_{n\to \infty} \H^-_n(t^-)(u_n, A)\leq \liminf_{n\to \infty} \H^-_n(t)(u_n, A) + \varepsilon,  
\end{equation*}
and, similarly, that there exists $v_n\to u$ \ZZZZ in $L^1(\Omega')$ \EEE such that
\begin{equation*}
\quad \H^-(t)(u, A)\geq \H^-(t^+)(u, A) \geq \limsup_{n\to \infty} \H^-_n(t^+)(v_n, A) \geq \limsup_{n\to \infty} \H^-_n(t)(v_n, A) -\varepsilon.
\end{equation*} 
(We notice that possibly some values in the above inequalities could be infinite.)  
By the arbitrariness of $\varepsilon>0$ we get that $\H^-_n(t)$ $\ol \Gamma$-converge to $\H^-(t)$ for any $t \in [0,T]\sm E$. As $E$ is countable, up to extracting a further subsequence such that $\H^-_n(t)$ $\ol \Gamma$-converge to $\H^-(t)$ for any $t \in E$ along such  a \EEE  subsequence, we conclude that $\H^-_n(t)$ $\ol \Gamma$-converge to $\H^-(t)$ for any $t \in [0,T]$. Therefore, \BBB by Definitions~\ref{def:sigmaconv} and \ref{def:sigmaconv2}  \EEE we have that $(K_n(t), \psi_n(t))$  $\sigma_{\rm cf}$-converge \EEE to \BBB some \EEE $(K(t), \psi(t)) \BBB \in \mathcal{C}\EEE$.
%
\BBB The estimates \EEE \eqref{lsc-app} and \eqref{lsc2-app}   follow, respectively, from \eqref{lsc} and \eqref{lsc-neu} in Theorem~\ref{sigma-comp}. \BBB Eventually, \eqref{irr-prop} is a consequence of  Remark \ref{rem: sigma1}(iii). \EEE
\end{proof}

\begin{lemma}\label{le:Helly}
Let $f_n\colon [0,T]\to [0,\ZZZZ +\infty)$ be a sequence such that \ZZZZ $\sup_{n \in \N} \Vert f_n \Vert_\infty <+\infty$ and \EEE $f_n(s)\leq f_n(t)$ for all $n \in \N$ and $s\leq t$. Then there exists a countable set $E\subset [0,T]$ such that, up to passing to a (not relabeled) subsequence, for all $t \in [0,T]\sm E$ and for any $\varepsilon>0$ there \ZZZZ exist \EEE $t^\pm\in [0,T] \BBB \cap \mathbb{Q} \EEE$ with $t^-<t<t^+$ such that
\begin{equation*}
\limsup_{n\to \infty} \big( f_n(t^+)-f_n(t^-) \big)   \le    \varepsilon.
\end{equation*}
\end{lemma}
\begin{proof}
By Helly's  theorem, \EEE up to passing to a (not relabeled) subsequence, $f_n(t)\to f(t)$ for any $t\in [0,T]$ for a suitable nondecreasing $f$. Let $E\subset [0,T]$ be the \ZZZZ (at most countable) \EEE set of discontinuity points of $f$. 
Then for every $t \in [0,T]\sm E$ there exist $t^\pm \in [0,T] \ZZZZ \cap \mathbb{Q} \EEE $ with $t^-<t<t^+$ such that $f(t^+)-f(t^-) \BBB \le \EEE \varepsilon/3$ and $|f_n(t^\pm)-f(t^\pm)|\leq \varepsilon/3$ for $n$ large enough, so we conclude by the triangle inequality.
\end{proof}

\EEE

\begin{proposition}[Convergence of displacements]\label{liminf-ineq'}
 \GGG Let $t\mapsto (u_n(t), K_n(t),\psi_n(t))$ be the evolution defined in \eqref{definitionofun}  and \eqref{definition-crackset}. \EEE Let  $(K(t),\psi(t))$ be the $\sigma_{\rm cf}$-limit of $(K_n(t),\psi_n(t))$ given by Proposition \ref{limit-crack-evo}. Then, for all $t \in [0,T]$   any subsequence of $(u_n(t))_n$ is precompact with respect to the weak convergence in \ZZZZ $\SBV^2(\Omega')$ \EEE and any limit point $u(t)$ belongs to $AD(w(t),K(t),\psi(t))$. \EEE
\end{proposition}

\begin{proof}
In view of  Lemma \ref{prop-control-on-ela} and Corollary \ref{cor: energy bound},   for any $t \in [0,T]$ the sequence $(u_n(t))_n$ is bounded in $\SBV^2(\Omega')$ and then, \ZZZZ \EEE by Ambrosio's theorem, \EEE precompact (with any subsequence) with respect to the weak convergence in $\SBV^2(\Omega')$. \EEE
 Given a limit point $u(t)$, as \EEE $u_n(t) = \ZZZZ w_n(t)$, \EEE the regularity of $w$ in time implies that $u(t) = w(t)$ on $\Omega' \setminus \overline{\Omega}$. Moreover, by Proposition~\ref{prop:stability}   we get $J_{u(t)} \subset K(t)$ and $g(|[u(t)]|) \le  \psi(t)$ \ZZZZ $\mathcal{H}^{d{-}1}$-a.e.\ \EEE on $J_{u(t)}$. Here, we used that \eqref{follow} holds (and is actually identically zero along the sequence), due to the definition of ($K_n(t),\psi_n(t))$ in \eqref{Tcrack} and \eqref{definition-crackset}. This shows $ u(t) \in AD(w(t),K(t),\psi(t))$ and concludes the proof. 
\end{proof}

%
%

\subsection{Proof of Theorem \ref{main def}}\label{subsec2}

 In this subsection we prove the main result of the paper.

 \begin{proof}[Proof of Theorem \ref{main def}]
 We start with a brief outline of the proof. In Step~1,   we prove the existence of a limiting evolution and validate the irreversibility of the crack sets by using $\sigma_{\rm cf}$-convergence. In Step~2, we use the weak unilateral stability property to show that the limiting strains are given uniquely at all times. As discussed in Section \ref{subsec:sigmaconv}, the weak stability property is not enough to obtain property (c) in the main statement. Our strategy therefore consists in \emph{first} proving the energy balance which will allow us to apply Corollary \ref{cor. unilat}. To this end, we proceed in several steps. In Step 3, we  pass to the limit in the time-discrete energy estimate \eqref{e-balance-lemma}. In Step 4, we introduce a suitable time discretization of the limiting problem which allows us to define an auxiliary energy balance in Step 5. This is not yet the correct balance, yet it provides   a control on certain \emph{defects} related to the crack appearing along the evolution. This control allows us to characterize the local structure of crack sets \BBB along the evolution, \EEE see Step~6 and Subsection \ref{subsec3} below. Based on this, we eventually obtain the correct energy balance in Step~7. Finally, in Step~8, using Corollary \ref{cor. unilat} we can   guarantee the strong unilateral minimality property in the limit.  
 
Along the proof, we will select subsequences multiple times:   firstly,   subsequences are selected \EEE \emph{independently of $t$} labeled with the 
index $n$;  later, \EEE from those ones 
we   extract $t$-dependent subsequences  labeled with the index \EEE  $m$.

    \emph{Step 1: Limiting evolution.} Let  $(K(t),\psi(t))$ be the $\sigma_{\rm cf}$-limit of $(K_n(t),\psi_n(t))$ for $t \in [0,T]$ given by Proposition~\ref{limit-crack-evo}, and recall that $t\mapsto (K(t),\psi(t))$ satisfies the irreversibility properties  stated in \eqref{irr-prop}.   This shows that condition (b) in Theorem \ref{main def} holds.

By   Proposition~\ref{liminf-ineq'},  for any $t \in [0,T]$ \EEE there exists $u(t) \in  AD(w(t),K(t),\psi(t)) \EEE$  which is a \EEE  limit point of $(u_n(t))_n$ with respect to the weak convergence in $\SBV^2(\Omega')$. \EEE

\emph{Step 2: Weak stability and uniqueness of limiting strains.} 
Fix $t \in [0,T]$ and let  $u(t)$ 
\EEE be  the \EEE  limit point of $(u_n(t))_n$   found \EEE above. \EEE
  As \EEE $(u_n(t),K_n(t),\psi_n(t))$ are \BBB strong \EEE unilateral minimizers for each $n \in \N$  (see below \eqref{definition-crackset}), \EEE  by  Proposition~\ref{prop. unilat} we get that  $(u(t),K(t),\psi(t))$ is a weak unilateral minimizer and that   $  \nabla \EEE u_n(t)$ converge   in \EEE $L^2$ to $\nabla u(t)$ along a suitable subsequence \ZZZZ (depending on $t$). \EEE
 In particular, \ZZZZ $u(t) \in  AD(w(t), K(t),\psi(t))$ and,  \EEE   in view of Remark~\ref{weakii},   for  every $v\in AD(w(t),K(t),\psi(t))$  we have
\begin{align}\label{2108241743-new000}
\int_{\Omega'} |\nabla u(t)|^2 \, {\rm d}x \le \int_{\Omega'} |\nabla v|^2 \, {\rm d}x .
    \end{align} \EEE 
Clearly, this property does not yet imply the desired stability property, but it allows to show that the  limiting strain $\nabla u(t)$ obtained in Step 1 is uniquely determined on $\Omega'$. Once this is shown,  Urysohn's lemma \EEE and Lemma~\ref{prop-control-on-ela} imply  that 
\begin{align}\label{independentoft}
\nabla u_n(t) \to \nabla u(t) \text{ in $L^2(\Omega';\R^d)$}  \quad \text{for every } t \in [0,T] \EEE
\end{align}
 for a subsequence $(u_n(t))_n$ which can be chosen \emph{independently of $t$}. This property will allow us to pass to the limit in the time-discrete energy estimate given by Lemma \ref{e-balance-lemma}, see Step 3 below. 
 
 Now, to show the uniqueness of the limiting strain $\nabla u(t)$, let us suppose that $\hat{u}(t) \in AD(w(t),K(t),\psi(t))$ denotes another limit   point \EEE of the sequence $u_n(t)$.   Proceeding as at the beginning of this step, we can show that \eqref{2108241743-new000} holds for $\hat{u}(t)$ in place of $u(t)$. This yields that both $u(t)$ and $\hat{u}(t)$ are minimizers of the strictly convex minimization problem $v \mapsto   \int_{\Omega'}|\nabla v|^2\, {\rm d}x $    for $v \in AD(w(t),K(t),\psi(t))$. \BBB (Note that  $AD(w(t),K(t),\psi(t))$ is convex by the monotonicity of $g$.)
 \EEE Therefore, $\nabla u(t) = \nabla \hat{u}(t)$ on $\Omega'$.

\emph{Step 3: Energy inequality.}  \BBB Next, we will prove the auxiliary energy balance. In this step, we pass  to the limit in the time-discrete energy estimate \ZZZZ of Lemma \ref{e-balance-lemma} \EEE which provides  \EEE   a first energy inequality.  In view of Lemma~\ref{prop-control-on-ela} and $w\in W^{1,1}(0,T; H^1(\Omega'))$, the \BBB functions  $t\mapsto  \langle  \nabla  {u}_n  (t),  \partial_{t} \nabla w(t) \rangle$ are dominated by $C\Vert \partial_t \nabla w(t) \Vert_{L^2(\Omega')} \in L^1(0,T)$.  \EEE Then, by \eqref{independentoft},  we can apply \ZZZZ the \EEE Dominated Convergence Theorem to deduce that  
     \begin{align}\label{limit-testing}
      & \lim_{n\to \infty} \int_{0}^{t}  \langle \EEE  \nabla  {u}_n  (\tau),   \partial_{t} \nabla w(\tau)  \rangle \EEE  \, {\rm d}\tau =  \int_{0}^{t}  \langle \EEE  \nabla u(\tau), \partial_{t} \nabla w(\tau)  \rangle \EEE \,{\rm d}\tau   \quad \text{for all } t \in [0,T] \EEE \,.
     \end{align} 
     By Lemma \ref{e-balance-lemma} we find $(\beta_n)_n$ with $\beta_n \to 0$ such that
  \begin{align*}
\mathcal{E}^k_n(u_n(t))  - \mathcal{E}^{0}(u^0) &\leq 2  \int_{0}^{t_{n}^{k}}  \langle \EEE\nabla {u}_n(\tau) , \partial_{t} \nabla  {w}(\tau)  \rangle \EEE\,{\rm d}\tau  +  \beta_n. 
\end{align*}
Recalling \eqref{eq: lim-en}, \eqref{Tcrack}, \eqref{definition-crackset}, and \eqref{1908241552} we find $\mathcal{E}^{k}_n(u_n(t)) =  \mathcal{E}(u_n(t), K_n^{k{-}1},\psi_n^{k{-}1}) = \mathcal{E}_n(u_n(t), K_n(t),\psi_n(t) )$ \ZZZZ for all   $t\in [t_n^{k}, t_n^{k{+}1})$. \EEE Then, by passing to another sequence $(\beta_n)_n$, still satisfying $\beta_n \to 0 $, \EEE \EEE  for every $k  \in \{ \EEE 0, \ldots, T/\delta_n\}$ and  $t\in [0,T]$ we have   
    \begin{align}\label{energy-diff}   
\ZZZZ \mathcal{E} \EEE (u_n(t), K_n(t),\psi_n(t) )     &\leq 2 \,\int_{0}^{t}  \langle \EEE \nabla   {u}_n  (\tau), \partial_{t} \nabla w(\tau) \rangle \EEE\,{\rm d}\tau  +  \BBB \beta_n  \EEE  +  \mathcal{E}^{0}(u^0)\,,
  \end{align}      
  where we have used  Lemma \ref{prop-control-on-ela} and $w\in W^{1,1}(0,T; H^1(\Omega'))$ to  estimate the integral from $t_n^k$ to $t$ in terms of  (the former) \EEE $\beta_n$. Recalling   \eqref{minimizing-scheme0}, \eqref{1908241552}, and using \eqref{lsc-app} for $t =0$  we have 
 \begin{align}\label{attimezero}
 \mathcal{E}(u(0), K(0), \psi(0) ) \le  \mathcal{E}^0(u^0) =    \min_{\SBV^2_{w(0)}(\Omega')} \mathcal{E}^0 \le   \mathcal{E}(v, J_{v}, \BBB g(|[v]|))   \EEE
  \end{align}
   for all  $v\in \SBV^2_{w(0)}( \Omega') \EEE $. This shows that condition (a) in Theorem \ref{main def} holds. By choosing $v = u(0)$  and using  $J_{u(0)} \subset   K(0)$, $g(|[u(0)]|) \le \psi(0)$ \ZZZZ on $J_{u(0)}$ \EEE (see Proposition \ref{liminf-ineq'}),  we get   $J_{u(0)} =K(0)$, $g(|[u(0)]|) = \psi(0)$ \ZZZZ on $J_{u(0)}$, \EEE  and \EEE $  \mathcal{E}^0 (u^0) = \mathcal{E}(u(0), K(0),\psi(0))$.   \ZZZZ Using \EEE   \eqref{limit-testing}--\eqref{energy-diff} we get
  \begin{align}\label{eq: for en ba}
 \limsup_{n\to \infty} \mathcal{E}(u_n(t),K_n(t),\psi_n(t)) &\leq \mathcal{E}^{0}(u^0)  +  \lim_{n\to \infty} 2\int_{0}^{t}  \langle \EEE \nabla {u}_n \EEE (\tau) , \partial_{t} \nabla w(\tau) \rangle \EEE\,{\rm d}\tau \notag  \\ &=  \mathcal{E}(u(0), K(0),\psi(0)) + 2 \int_0^t   \langle \EEE \nabla u(\tau), \partial_{t} \nabla  w(\tau)  \rangle \EEE \, {\rm d}\tau.
  \end{align}
  
Our next steps consist in finding a lower bound for the expression  $\limsup_{n\to \infty} \mathcal{E}(u_n(t),K_n(t),\psi_n(t))$  which matches the right-hand side of \eqref{eq: for en ba} up to some defects that we will control later.  This will establish an   auxiliary energy balance with defects, see \eqref{eq: auxiliary balance} below.  We start with some preliminaries.

\emph{Step 4: Time discretization of the limit.} As customary in the derivation of energy balances, a time discretization of the limit along with a stability property is used. The challenge in our setting lies in the fact that from Step 2 we only have a (insufficient) weak unilateral stability.   
Let us introduce a suitable discretization. 
 By \cite[Lemma 4.12, Remark 4.13]{dMasoFranToad} (and also inspecting the proof of \cite[Lemma 4.12]{dMasoFranToad}), there exists $\ol s \in \R$ such that for any $\varepsilon>0$ there exists $N_\varepsilon \in \N$ such that, for any $t \in [0,T]$, 
the partition $(s_i)^\kt_{i=0}$ defined by  
\begin{equation}\label{23061152}
s_0:=0,\ \   s_\kt := t,\  \ \ZZZZ  \text{and} \ \ \EEE (s_i)^{\kt-1}_{i=1}= (0,t) \cap \Big(\ol s + \frac{1}{N_\varepsilon} \ZZZZ \mathbb{Z} \EEE \Big), \text{ with }s_i\leq s_{i{+}1} \text{ for }i \in \{ \ZZZZ 0, \EEE \dots, \ZZZZ \kt-1 \EEE \},\, 
\end{equation}
 satisfies \EEE
(denoting   $\|\cdot\|_{\Ldue} \equiv \|\cdot\|_{L^2(\Omega')}$ \ZZZZ for shorthand) \EEE
\begin{align}\label{Riemann}
{\rm (i)} & \ \ \sum_{i=1}^\kt \int_{s_{i{-}1}}^{s_i}  \Vert \nabla u(s_i) \cdot \partial_t \nabla w(s_i)  -  \nabla u(\tau) \cdot \partial_t \nabla w(\tau)\Vert_{2}  \, {\rm d}\tau  \le \eps,\notag \\
{\rm (ii)} & \ \  \sum_{i=1}^\kt \int_{s_{i{-}1}}^{s_i}  \Vert   \partial_t \nabla w(s_i)  -   \partial_t \nabla w(\tau)\Vert_{2}  \, {\rm d}\tau  \le \eps  . 
\end{align}
Recalling the regularity of $w$, by choosing   $N_\varepsilon$ large enough, \EEE
we can additionally assume that 
\begin{align}\label{Riemann2}
\sum_{i=1}^\kt \Vert \nabla w(s_i) - \nabla w(s_{i{-}1}) \Vert_{2}^2 \le \eps. 
\end{align}
By the Fundamental Theorem of Calculus we compute
\begin{align*}
\sum_{i=1}^\kt \langle \nabla u(s_i) , \nabla w(s_i)-\nabla w(s_{i{-}1})  \rangle & =\sum_{i=1}^\kt  \Big\langle \nabla u(s_i) , \int_{s_{i{-}1}}^{s_i}  \partial_t \nabla w(\tau) \, {\rm d}\tau  \Big\rangle  =   \sum_{i=1}^\kt  \int_{s_{i{-}1}}^{s_i} \langle \nabla u(s_i) , \partial_t \nabla w(s_i)\rangle \,  {\rm d}\tau + \Psi,
\end{align*}
where $\Psi := \sum_{i=1}^\kt \langle  \nabla u(s_i) ,    \int_{s_{i{-}1}}^{s_i}(  \partial_t \nabla w(\tau) -\partial_t \nabla w(s_i) )  \,{\rm d}\tau \rangle$ satisfies, by Minkowski's integral inequality, \eqref{Riemann}(ii), and Lemma~\ref{prop-control-on-ela},
\begin{equation*}
|\Psi| \le \sum_{i=1}^\kt  \Vert \nabla u(s_i) \Vert_{2} \int_{s_{i{-}1}}^{s_i}  \Vert\partial_t \nabla w(s_i) - \partial_t \nabla w(\tau)
\Vert_{\Ldue}  \,{\rm d}\tau  \le C\sum_{i=1}^\kt \int_{s_{i{-}1}}^{s_i}  \Vert\partial_t \nabla w(s_i) -  \partial_t \nabla w(\tau)
\Vert_{\Ldue}  \,{\rm d}\tau  \le C\eps.
\end{equation*}
This along with \eqref{Riemann}(i) then yields
\begin{align}\label{Riemann3}
&\Big|\sum_{i=1}^\kt \langle \nabla u(s_i) , \nabla w(s_i)-\nabla w(s_{i{-}1})  \rangle -  \int_0^{t}  \langle \EEE \nabla  u(\tau) , \partial_t  \nabla w(\tau) \rangle \EEE\, \mathrm{d}\tau \Big|  \notag \\ &\le \sum_{i=1}^\kt \int_{s_{i{-}1}}^{s_i}  \Vert \nabla u(s_i) \cdot \partial_t \nabla w(s_i)  -  \nabla u(\tau)\cdot \partial_t \nabla w(\tau)
\Vert_{\Ldue} \, {\rm d}\tau   + |\Psi| \le C\eps.
\end{align}
Given $\bar{s} \in \R$ as above, by  Proposition~\ref{liminf-ineq'}, \EEE potentially redefining the limit $u(t)$ (without changing its  gradient, \EEE see Step~2), and by a diagonal argument,  up to passing to a (not relabeled) subsequence independent of $t$ it holds that
 $u_n(t) \wto u(t)$ in $\SBV^2(\Omega')$ as $n \to \infty$ \EEE and $\lim_{n\to \infty} \mathcal{E}_{n}(u_n(t),K_n(t),\psi_n(t))$ exists for all  $t \in (\bar{s} + \mathbb{Q}) \cap [0,T]$.
  Moreover, for all $[0,T] \setminus (\bar{s} + \mathbb{Q})$ there exists a $t$-dependent subsequence
   of $n$, denoted by $m$ in the sequel, such that $u_m(t) \wto u(t)$ in $\SBV^2(\Omega')$ and
\begin{equation}\label{mmm0} 
 \lim_{m\to \infty} \mathcal{E}_{m}(u_m(t),K_m(t),\psi_m(t))=\liminf_{n\to \infty} \mathcal{E}_{n}(u_n(t),K_n(t),\psi_n(t)). \EEE
 \end{equation}
Let us  fix \EEE
$\varepsilon>0$ and $t \in [0,T]$. In particular, also $\kt$  and the corresponding partition $(s_i)_{i=0}^{\kt}$ are fixed, see \eqref{23061152}. From now on we denote $k\equiv \kt$ for shorthand. \EEE
   \ZZZZ For any $i \in \{0,\dots, k\}$ we have $ u(s_i)  \in AD(w(s_i),K(s_i),\psi(s_i))$ and thus particularly \EEE
\begin{equation}\label{13050959}
 J_{u(s_i)}\tsubset K(s_i), \quad   g(|[u(s_i)]|) \le \psi (s_i) \text{ on }J_{u(s_i)}.
\end{equation}   \EEE
Let $ \overline{\psi}_i:=\max_{0\leq l\leq i} g(|[u(s_l)]|)$  for $i \in \{0,\dots, k\}$. \BBB (We adopt \EEE the notation $ \overline{\psi}_i$ to highlight the fact that this \ZZZZ function \EEE  is defined by considering only the values of $[u]$ on $(s_i)_{i=0}^k$, differently from $\psi(s_i)$ which represents an evaluation of $\psi$ defined for every $s \in [0,T]$.) For $i \in \{0,\dots, k\}$, let us define the measures
\begin{equation}\label{defmu}
\mu_n^i:=\psi_n(s_i) \hn\mres_{K_n(s_i)},\quad  \mu^i:= \BBB \ol \psi_i  \EEE  \hn\mres_{K(s_i)}
\end{equation}
related to the surface energy along the sequence and in the limit, respectively. By \eqref{Tcrack} and property (b) in Theorem \ref{main def} we have,  for every $i \in \{0,\dots, k{-}1\}$ and every $n\in\N$,
\begin{equation}\label{defmu-mono}
\mu_n^i\leq \mu_n^{i{+}1}, \quad \quad \mu^i\leq \mu^{i{+}1}.
\end{equation}
By a covering argument we can choose an increasing sequence $U_i \in \mathcal{A}(\Omega')$, $i\in \{ \EEE 0,\ldots, k\}$, with $U_k = \Omega'$ such that 
\begin{align}\label{uu}
\mu^i(\Omega' \setminus U_i)  < \EEE \frac{\eps}{k}, \quad  \quad \mu^k\big(U_i\sm K(s_i)\big)  < \EEE  \frac{\eps}{k}, \quad \quad  \mu^k(\partial U_i)=0.
\end{align}
Intuitively, $U_i$ contains almost all crack $K(s_i)$ that appeared up to time $s_i$ and there almost does not appear any further crack $K(s_k) \sm K(s_i)$ inside $U_i$ at later times. It is elementary to check that  $\mu_n^l(U_l)=  \sum_{i=1}^l (\mu_n^i-\mu_n^{i{-}1})(U_{i{-}1}) + \sum_{i=0}^l \mu_n^i(U_i \sm U_{i{-}1})$ by an induction over $l \in \lbrace 0,\ldots, k \rbrace$, where for definiteness we have set   
$U_{-1}:=\emptyset$.  In particular, for $l = k$ we obtain 
\begin{equation}\label{1910251906}
\mu_n^k(\Omega')=  \sum_{i=1}^k (\mu_n^i-\mu_n^{i{-}1})(U_{i{-}1}) + \sum_{i=0}^k \mu_n^i(U_i \sm U_{i{-}1}).
\end{equation}
Contributions in the second sum will be labeled with `new' as they correspond to the new crack appearing in the set $U_i \sm U_{i{-}1}$ in the time interval $(s_{i{-}1},s_i]$.  Contributions in the first sum will be labeled with `add' as they correspond to additional jump in $U_{i{-}1}$ which appears  in the time interval $(s_{i{-}1},s_i]$. (Note that in $U_{i{-}1}$ new crack appeared in some earlier time interval, namely $(s_{j-1},s_j]$ for some $j \in \{ \EEE 1,\ldots,i{-}1\}$.) The additional crack can consist in new crack $K_n(s_i) \setminus K_n(s_{i{-}1})$ or in  bigger jump height on $K_n(s_{i{-}1})$.  Note that, due to the second item in  \eqref{uu}, up to small errors,  additional crack in the limit corresponds  only to bigger jump height on $K(s_{i{-}1})$.

We formalize this by introducing additional measures representing these two distinct contributions to the surface energy. For every $i \in \{0,\dots, \BBB k \EEE \}$ and every $n\in\N$ we define
\begin{equation}\label{defmutilde}
{\mu}_{n, {\rm new}}^i:=\mu_n^i\,\mres_{U_i\sm {U_{i{-}1}}},\quad \quad  \BBB {\mu}_{ {\rm new}}^i:= \mu^i\,\mres_{U_i\sm {U_{i{-}1}}} . \EEE
\end{equation}
\BBB Moreover, we let ${\mu}_{n, {\rm add}}^0 = {\mu}_{ {\rm add}}^0 = 0$, and for $i \in \{1,\dots,  k \}$ \EEE 
\begin{equation}\label{defmutilde2}
{\mu}_{n, {\rm add}}^i:=(\mu_n^i-\mu_n^{i{-}1})\,\mres_{U_{i{-}1}} ,\quad \quad  {\mu}_{\rm add}^i:=(\mu^i-\mu^{i{-}1})\mres_{U_{i{-}1}} .  
\end{equation} 
 Up to extracting  from $m$ a subsequence (not relabeled), 
 \EEE we may assume that   
 the limits 
 \begin{equation}\label{1910252147}
{\mu}_{\n}^i  \weaklystar {\mu}_{\rm conv}^i, \quad {\mu}_{\n, {\rm new}}^i  \weaklystar {\mu}_{\rm new, conv}^i, \quad  {\mu}_{\n, {\rm add}}^i \EEE \weaklystar  {\mu}_{\rm add, conv}^i   \quad  \text{in } \ZZZZ \M^+_b \EEE (\Omega') \EEE
 \end{equation}
  exist for $i\in\{0,\dots, k\}$.  (Notice that for $t \in (\ol s +\mathbb Q) \cap [0,T]$ we could  extract also from $n$, see below \eqref{Riemann3}, but we do not need this and denote the subsequence in \eqref{1910252147} always with $m$.) \EEE

%
 Here, the label `conv' indicates that the object is not related to the limiting objects (as \BBB the second items \EEE in \eqref{defmu} or \eqref{defmutilde}--\eqref{defmutilde2}) but obtained from the weak convergence of the sequence $({\mu}_{\n}^i )_\n$. For establishing the energy balance, it will be fundamental to show that ${\mu}_{\rm conv}^k$ and  ${\mu}^k$ coincide    up to   an error of order $\eps$, \BBB see \eqref{asmotivated}  below.

\emph{Step 5: Auxiliary energy balance with defects.} For any  $i \in  \{1,\ldots, k\}$, as  $(u_n(s_{i{-}1}),K_n(s_{i{-}1}),\psi_n(s_{i{-}1}))$ are strong unilateral minimizers for each $n \in \N$, and as $u_\n(s_{i{-}1})
\wto u(s_{i{-}1})$ in $\SBV^2(\Omega')$ and   $w_\n(s_{i{-}1}) \to w(s_{i{-}1})$   in $H^1(\Omega')$, \EEE  we derive that $(u(s_{i{-}1}),K(s_{i{-}1}),\psi(s_{i{-}1}))$ is a weak unilateral minimizer by Proposition~\ref{prop. unilat}, i.e., for all functions $v \in \SBV^2_{w(s_{i{-}1})}(\Omega')$ we have  
$$
\int_{\Omega'} |\nabla u(s_{i{-}1})|^2 \, {\rm d}x \le \int_{\Omega'} |\nabla v|^2 \, {\rm d}x   +  \int_{J_v} h^-_{s_{i{-}1}}(\cdot,[v],\nu_v)   \, {\rm d}\mathcal{H}^{d{-}1}   ,   
$$
where $h^-_{s_{i{-}1}}$ denotes the density  in \eqref{1708240957}, for  the sequence of functionals in \eqref{1708240952} defined  by 
\begin{align}\label{elapart-newdefin}
\mathcal{H}_{n}^-(s_{i{-}1})(u, A)= \int_{  A \cap J_u \EEE}  \big(  \widehat{g}(\cdot,|[u]|) \EEE - \psi_{ n \EEE}(s_{i{-}1})\big)^+   \,  \dhn \quad \text{ \ZZZZ if $u|_A \in \PC(A)$.}
 \end{align}
Noting that  $h^-_{s_{i{-}1}}$  satisfies  $h^-_{s_{i{-}1}}(x, \xi, \nu) \le  g(|\xi|)$ \BBB for $\mathcal{H}^{d{-}1}$-a.e.\ $x \in \Omega \cup \partial_D\Omega$  \ZZZZ and all $\xi \in \R$, $\nu \in \mathbb{S}^{d{-}1}$ \EEE  by Remark \ref{rem: sigma1}(i),   and taking    $ v_i= u(s_i)-w(s_{i})+w(s_{i{-}1})  \in \SBV^2_{w(s_{i{-}1})}(\Omega')$ as an admissible competitor at time $s_{i{-}1}$ we get  
 \begin{align}\label{for diffpart}
 \int_{\Omega'} |\nabla u(s_{i{-}1})|^2 \, {\rm d}x \le  F_{\rm el} + F_{\rm cr},
 \end{align}
where $F_{\rm el} := \int_{\Omega'} |\nabla v_i|^2 \, {\rm d}x$ and  
$$F_{\rm cr} :=    \int_{ J_{v_i} \cap U_{i{-}1}} h^-_{s_{i{-}1}} (\cdot,[v_i],\nu_{v_i})  \,  {\rm d} \mathcal{H}^{d{-}1} + \int_{ J_{v_i} \sm U_{i{-}1}} g(|[v_i]|) \, {\rm d}\mathcal{H}^{d{-}1}.  $$ 
Here, in the spirit of the comment below \eqref{1910251906}, the first term in $F_{\rm cr}$ corresponds to an additional crack, whereas the second one \BBB corresponds \EEE to a new crack. 
 We now estimate the terms on the right-hand side of \eqref{for diffpart} separately. 
 First, for the elastic part  we get 
\begin{equation}\label{elapart}
\begin{split}
F_{\rm el}  &= \int_{\Omega'} |\nabla u(s_i)|^2 \, {\rm d}x - 2 \langle\nabla u(s_i) , \nabla w(s_{i})-\nabla w(s_{i{-}1}) \rangle    + \int_{\Omega'} | \nabla w(s_{i})- \nabla w(s_{i{-}1})|^2    \dx .
\end{split}
\end{equation}
  For the surface part $F_{\rm cr}$, we treat the two contributions separately. For the part inside of $ U_{i{-}1}$, we additionally introduce the  measures
\begin{align}\label{muhi}
  \muim  :=  h^-_{s_{i{-}1}} (\cdot,[u(s_i)],\nu_{u(s_i)}) \hn\mres_{J_{u(s_i)} \cap U_{i{-}1}} \quad \text{for all $i \in  \{1,\ldots, k\}$}.   
\end{align}
Note that $\muim(\Omega')$ corresponds to the additional crack energy in terms of the weak stability notion. In general, it may differ from the additional crack energy $\mu^i_{\rm add}(\Omega')$, see \eqref{defmutilde2}, related to the limiting evolution. However, as a key step of the proof, we will show below that they coincide   up to small errors in  terms of $\eps$. We \EEE recall that 
$u_\n(s_i) \wto u(s_i)$ in $\SBV^2(\Omega')$.   For any $A \in \mathcal{A}(\Omega')$, \EEE choose $V_{i{-}1} \subset \subset  A \cap \EEE U_{i{-}1}$. \EEE 
 By \EEE the  $\Gamma$-convergence of $ \H_n(s_{i{-}1})\EEE(\cdot, \ZZZZ V_{i{-}1} \EEE )$  and since \EEE
 $v_i^{\n} := u_\n(s_i)-w(s_{i})+w(s_{i{-}1})\wto v_i$ in $\SBV^2(\Omega')$ as $\n\to \infty$, 
    we get by   \eqref{defmutilde2},    \eqref{1910252147},   \eqref{elapart-newdefin}, \ZZZZ and $[v_i] = [u(s_i)]$ on  $J_{v_i} = J_{u(s_i)}$ \EEE  
 \begin{align}\label{localozation}
\int_{ J_{\ZZZZ {u}(s_i)} \cap    \ZZZZ V_{i{-}1} \EEE } h^-_{s_{i{-}1}} (\cdot, [u(s_i)]\EEE, \nu_{u(s_i)} )   \, \dhn &\le  \liminf_{\n\to \infty}  \mathcal{H}^-_\n(s_{i{-}1})\EEE(v_i^{\n}, \ZZZZ V_{i{-}1} \EEE )  \leq \lim_{\n\to \infty}(\mu^i_\n-\mu^{i{-}1}_\n)(\ZZZZ V_{i{-}1} \EEE ) \notag \\ & \ZZZZ \le  {\mu}_{\rm add, conv}^i\big(\overline{V_{i{-}1}}\big) \le \EEE {\mu}_{\rm add, conv}^i( A \EEE). 
 \end{align}
Here,  in the \ZZZZ second \EEE inequality we used  the fact that  $g(|[v_i^\n(s_i)]|) = g(|[{u}_\n(s_i)]|) \le \psi_\n(s_i)$ (in particular $J_{v_i^\n}  \subset K_\n(s_i) \ZZZZ \subset \Omega \cup \partial_D\Omega \EEE $) by \eqref{Tcrack}, \ZZZZ as well as \eqref{widehtatg}.  By the arbitrariness of $V_{i{-}1}$  and $A \in \mathcal{A}(\Omega')$ \EEE we get 
 \begin{align}\label{for diffpart1}
  \muim  \le  {\mu}_{\rm add, conv}^i \quad \text{in }\M_b^+(\Omega'). \EEE 
 \end{align}
\EEE We now come to the second term in $F_{\rm cr}$, i.e., the part outside of $U_{i{-}1}$ corresponding to the new crack.  
By  \BBB $[v_i] = [u(s_i)]$ on  $J_{v_i} = J_{u(s_i)}$, \EEE \eqref{13050959},   \eqref{defmu}, \eqref{defmutilde},  
 and  $\mu^i(\Omega' \setminus U_i)  < \EEE \frac{\eps}{k}$  (see \eqref{uu}) we get, \ZZZZ for $i \in \lbrace0,\ldots,k \rbrace$, \EEE
\begin{equation}\label{1910252254}
\begin{split}
\int_{J_{v_i} \sm  U_{i{-}1}}   g(|[{v_i}]|)   \, \dhn& =\int_{J_{v_i} \cap (U_i \sm  U_{i{-}1})}  g(|[{v_i}]|)  \, \dhn + \int_{J_{v_i}\sm U_i} g(|[{v_i}]|) \EEE \, \dhn\EEE \\&  \leq   {\mu}_{\rm new}^i(\Omega') \EEE + \mu^i(\Omega'\sm U_i) \le   {\mu}_{\rm new}^i(\Omega') \EEE + \frac{\varepsilon}{k}.  
\end{split} 
\end{equation}
\ZZZZ By using a similar localization argument as in the proof of \eqref{for diffpart1},  by applying     \eqref{lsc-app} for arbitrary open sets $U  \subset \subset  U_i \setminus  {U_{i{-}1}}$ and using   the third item in \eqref{uu}   and \eqref{13050959},  we get
\begin{align}\label{for diffpart1000}
 {\mu}_{\rm new}^i   \le   {\mu}_{\rm new, conv}^i \quad \text{in }\M_b^+(\Omega') \quad \text{for $i \in \lbrace0,\ldots,k \rbrace$}. 
 \end{align}
\EEE  Collecting \eqref{for diffpart}--\eqref{muhi}, \ZZZZ \eqref{1910252254},   and summing over $i\in\{1, \ZZZZ \ldots, \EEE  k\}$, we get by \eqref{Riemann2} \ZZZZ and  \eqref{Riemann3} \EEE  
\begin{equation}\label{1910252257}
\|\nabla u(0)\|_{\Ldue}^2 \leq \|\nabla u(\Ti)\|_{\Ldue}^2- 2   \int_0^\Ti  \langle \EEE \nabla  u(\tau), \partial_t  \nabla w(\tau) \rangle \EEE\, \mathrm{d}\tau   + \sum_{i=1}^k  \ZZZZ    \muim (\Omega') \EEE 
+ \sum_{i=1}^k  {\mu}_{\rm new}^i(\Omega') \EEE + C\eps.
\end{equation}
On the other hand, by \eqref{eq: lim-en},  \eqref{mmm0},  \eqref{1910251906}\BBB--\EEE\eqref{1910252147}, and \ZZZZ by \EEE 
 $u_\n(t)\wto u(t)$ in $\SBV^2(\Omega')$ as $ \n\to \infty$,
 we get 
\begin{equation}\label{eq: for en ba2}
\liminf_{n\to \infty} \EEE \mathcal{E}_{n}(u_n(t),K_n(t),\psi_n(t))  \ge  \|\nabla u(\Ti)\|_{\Ldue}^2 {+}      {\mu}_{\rm conv}^k(\Omega')   \ge  \|\nabla u(\Ti)\|_{\Ldue}^2 {+}    \sum_{i=1}^k  {\mu}_{\rm add, conv}^i(\Omega')  {+} \sum_{i=0}^k {\mu}_{\rm new, conv}^i(\Omega').  
\end{equation}
\BBB In view of \eqref{attimezero}, \ZZZZ the comment thereafter, \EEE and the fact that $u(0) = u^0$, \EEE we also observe that  
\begin{align}\label{auno}
\mathcal{E}(u(0), K(0),\psi(0)) = \|\nabla u(0)\|_{\Ldue}^2 +  \int_{J_{u(0)}}g(|[u(0)]|)\,\dhn \le  \|\nabla u(0)\|_{\Ldue}^2 +   {\mu}^0_{\rm new}(\Omega') \EEE +  \frac{\eps}{k}
\end{align}
since $\int_{J_{u(0)}}g(|[u(0)]|)\,\dhn =  {\mu}^0(\Omega')=  {\mu}^0_{\rm new}(\Omega') \EEE +  {\mu}^0(\Omega' \setminus U_0) \le  {\mu}^0_{\rm new}(\Omega') \EEE +  \frac{\eps}{k}$ by \eqref{uu}. Now, defining the \ZZZZ \emph{defect measures}  \EEE
\begin{align}\label{newepsbound0}
\delta_k^0 :=    {\mu}_{\rm new, conv}^0   -  {\mu}_{\rm new}^0, \quad \quad   \EEE 
\delta_k^i :=  \big( {\mu}_{\rm new, conv}^i   -  {\mu}_{\rm new}^i  \EEE \big) + \big(  {\mu}_{\rm add, conv}^i  -   \ZZZZ    \muim   \EEE  \big) \quad \text{for $i \in \EEE \{1,\ldots,k\}$}, 
\end{align}
and writing $\mathcal{F} = 2 \int_0^\Ti   \langle \EEE \nabla u(\tau), \partial_{t} \nabla  w(\tau) \rangle \EEE \, {\rm d}\tau$ for shorthand, we combine  \eqref{eq: for en ba}, \eqref{1910252257}--\eqref{auno} to get the auxiliary energy balance 
\begin{align}\label{eq: auxiliary balance}
\mathcal{E}(u(0), K(0),\psi(0)) + \mathcal{F}  & \ge  
   \limsup_{n\to \infty} \mathcal{E}(u_n(\Ti),K_n(\Ti),\psi_n(\Ti)) \ge  \liminf_{n\to \infty} \mathcal{E}(u_n(\Ti),K_n(\Ti),\psi_n(\Ti)) \notag \\
& \ge   \|\nabla u(\Ti)\|_{\Ldue}^2 +      {\mu}_{\rm conv}^k(\Omega')   \ge  \mathcal{E}(u(0), K(0),\psi(0))  + \mathcal{F} + \sum_{i=0}^k \delta_k^i \ZZZZ (\Omega') \EEE - C\eps.
  \end{align}
From \eqref{for diffpart1} and   \eqref{for diffpart1000},   and by comparing the leftmost and rightmost  expression of  \eqref{eq: auxiliary balance}, we obtain \EEE
\begin{align}\label{newepsbound}
 \delta_k^i \in \M_b^+(\Omega') \EEE \ \text{for } i  \in \EEE \{0, \ldots, k\}, \quad \quad \quad   \sum_{i=0}^k \delta_k^i \ZZZZ (\Omega') \EEE  \le C\eps.
\end{align}
Note that \eqref{eq: auxiliary balance} for $\eps \to 0$ does not yet yield the desired energy balance.  This will be achieved by showing that also 
\EEE    \begin{align}\label{asmotivated} 
 \big|{\mu}_{\rm conv}^k(\Omega') - \BBB \int_{K(t)} \psi(t) \, {\rm d}\mathcal{H}^{d{-}1} \EEE \big| \ZZZZ \le  |{\mu}_{\rm conv}^k(\Omega') - {\mu}^k(\Omega') |   \EEE  \le C\eps
\end{align}
 holds. \ZZZZ Our idea will consist in showing that    $ \muim \EEE(\Omega')$ defined in  \eqref{muhi}    can be replaced, up to small errors, by  $\mu^i_{\rm add}(\Omega')$. \EEE  Then, \eqref{asmotivated}  will follow by combining  (the analogue without $n$ of) \EEE \eqref{1910251906} with \eqref{newepsbound}. The fact that the integral can be replaced by    $\mu^i_{\rm add}(\Omega')$ \EEE
  relies on a delicate property related to the local structure of crack sets, \ZZZZ see \eqref{for subsectionlast} \EEE and Subsection \ref{subsec3} below.   \EEE  As a preparation, we need to localize the problem on small cubes around the crack set.

\emph{Step 6:  Covering the crack set with cubes.} 
 For any $p\in \{0,\dots, k\}$, by a Besicovitch covering argument we can choose pairwise disjoint  open cubes $(Q_{p,j})_{j=1}^{N_p} \subset \subset U_p \sm \overline{U_{p-1}}$ with centers  $x_{p,j}\in K(s_{p})$,   sidelengths \EEE $r_{p,j}>0$, and two sides orthogonal to the normal vector $\nu_{p,j}:=\nu_{u(s_p)}(x_{p,j})$, such that  $\mu^k_{\rm conv}(\bigcup_{p,j}\partial Q_{p,j})= \ZZZZ \mu^k (\bigcup_{p,j}\partial Q_{p,j}) = \EEE 0$,    for all $p \le i \le k$ \ZZZZ it holds that (recall notation in Subsection~\ref{sec:slicing})\EEE
  \begin{align}\label{2prop}
 {\rm (i)} & \ \   \mu^i(V_p) +  \muim \EEE(V_p) \EEE  \le \BBB 2 \EEE \frac{\eps}{k} \text{ for } V_p:=(U_p\sm \ZZZZ { U_{p-1}}) \EEE \sm \bigcup\nolimits_{j=1}^{N_p} Q_{p,j},   \notag \\
 {\rm (ii)} & \ \ \text{there are sets } D_{p,j} \subset Q_{p,j} \text{ and }  \Gamma_{p,j}=\partial D_{p,j} \cap Q_{p,j} \text{ such that } r_{p,j}^{d{-}1} \le \mathcal{H}^{d{-}1}( \Gamma_{p,j}) \le  r_{p,j}^{d{-}1}  +  \frac{\eps}{k} r_{p,j}^{d{-}1} \text{ and }   \notag  \\ \notag
 & \quad \quad \quad    \mathcal{H}^{d{-}1}\big( (K(s_p) \cap Q_{p,j}  ) \triangle \Gamma_{p,j}\big)  \le   \frac{\eps}{k}r_{p,j}^{d{-}1}, \quad  \BBB \text{$\# (\Gamma_{p,j})^{\nu}_y = \EEE 1$ for $\hn$-a.e.\ $y\in P^\nu(Q_{p,j})$}, \EEE \\
 {\rm (iii)} & \ \  \BBB \text{if $Q_{p,j} \not\subset \Omega$, then $x_{p,j} \in  \ZZZZ \partial_D \EEE \Omega$, $\partial \Omega$ is differentiable at $x_{p,j}$, and $\Gamma_{p,j} =  \partial_D \Omega \cap Q_{p,j}$,} 
  \end{align}
 and such that, setting 
  $\ZZZZ \etamac{i} \EEE :=  g^{-1}( \ol \psi_i(x_{p,j}))= \max_{0 \leq l\leq i} |[u(s_l, x_{p,j})]|$   (see \BBB definition below \eqref{13050959}),  
we have, again for all  $p \le i \le k$,   
\begin{align}\label{local-h1}  
  \int_{\Gamma_{p,j}} h^-_{s_{i}} (\cdot,\etamac{i},\nu^{\Gamma_{p,j}})   \, \dhn & \le   \frac{\eps}{k}r_{p,j}^{d{-}1}
  \end{align}
  as well as,  for all $p \le i \le k{-}1$, 
\begin{align}\label{local-h2}   
 \mu^{i{+}1}_-(Q_{p,j}) \EEE &\le      \int_{\Gamma_{p,j}} h^-_{s_{i}} (\cdot,\etamac{i{+}1},\nu^{\Gamma_{p,j}})   \, \dhn +  \frac{\eps}{k}r_{p,j}^{d{-}1},
  \end{align}
  and 
\begin{align}\label{local-h3}   
\Bigg| \mathcal{H}^{d{-}1}(Q_{p,j}  \cap  K(s_{i{+}1}) ) g( \etamac{i{+}1} )   &-    \int_{Q_{p,j}  \cap  K(s_{i{+}1}) }  \ol \psi_{i{+}1} \EEE  \, \dhn \Bigg| \EEE \leq   \frac{\eps}{k}r_{p,j}^{d{-}1}.
  \end{align}  \EEE
\BBB Here, \EEE $\nu^{\Gamma_{p,j}}$ denotes a unit normal to ${\Gamma_{p,j}}$. 

Indeed,  property \eqref{2prop}(i) can be achieved by a Besicovitch covering for the measure 
  $\mu^k$    taking also   $\mu^k(U_p\sm K(s_p)) < \frac{\eps}{k}$ into account, see \eqref{uu}, \BBB as well as  $\mu^i \le \mu^k$ and \EEE $ \muim \EEE  \le \mu^i$, by Remark \ref{rem: sigma1}(i), \eqref{13050959},  \eqref{defmu}, and   \eqref{muhi}.   The condition   $\mu^k_{\rm conv}(\bigcup_{p,j}\partial Q_{p,j}) =  \ZZZZ \mu^k (\bigcup_{p,j}\partial Q_{p,j}) = \EEE 0$   can be obtained by outer regularity by choosing $r_{p,j}$ suitably.   Property \eqref{2prop}(ii) follows from the rectifiability of  $K(s_p)$.  \BBB Property (iii) follows from  \ZZZZ  $K(s_p) \subset \Omega \cup \partial_D \Omega$, \EEE the fact that $\partial_D\Omega$ is differentiable $\mathcal{H}^{d{-}1}$-a.e., and $\nu_{K(s_p)} = \nu_\Omega$ $\mathcal{H}^{d{-}1}$-a.e.\ on $K(s_p) \cap \partial_D\Omega$.

The estimates in \eqref{local-h1}--\eqref{local-h3}  can be achieved by choosing the centers  $x_{p,j}$ as Lebesgue points of the mappings $h^-_{s_{i}} (\cdot,[u(s_{i{+}1})], \EEE \ZZZZ \nu_{K(s_{i{+}1})}) \EEE $   \BBB on $K(s_{i{+}1})$ for $ p \le i \le k{-}1$ \EEE (for \eqref{local-h2}) \EEE and 
 $\overline{\psi}_l$ on $K(s_l)$ \BBB for $p \le l \le k$ \EEE (for \eqref{local-h3} and \eqref{local-h1}), along with \BBB some additional properties, \ZZZZ in particular using \eqref{2prop}(ii): \EEE for \eqref{local-h1} we use  \EEE   the fact that $h^-_{s_{i}} (\cdot,g^{-1}(\psi(s_{l})), \ZZZZ \nu_{K(s_l)}) \EEE = 0$ on $K(s_l)$  
for $l \in \{0,\ldots,i\}$.  For \eqref{local-h2}, we further use \EEE  the monotonicity of $h^-_{s_{i}}$ in the second variable \BBB (see Remark \ref{rem: sigma1}(ii)) and \EEE additionally we take into account that $ K(s_{i{+}1}) \setminus  K(s_{p})$ has density $0$ for $\mathcal{H}^{d{-}1}$-a.e.\ $x \in K(s_{p})$, i.e., $\mathcal{H}^{d{-}1}((Q_{p,j}  \cap  K(s_{i{+}1})) \setminus \Gamma_{p,j})$ can be taken of order $\frac{\eps}{k}r_{p,j}^{d{-}1}$.

  Since $(Q_{p,j})_{p,j}$ are pairwise disjoint, \EEE from Corollary \ref{cor: energy bound} with   \eqref{2prop}(ii), and by   \eqref{newepsbound}   it follows that \EEE 
\begin{align}\label{sumr}
 \sum_{p=0}^k  \sum_{j=1}^{N_p} r_{p,j}^{d{-}1} \le C\mathcal{H}^{d{-}1}(K(t)) \le   C, \qquad \sum_{p=0}^k  \sum_{j=1}^{N_p}\sum_{i =p}^k \delta^{i}_k(Q_{p,j}) \le C\eps.  \EEE
\end{align}
Below in Subsection \ref{subsec3} we will show that for all $i \in \{p{+}1, \ldots, k\}$ \BBB and each $Q_{p,j}$, $j\in\{1,\ldots, N_p\}$, 
\begin{align}\label{for subsectionlast}
   \muim \EEE( Q_{p,j})  \EEE   & \le   {\mu}^i(Q_{p,j}) - {\mu}^{i{-}1}(Q_{p,j}) \EEE + \BBB C\frac{\eps}{k}r_{p,j}^{d{-}1} \EEE
     +    C  \Big(  \varepsilon r_{p,j}^{d{-}1} +\EEE \sum_{l=p}^k \delta^{l}_k(Q_{p,j}) \Big)  \, \sigmamac{i} 
    \end{align} 
holds for some universal constant $C>0$, where $\ZZZZ \sigmamac{i}  \EEE := \etamac{i}-\etamac{i{-}1}$ \ZZZZ for all $i \in \{p{+}1, \ldots, k\}$. \EEE  
 
 \emph{Step 7: Energy balance and energy convergence.} We are now in a position to show  \eqref{asmotivated}   which will allow us to deduce the energy balance by using  \eqref{eq: auxiliary balance}--\eqref{newepsbound}. First, by \eqref{1910251906},  \eqref{newepsbound0}, \eqref{newepsbound}, and \eqref{2prop}(i),    we get
\begin{align}\label{plugger}
 {\mu}_{\rm conv}^k(\Omega')  & =  \sum_{i=0}^k {\mu}_{\rm new, conv}^i(\Omega') + \EEE \sum_{i=1}^k  {\mu}_{\rm add, conv}^i(\Omega')   \le  \sum_{p=0}^k {\mu}_{\rm new}^p(\Omega') \EEE + \sum_{i=1}^k
 \ZZZZ \muim  (\Omega') \EEE  + \sum_{i=0}^k \delta_k^i \ZZZZ (\Omega') \EEE \\
 & \ZZZZ  =  \sum_{p=0}^k {\mu}^p(U_p \sm U_{p-1})  + \sum_{i=1}^k
  \muim  (U_{i{-}1})    + \sum_{i=0}^k \delta_k^i  (\Omega') \EEE    \le   \sum_{p=0}^k    \sum_{j=1}^{N_p}  \Big( {\mu}^p(Q_{p,j}) \EEE +   \sum_{i =p{+}1}^k   \muim \EEE( Q_{p,j}) \Big) \EEE \notag 
   +  C\eps,
 \end{align}
where  \ZZZZ in the third step we also employed the definitions in  \eqref{defmutilde} and \eqref{muhi}.  \EEE  We now use the key estimate  \eqref{for subsectionlast}. Note that by the definition of  $\etamac{i}$ and Lemma~\ref{prop-control-on-ela} we get $\ZZZZ \sum_{i =p+1}^k \EEE \sigmamac{i}  \le  \etamac{k} \EEE  \le  C_w$. \EEE Then, using \eqref{for subsectionlast} and taking the sum over $i \in \{p{+}1, \ldots, k\}$ we get 
\begin{align*}
\sum_{i = p{+}1}^{k}   \muim \EEE( Q_{p,j})  \EEE   & \le  {\mu}^k(Q_{p,j}) - {\mu}^{p}(Q_{p,j}) \EEE  +  C\eps r_{p,j}^{d{-}1} +    C \sum_{l=p}^k   \delta^{l}_k (Q_{p,j}). 
    \end{align*} 
Plugging this into \eqref{plugger} \BBB and \EEE   using  
 \eqref{sumr} \EEE     we get 
\begin{align*}
 {\mu}_{\rm conv}^k(\Omega')  \le    \sum_{p=0}^k    \sum_{j=1}^{N_p}  {\mu}^k(Q_{p,j}) \EEE +    C\eps\sum_{p=0}^k    \sum_{j=1}^{N_p}     r_{p,j}^{d{-}1}   +    C  \sum_{p=0}^k    \sum_{j=1}^{N_p}  \sum_{l=p}^k \BBB \delta^{l}_k \EEE (Q_{p,j})    +  C\eps  \le  {\mu}^k(\Omega') \EEE  +  C\eps.
 \end{align*}
 By \eqref{lsc-app} \ZZZZ and a localization argument as in \eqref{localozation}, \EEE  we also have $ {\mu}_{\rm conv}^k(\Omega')  \ge \BBB \int_{K(t)} \psi(t) \, {\rm d}\mathcal{H}^{d{-}1} $. \BBB Using that $ {\mu}^k(\Omega') \le \int_{K(t)} \psi(t) \, {\rm d}\mathcal{H}^{d{-}1} $ by \eqref{13050959} and \EEE \eqref{defmu}, we deduce   \eqref{asmotivated} \ZZZZ  and, as a byproduct,    $|\int_{K(t)} \psi(t) \, {\rm d}\mathcal{H}^{d{-}1} -    {\mu}^k(\Omega')| \le C\eps$. Then, \EEE using that 
 $${\mu}^k(\Omega') = \int_{\bigcup_{0 \le i \le  k \EEE} J_{
 u(s_i)}} \max_{0\leq i\leq  k \EEE} g(|[u(s_i)]|) \, {\rm d}\mathcal{H}^{d{-}1}$$ by \eqref{defmu} and $g(0)=0$, 
 we obtain \eqref{17050743}   for $t \in [0,T]$ \ZZZZ   by the arbitrariness of $\eps$. \EEE Combining \BBB \eqref{asmotivated} \EEE with \eqref{eq: auxiliary balance} and \eqref{newepsbound}, and sending $\eps \to 0$ we get 
\begin{equation*}
 \lim_{n\to \infty} \mathcal{E}(u_n(t),K_n(t),\psi_n(t)) =  \mathcal{E}(u(t), K(t),\psi(t)) =  \mathcal{E}(u(0), K(0),\psi(0))+ 2 \int_0^t   \langle \EEE \nabla u(\tau) , \partial_{t} \nabla  \ZZZZ w \EEE (\tau) \rangle \EEE \, {\rm d}\tau.
 \end{equation*}
\ZZZZ As this property holds \EEE  for every $t\in[0,T]$, we have  shown property (d) in Theorem \ref{main def}. As a byproduct we have also obtained  $\lim_{n\to \infty} \mathcal{E}(u_n(t),K_n(t),\psi_n(t) )  = \mathcal{E}(u(t),K(t),\psi(t) )$  for all $t\in[0,T]$. Thus,  by the separate lower semicontinuity of elastic and surface energy (use \eqref{lsc-app} \ZZZZ and \eqref{independentoft}), \EEE we find  
\begin{equation*}\label{surfconv'}
\lim_{n \to \infty} \int_{K_n(t)} \psi_n(t) \, {\rm d}\mathcal{H}^{d{-}1}  =  \int_{K(t)} \psi(t)  \, {\rm d}\mathcal{H}^{d{-}1} \EEE  \quad \text{for all $t \in [0,T]$.}
\end{equation*}
From  \ZZZZ  \eqref{lsc2-app} \EEE and \eqref{lsc-neu22} (for $\lambda =0$) we then also get 
\begin{align}\label{surfconv}
\lim_{n \to \infty}\Hd(K_n(t))  = \Hd( K\EEE(t))  \quad \text{for all $t \in [0,T]$.}
\end{align}

\emph{Step 8: Stability  \eqref{finalstability}.} 
 Given the   convergence \eqref{surfconv}, we are finally in the position to improve the weak unilateral minimality to the strong version.  For each $x \in K(t)$ and $r >0$, with $\Hd(K(t) \cap \partial B_r(x)) =0$, it holds that  $\Hd(K_n(t) \cap  B_r(x)) \to \Hd(K(t) \cap  B_r(x))$      by \eqref{lsc2-app} and \eqref{surfconv}. Defining  the measures $\lambda_{t,n} = \mathcal{H}^{d{-}1}\mres_{K_n(t)}$, for a suitable subsequence depending on $t$ we can pass to a weak$^*$ limit $\lambda_t \in  \mathcal{M}^+_{\rm b}(\Omega')$. \EEE As $\lambda_t (B_r(x)) =  \Hd(K(t) \cap  B_r(x))$, using the  Radon–Nikod\'ym theorem  and   the rectifiability of $K(t)$, we get 
$$
  \frac{{\rm d}\lambda_t}{{\rm d}\mathcal{H}^{d{-}1}\mres_{K(t)}} = 1
$$
for $\Hd$-a.e.\ $x \in K(t)$. We recall from Step 1  \ZZZZ and the comment below \eqref{Riemann3} that there exists a subsequence $ (u_{ m \EEE}(t))_{ m \EEE} \subset (u_{ n}(t))_n$ such that \EEE   
 $u_{ m \EEE }(t)\wto u(t)$  in $\SBV^2(\Omega')$ 
  as $ m \EEE \to \infty$.  As  $(u_n(t),K_n(t),\psi_n(t))$ are strong unilateral minimizers for each $n \in \N$,  by  Corollary \ref{cor. unilat} we get that  $(u(t),K(t),\psi(t))$   is a strong unilateral minimizer, i.e.,  for all functions $v \in \SBV^p_{w(t)}(\Omega')$ it holds that
$$\int_{\Omega'} |\nabla u(t)|^2 \, {\rm d}x \le \int_{\Omega'} |\nabla v|^2 \, {\rm d}x + \int_{J_v \setminus K(t)} g(|[v]|) \, {\rm d}\mathcal{H}^{d{-}1} +  \int_{J_v \cap K(t)} \big( g(|[v]|) - \psi(t)\big)^+  \, {\rm d}\mathcal{H}^{d{-}1}.    $$
In particular, this implies that for  every  $(H,\zeta)$ with $K(t) \subset H $ and  $\psi(t) \le \zeta$ on $K(t)$,  and for every $v\in AD( \BBB w(t), \EEE H,\zeta)$  it holds that 
\begin{equation*}\label{2108241743-new}
\int_{\Omega'} |\nabla u(t)|^2 \, {\rm d}x \le \int_{\Omega'} |\nabla v|^2 \, {\rm d}x + \int_{H \setminus K(t)} \BBB \zeta \EEE  \, {\rm d}\mathcal{H}^{d{-}1} +  \int_{  K(t)}  (  \BBB \zeta \EEE - \psi(t) )  \, {\rm d}\mathcal{H}^{d{-}1}    .
    \end{equation*}
 Adding $\int_{K(t)} \psi(t)  \, {\rm d}\Hd$ on both sides and recalling \eqref{eq: lim-en-an}, we  \BBB get \EEE property (c) in Theorem~\ref{main def}. This finally concludes the  proof.
 \end{proof}

  \subsection{Local structure of crack sets}\label{subsec3}
 
 This subsection is devoted to the proof of \eqref{for subsectionlast}.
 
 \begin{proposition}\label{lastsecpro}
 In the setting of the proof of Theorem \ref{main def}, it holds that 
  \begin{equation*}
  \begin{split}
 \muim(Q_{p,j}) \EEE   & \le  
     {\mu}^i(Q_{p,j}) - {\mu}^{i{-}1}(Q_{p,j}) \EEE + \BBB C\frac{\eps}{k}r_{p,j}^{d{-}1} \EEE
     +    C  \Big(  \varepsilon r_{p,j}^{d{-}1} +\EEE \sum_{l=p}^k \delta^{l}_k(Q_{p,j}) \Big)  \, \sigmamac{i}  
     \end{split}
    \end{equation*} 
 for all $p\in\{0,\ldots,k\}$, $j\in\{1,\ldots, N_p\}$, and $i \in \lbrace p +1,\ldots,k\rbrace$. 
   \end{proposition} 
 



 \begin{proof}[Proof of Proposition \ref{lastsecpro}]
We fix a cube $Q_{p,j}$  for some $p \in \lbrace 0,\ldots,k \rbrace$ and $j  \in \lbrace 1 , \ldots, N_p \rbrace$.  For notational convenience, we will drop all indices $(p,j)$ in the sequel. In particular, we write $Q$, $r$, $\nu$, $\eta_i$, $\sigma_i$, $D$,  and  $\Gamma$  in place of $Q_{p,j}$, $r_{p,j}$, $\nu_{p,j}$, $\etamac{i}$, $\sigmamac{i}$,  $D_{p,j}$,  and $\Gamma_{p,j}$.  Moreover, we write $\delta^i$ in place of $\delta_k^i$.  \ZZZZ Recalling \ZZZZ $Q \subset \subset U_p \sm \overline{U_{p-1}}$, \EEE \eqref{defmutilde}--\eqref{defmutilde2}, \eqref{muhi}, and \eqref{newepsbound0}, we first observe 
\begin{align}\label{countmist}
\delta^{p}(Q) =  {\mu}_{\rm  conv}^p(Q)  -  {\mu}^p(Q), \quad \   \  
\delta^{i{+}1}(Q) = \mu^{i{+}1}_{\rm conv}(Q) - \mu^{i}_{\rm conv}(Q) - \mu_-^{i{+}1}(Q) \  \  \text{ for     $p \le i \le k{-}1$}.
\end{align} \EEE 
 Recalling the notation introduced in Subsection \ref{sec:slicing}, we \EEE let   \ZZZZ
\begin{equation}\label{defPi}
 \Pi_{i}  :=  \big\{ y  \in P^{\nu}(Q)  \colon \liminf\nolimits_{n\to +\infty} \#   ( (K_n(s_i) \cap Q)^{\nu}_y   )   \ge 2  \big\}, \quad \EEE  B_{i}   :=   \big( \Pi_{i} +  \R \nu \big) \cap Q, 
\end{equation}
corresponding to the slices which asymptotically hit the crack at least twice. 

Let us formulate the main estimates leading to the proof of the statement.  
For $i \in \lbrace p, \ldots, k{-}1 \rbrace$ 
we will show \ZZZZ that \EEE
\begin{subequations}\label{firtmain}
\begin{align}\label{firtmain1}
  \mu^{i{+}1}_-(Q) \EEE   \le    \liminf_{n \to \infty}  \int_{ \Gamma_n  \cap (Q \setminus B_i) } \big(g( \eta_{i{+}1} ) -   g(\eta_i) \big) \,{\rm d} \mathcal{H}^{d{-}1}  + C\frac{\eps}{k}r^{d{-}1} +    C \big(\delta^{p}(Q)  + \mathcal{H}^{d{-}1}(\Pi_i) \big) \sigma_{i{+}1}, \quad
\end{align}
where $\sigma_{i{+}1} = \eta_{i{+}1} - \eta_i$, and  $\Gamma_n  \subset Q$ is a suitable set such that 
\begin{equation}\label{0305261104}
\limsup_{n \to \infty}\mathcal{H}^{d{-}1}(\Gamma_n ) \le  r^{d{-}1} + C\frac{\eps}{k}r^{d{-}1}  + C\delta^{p}(Q).
\end{equation} 
\end{subequations}
\BBB The underlying idea lies in defining a competitor whose construction is related to the jump-transfer lemma of \cite{Francfort-Larsen:2003} (see Steps 1--2).  In particular, we will exploit that \EEE the slices outside of $B_i$ hit the crack \ZZZZ at most once \EEE and therefore, \ZZZZ if they hit the crack, \EEE the corresponding jump height at time $s_{i{+}1}$ needs to be approximately $\eta_{i{+}1}$. A similar conclusion is not possible on the slices in $B_i$, leading to the error term $\mathcal{H}^{d{-}1}(\Pi_i)$ on the right-hand side of  \eqref{firtmain1}. \EEE Therefore, a second fundamental ingredient lies in the derivation of the bound   
\begin{align}\label{projection bound}
\mathcal{H}^{d{-}1}(\Pi_k) \le  C \eps r^{d{-}1} + C\sum_{i=p}^k \delta^{i}(Q)
\end{align}
  noticing that $\Pi_i \subset \Pi_k$ for $i \in \{p,\dots, k{-}1\}$ as $K_n(s_i) \subset K_n(s_{k})$.  
The idea consists in deriving a suitable lower bound for  $\mu^{i{+}1}_{\rm conv}(Q) - \mu_{\rm conv}^{i}(Q)$ involving  $\mathcal{H}^{d{-}1}(\Pi_{i{+}1} \setminus \Pi_{i})$, and using \eqref{countmist} (see Steps 3--4). \EEE 

Indeed, the statement follows herefrom: combining \eqref{firtmain}, \eqref{projection bound} with  \eqref{2prop}(ii) (with $K(s_p) \BBB \subset K(s_{i})  \EEE \subset K(s_{i{+}1})$), \EEE \eqref{local-h3},  and recalling that $g(\eta_l)=\ol\psi_l(x_{p,j})$ for $p\leq l\leq k$   we get
\begin{align*}
   \mu^{i{+}1}_-(Q) \EEE  &  \le    \int_{Q  \cap  K(s_{i{+}1}) }   \ol\psi_{i{+}1} \EEE  \, \dhn -  \int_{Q   \cap  K(s_{i}) }  \ol \psi_{i} \EEE  \, \dhn + C \frac{\eps}{k}r^{d{-}1}  \notag \\ & \ \ \  \BBB +   C\Big(\frac{\eps}{k}r^{d{-}1}  + \delta^{p}(Q)\Big) \big(g( \eta_{i{+}1} ) -   g(\eta_i)  \big)  \EEE +   C  \Big(  \varepsilon r^{d{-}1} +\EEE \sum_{l=p}^k \delta^{l}(  Q \EEE ) \Big) \sigma_{i{+}1}.
    \end{align*} 
\BBB In view of \ZZZZ \eqref{defmu} and \EEE the fact that $g( \eta_{i{+}1} ) -   g(\eta_i) \le C_g(\eta_{i{+}1}-\eta_i ) = C_g \sigma_{i{+}1}$ by \eqref{gproperty2}, \EEE this implies the statement, where we note that for notational convenience (related to the induction proof employed below) the indices are shifted by $1$. \EEE 
We now come to the proof of \eqref{firtmain} and \eqref{projection bound}. \BBB Clearly, it suffices to establish the estimates on a subsequence of $n$ (not relabeled), which we choose in  such a way that \ZZZZ  \eqref{1910252147} and \EEE  the statement of Lemma \ref{recovboundra} hold  for  $ \ZZZZ \H^-_{n}(s_{i}) \EEE$  and $\ZZZZ \H^-(s_{i}) \EEE$ in place of $\H^-_{n}$ and $\H^-$, for any $p\leq i \leq k$. (We notice that actually we consider a subsequence of $m$, which may depend on $t$, but for notational convenience we still denote it by $n$.)

\BBB 
    \emph{Step 1: Proof of \eqref{0305261104}.} \EEE 
 Let us introduce a first \emph{standard competitor}, namely
\begin{align}\label{vs}
\bar{v}(s_i)  :=\eta_i \chi_{D} \quad \text{ on $Q$  \quad  for $p \le i \le k$},
\end{align}
where $D$ is given in \eqref{2prop}(ii) with $\Gamma=\partial D \cap Q$, i.e.,  $J_{\bar{v}(s_i)} = \Gamma$.   The goal of  this step is to identify a sequence of hypersurfaces $(\Gamma_n)_n$ converging suitably to  $\Gamma$. In later steps, these hypersurfaces will be used to build competitors for the subsequent time step. 
Let $({v}_{n}(s_{p}))_n \subset \PC(Q)$   be a recovery sequence of  $\bar{v}(s_{p})$ for $\ZZZZ \H^-_{n}(s_{p}) \EEE(\cdot, Q)$ (recall \eqref{elapart-newdefin} and \eqref{1708240957}), i.e., ${v}_{n}(s_{p}) \to \bar{v}(s_{p})$   in   $L^1(Q)$ and 
$${\limsup_{n \to \infty}\ZZZZ \H^-_{n}(s_{p}) \EEE({v}_{n}(s_{p}), Q) \le  \int_{\Gamma} h^-_{s_{p}} \big(\cdot,\eta_p,\nu^{\Gamma}\big)   \, \dhn \le \frac{\eps}{k}r^{d{-}1},}$$
where in the last step we used \eqref{local-h1}. By \BBB Lemma~\ref{recovboundra} \EEE we may assume that   \EEE   $v_n(s_p)=\ol v(s_p)$ \ZZZZ  near \EEE $\partial Q$. \BBB In case that $Q \not\subset \Omega$, we have  $x \in \partial_D\Omega$ and  $\Gamma = \partial_D \Omega \ZZZZ \cap Q \EEE $ by \eqref{2prop}(iii), and we can suppose without restriction that $Q \cap \Omega = D \cap \Omega$. Thus, in view of \eqref{widehtatg}, possibly replacing  ${v}_{n}(s_{p})$ by ${v}_{n}(s_{p}) \chi_{D}$ (not relabeled), we may further assume that the recovery sequence $({v}_{n}(s_{p}))_n$ satisfies $J_{{v}_{n}(s_{p})} \subset \Omega \cup \partial_D \Omega$. \EEE      By the fact that  ${v}_{n}(s_{p})$ is piecewise constant and  $g(t)  \ge c_g$ for all $t >0$, see \eqref{gproperty},  we find sets of finite perimeter $D_n \subset Q$ such that $\Gamma_n := Q \cap \partial D_n$ satisfies   $\Gamma_n \subset J_{{v}_{n}(s_{p})}$ \ZZZZ and  \EEE
\begin{align}\label{firstgammaproperty} 
 \limsup_{n \to \infty}\mathcal{H}^{d{-}1}(\Gamma_n \setminus K_n(s_p)) \le  \limsup_{n \to \infty}\mathcal{H}^{d{-}1}(J_{{v}_{n}(s_{p})} \setminus K_n(s_p)) \le C\frac{\eps}{k}r^{d{-}1},
 \end{align}
 where $C>0$ only depends on $ c_g$, \BBB and we have \EEE  $\mathcal{L}^d(D_n \triangle D) \to 0$ as $n \to \infty$, with $D$ as in \eqref{vs}.   \ZZZZ For future reference, \EEE by \EEE   $v_n(s_p)=\ol v(s_p)$ on $\partial Q$, a slicing argument (recalling that $({v}_{n}(s_{p}))_n \subset \PC(Q)$), \ZZZZ and    \eqref{2prop}(ii),  \EEE     we \ZZZZ also \EEE  deduce that 
\begin{align}\label{bgammabound}
 \# (\Gamma_n)^{\nu}_y\geq 1 \ \ \ \text{ for $\hn$-a.e.\ $y\in P^\nu(Q)$ \ZZZZ for all $n \in \N$.}
 \end{align}   
From the first item in   \eqref{countmist}, \BBB $\overline{\psi}_p \le \psi(s_p)$  (see \eqref{13050959}), \EEE  ${\mu}_{\rm  conv}^p(\partial Q) = 0 $, \EEE and  Corollary~\ref{cor:lsc-neu22} (for the cube $Q$ in place of  $A$) \EEE  we get 
  \begin{equation}\label{0102261700}
  \limsup_{n \to \infty}\mathcal{H}^{d{-}1}(K_n(s_p) \cap Q) \le    \mathcal{H}^{d{-}1}(K(s_p) \cap Q)  + C\delta^{p}(Q).
  \end{equation}
  Thus,  combining \eqref{firstgammaproperty}, \eqref{0102261700}, and using also \eqref{2prop}(ii) \ZZZZ we calculate \EEE
\begin{align}\label{basiccracklength}
\limsup_{n \to \infty}\mathcal{H}^{d{-}1}(\Gamma_n ) \le  r^{d{-}1} + C\frac{\eps}{k}r^{d{-}1}  + C\delta^{p}(Q).
\end{align}

\BBB 
    \emph{Step 2: Proof of estimate \eqref{firtmain1}.} \EEE
  Fix $i \in \lbrace p,\ldots,k{-}1\rbrace$ and \EEE let $({v}_{n}(s_{i}))_n \subset \PC(Q)$   be a recovery sequence of  $\bar{v}(s_{i})$ (see \eqref{vs}) for $\ZZZZ \H^-_{n}(s_{i}) \EEE (\cdot, Q)$, i.e., ${v}_{n}(s_{i}) \to \bar{v}(s_{i})$   in   $L^1(Q)$ and 
 \begin{align}\label{nochmali}
 {\limsup_{n \to \infty}\ZZZZ \H^-_{n}(s_{i}) \EEE({v}_{n}(s_{i}), Q) \le  \int_{\Gamma  } h^-_{s_{i}} (\cdot,\eta_i,\nu^\Gamma)   \, \dhn \le \frac{\eps}{k}r^{d{-}1},}
 \end{align}
  where we again used \eqref{local-h1}. In particular, by \eqref{gproperty} this shows 
  \begin{equation}\label{0202262322}
  \limsup_{n \to \infty}\mathcal{H}^{d{-}1}(J_{{v}_{n}(s_{i})} \setminus K_n(s_i)) \le C\frac{\eps}{k}r^{d{-}1},
  \end{equation} see \eqref{firstgammaproperty}  for the same argument.  Similarly to Step~1, by \BBB Lemma~\ref{recovboundra} \EEE  we may assume that $v_n(s_i)=\ol v(s_i)$ on $\partial Q$, \BBB and that $J_{{v}_{n}(s_{i})} \subset \Omega \cup \partial_D \Omega$. \EEE Moreover, by truncation and monotonicity of $g$, we may assume that $v_n(s_i)$ takes values in $[0,\eta_i]$. \EEE  
 Define the sets 
\begin{align*}
  \Pi_i^{n}:=\Pi_i \cup \big\{y \in P^\nu(Q) \colon \#\big((J_{v_n(s_i)} \cup \Gamma_n) \sm K_n(s_i)\big)^\nu_y \ge 1\big\}, \quad B_i^{n}  := \EEE  (\Pi_i^{n} +  \R \nu) \cap Q.
\end{align*}
  Then, from \eqref{firstgammaproperty}, \ZZZZ \eqref{basiccracklength}, \EEE \eqref{0202262322}, the fact that $K_n(s_p)\subset K_n(s_i)$ for $i \ge p$,    the Area Formula, \ZZZZ and Lemma~\ref{le:projection} applied for $\Psi = \Gamma_n$ and $A = B_i^{n} \setminus  B_i$  \EEE  we get   
\begin{align}\label{similarly-neu} 
& \ZZZZ \limsup_{n \to \infty}\mathcal{H}^{d{-}1}\big(  \Pi_i^{n} \setminus  \Pi_i \big) \le      C \frac{\eps}{k}r^{d{-}1}, \EEE \\ 
\label{similarly}  & \limsup_{n \to \infty}\mathcal{H}^{d{-}1}\big(    (B_i^{n} \setminus  B_i) \cap \Gamma_n\big) \le  \ZZZZ   C \frac{\eps}{k}r^{d{-}1}  + C\delta^{p}(Q). \EEE 
\end{align}
\ZZZZ Recalling the definition in \eqref{defPi}  and \eqref{bgammabound}, we observe that   for any $y \in P^\nu(Q) \sm \Pi_i^n$ there is a unique value in $(K_n(s_i) \cap Q)^\nu_y$, denoted by   $\ol t^y_n$. \EEE Note that   this holds for $n$ sufficiently large \BBB depending on $y$, \EEE \ZZZZ but    regardless of $i$ in the sense that $(K_n(s_i) \cap Q)^\nu_y = (K_n(s_l) \cap Q)^\nu_y$ if $y \in P^\nu(Q) \sm (\Pi_i^n \cup \Pi_l^n)$ for $l>i$, due to  the monotonicity of $(K_n(s_i))_i$.    \ZZZZ In particular, we have \EEE 
 \begin{equation}\label{0302260850}
  (J_{{v}_{n}(s_{i})})^\nu_y  =  (\Gamma_n)^\nu_y  =  \lbrace \ol t^y_n  \rbrace, \  \ [({v}_{n}(s_{i}))^\nu_y](\ol t^y_n) = \eta_{i} \ \  \text{for $n$ large enough  for $\mathcal{H}^{d{-}1}$-a.e.\ $y \in P^\nu(Q) \setminus   \Pi_i^{n}$.}
\end{equation}  
\ZZZZ Here, we used \EEE that   $v_n(s_i) \in {\rm PC}(Q)$ with $v_n(s_i)= \BBB \eta_i \chi_D \EEE $ on $\partial Q$ \ZZZZ which \EEE is employed in both conditions in \eqref{0302260850}: for the first since it gives $\#(J_{{v}_{n}(s_{i})})^\nu_y\geq 1$ for any $y \in P^\nu(Q)$ (along with $\# (\Gamma_n)^\nu_y\geq 1$ for any $y \in P^\nu(Q)$, \ZZZZ see \eqref{bgammabound}), \EEE and for the second since it implies that, for  $y \in P^\nu(Q)\sm \Pi^n_i$, the amplitude of the unique jump is $\eta_i$. \EEE 
  By \eqref{local-h2}   we get  
\begin{align}\label{combon0}
  \mu^{i{+}1}_- \ZZZZ (Q) \EEE  \le      \int_{\Gamma } h^-_{s_{i}} (\cdot,\eta_{i{+}1},\nu^\Gamma)   \, \dhn +  \frac{\eps}{k}r^{d{-}1}.
\end{align}
Recall the definition   $\sigma_{i{+}1} = \eta_{i{+}1} - \eta_i$ below \eqref{for subsectionlast}, \eqref{vs}, and the definition of $D_n$ preceding \eqref{firstgammaproperty}. We use  $w_n(s_{i{+}1}) := {v}_{n}(s_{i}) + \sigma_{i{+}1} \chi_{D_n}$   as competitors on the right-hand side of \eqref{combon0}. As $w_n(s_{i{+}1})\to \bar{v}(s_{i{+}1}) = \bar{v}(s_{i}) + \sigma_{i{+}1} \chi_{D} $  in $L^1(\ZZZZ Q \EEE )$ \EEE for $n \to \infty$, the $\Gamma$-liminf inequality for the sequence  $\ZZZZ \H^-_{n}(s_{i}) \EEE(\cdot, Q)$ (recall \eqref{elapart-newdefin} and \eqref{1708240957})  yields 
\begin{align}\label{66} 
\int_{\Gamma } h^-_{s_{i}} (\cdot,\eta_{i{+}1},\nu^\Gamma)   \, \dhn     \le   \liminf_{n \to \infty}  \int_{J_{w_n(s_{i{+}1})}}  \big( g(|[w_n(s_{i{+}1})]|) -   \psi_{n}(s_{i})\big)^+ \,{\rm d} \mathcal{H}^{d{-}1}.
\end{align}
From this, using that (employ \eqref{0302260850} \BBB and $\Gamma_n = Q \cap \partial D_n$) \EEE
\begin{align}\label{13051859}
J_{w_n(s_{i{+}1})}\subset \big(\Gamma_n \sm B_i^{n}\big) &\cup \Big(( \Gamma_n \setminus  K_n(s_i)  ) \cap B^n_i\Big) \cup \Big(\big(  (\Gamma_n \cap K_n(s_i))\setminus J_{{v}_{n}(s_{i})}  \big) \cap B^n_i\Big) \notag \\& \cup \big( \Gamma_n  \cap J_{v_n(s_i)}\cap B^n_i\big) 
\cup \Big((  J_{v_n(s_i)} \setminus  \Gamma_n) \cap B^n_i\Big),
\end{align} \EEE
 we get that 
 \begin{align}\label{combon1.5}
   &  \BBB \int_{\Gamma } h^-_{s_{i}} (\cdot,\eta_{i{+}1},\nu^\Gamma)   \, \dhn   {\le}  \liminf_{n \to \infty}  \int_{\Gamma_n \setminus B^{n}_{i}} \hspace{-1.5em}  \big( g( \eta_{i{+}1}) -  \BBB \psi_n(s_i) \EEE \big)^+ {\rm d} \mathcal{H}^{d{-}1}   { + } \limsup_{n\to \infty}  \mathcal{H}^{d{-}1}\big( (\Gamma_n  \sm K_n(s_i) ) \cap B^n_i  \big) \EEE   g( \eta_{i{+}1} ) \EEE    \notag \\ 
 & \ \ \  + \limsup_{n\to \infty} \Big[  \mathcal{H}^{d{-}1}\big((   (\Gamma_n \cap K_n(s_i))\setminus J_{{v}_{n}(s_{i})}  ) \cap B^n_i\big) C_g\sigma_{i{+}1} +     \mathcal{H}^{d{-}1}( \Gamma_n  \cap J_{{v}_{n}(s_{i})}  \cap B^{n}_{i}) C_g\sigma_{i{+}1} \Big]  \notag  \\
 & \ \ \ + \limsup_{n\to \infty}   \int_{J_{{v}_{n}(s_{i})} \cap B^{n}_{i}}  \big(  g(|[{v}_{n}(s_{i})]|) -   \psi_{n}(s_{i})\big)^+ \,{\rm d} \mathcal{H}^{d{-}1} .
    \end{align}
  \BBB Let us explain the five \EEE addends in \eqref{combon1.5}  which are estimated \EEE from \eqref{66} evaluated on the \BBB five \EEE sets in the decomposition \eqref{13051859} of $J_{w_n(s_{i{+}1})}$.   \BBB For the first addend (coming from the first set in \eqref{13051859}), we use both items  in \eqref{0302260850}. \EEE For the second addend (coming from the second set in \eqref{13051859}), 
   we use  the fact that $\psi_n(s_i)$ is supported on $K_n(s_i)$ and the monotonicity of $g$ with $|[w_n(s_{i{+}1})]|\leq \eta_{i{+}1}$ (as $|[v_n(s_i)]|\leq \eta_i$); for the third addend (coming from the third set in \eqref{13051859}) we use the identity $|[w_n(s_{i{+}1})]|=\sigma_{i{+}1}$ on $\Gamma_n\sm J_{v_n(s_i)}$, $g(t)\leq C_g\,t \BBB +c_g \EEE$ (see \eqref{gproperty2}), \ZZZZ and  $\psi_n \ge c_g$,  \EEE  so that $(g(|[w_n(s_{i{+}1})]|)-\psi_n(s_i))^+\BBB \leq \EEE C_g \sigma_{i{+}1}$ \BBB on $(\Gamma_n \cap K_n(s_i))\setminus J_{{v}_{n}(s_{i})}  $; \EEE 
   \EEE then, we use the estimate $g(a+t)\leq g(a)+C_g\,t$ (see \eqref{gproperty2})  applied \EEE for $a=v_n(s_i)$ and $t=\sigma_{i{+}1}$, \ZZZZ and \EEE  $((\Gamma_n  \cap J_{v_n(s_i)})\cap B^n_i) \cup ((  J_{v_n(s_i)} \setminus  \Gamma_n) \cap B^n_i) \BBB = \EEE J_{v_n(s_i)} \cap B_i^n$  to estimate the last two addends  (evaluated on the last two sets in \eqref{13051859}). \EEE
 
 Eventually,  we combine  \eqref{combon0} \ZZZZ and \EEE \eqref{combon1.5} 
 to \BBB deduce \EEE
        \begin{equation}\label{combon1''}
 \begin{split}
 \mu^{i{+}1}_- \ZZZZ (Q) \EEE  &\le    \liminf_{n \to \infty} \ZZZZ  \int_{\Gamma_n \setminus B^n_{i}} \EEE  \big( g( \eta_{i{+}1} ) - \psi_n(s_i) \big)^+ \EEE \,{\rm d} \mathcal{H}^{d{-}1}   {+} C\frac{\eps}{k}r^{d{-}1} {+}   C_g    \limsup_{n\to \infty} \mathcal{H}^{d{-}1}( \Gamma_n  \cap \ZZZZ B^n_{i}) \EEE   \EEE \sigma_{i{+}1}.
\end{split}
    \end{equation} 
 \BBB Indeed,  \ZZZZ we estimate \EEE the sum of the third and fourth addends in \eqref{combon1.5}  with the last term \ZZZZ on \EEE  the right-hand side of \eqref{combon1''}. \ZZZZ Moreover,  \EEE  we use \EEE  \eqref{firstgammaproperty}   and $K_n(s_p)\subset K_n(s_i)$  (with \eqref{gproperty2}) \EEE to estimate the second addend in \eqref{combon1.5}  with $C\eps r^{d{-}1}/k$, and  \EEE \eqref{nochmali}   to estimate the last addend in \eqref{combon1.5}   with $\eps r^{d{-}1}/k$. \ZZZZ From \EEE  \eqref{combon1''} it \ZZZZ also \EEE follows that \EEE
     \begin{equation}\label{combon1}
 \begin{split}
 \mu^{i{+}1}_- \ZZZZ (Q) \EEE  \leq \liminf_{n \to \infty}  \int_{\Gamma_n \setminus \ZZZZ B^n_{i} \EEE }  \big( g( \eta_{i{+}1} ) -   g( \eta_i ) \big) \EEE \,{\rm d} \mathcal{H}^{d{-}1} 
+  C\frac{\eps}{k}r^{d{-}1} +   C_g  \limsup_{n\to \infty} \mathcal{H}^{d{-}1}( \Gamma_n  \cap  \ZZZZ B^n_{i}) \EEE  \sigma_{i{+}1}.
\end{split}
    \end{equation} 
In fact, \ZZZZ given a rectifiable set $\Phi_n \subset \Gamma_n$, \EEE  by  \eqref{nochmali}    and \EEE the second item in \eqref{0302260850} we get
\begin{equation}\label{25061443}
\BBB \limsup_{n \to \infty} \EEE \int_{\ZZZZ \Phi_n  \sm B_i^n \EEE} (g(\eta_i) -   \psi_{n}(s_{i}))^+ \,\dhn= \BBB \limsup_{n \to \infty} \EEE \int_{\ZZZZ \Phi_n  \sm B_i^n \EEE} (g(|[{v}_{n}(s_{i})]|) -   \psi_{n}(s_{i}))^+\EEE \,\dhn \leq \frac{\varepsilon}{k} r^{d{-}1}, \ \  
\end{equation}
 and then, since $(g(\eta_{i{+}1}) - \psi_{n}(s_{i}))^+ \leq g(\eta_{i{+}1})- g(\eta_i)+ (g(\eta_i) -   \psi_{n}(s_{i}))^+$, we get that for $n$ large enough
\begin{align}\label{isneededlater}
 \int_{\Phi_n \setminus B^{n}_{i}} \big( g(\eta_{i{+}1}) - \psi_{n}(s_{i}) \big)^+ \,{\rm d} \mathcal{H}^{d{-}1}  
\leq   \int_{\Phi_n \setminus B^{n}_{i}}  \big( g( \eta_{i{+}1}) - g( \eta_i) \big) \,{\rm d} \mathcal{H}^{d{-}1} + 2 \frac{\varepsilon}{k} r^{d{-}1}.
\end{align}
 For $\Phi_n = \Gamma_n$ we deduce   \eqref{combon1}  from  \eqref{combon1''}. \EEE (We have chosen a general set $\Phi_n$ for later purposes, see \eqref{newcollect-forlater} below.)  
 
 
  Moreover, by Lemma~\ref{le:projection} (applied to \BBB $\Psi=\Gamma_n$ \EEE and $A=B_i$), \eqref{basiccracklength}, \ZZZZ and \eqref{similarly-neu}, \EEE  we have
 \begin{align}\label{combon4}
 \limsup_{n\to \infty}   \mathcal{H}^{d{-}1}( \Gamma_n  \cap \ZZZZ B^n_{i} \EEE ) \le  C\frac{\eps}{k}r^{d{-}1}  + C\delta^{p}(Q)  + \mathcal{H}^{d{-}1}(\Pi^n_i) \ZZZZ \le  C\frac{\eps}{k}r^{d{-}1}  + C\delta^{p}(Q)  + \mathcal{H}^{d{-}1}(\Pi_i). \EEE  
 \end{align}
\ZZZZ Eventually, \EEE 
by \eqref{combon1}, \eqref{combon4}, \ZZZZ $B_i\subset B_i^n$, \EEE and \ZZZZ as  $\sigma_{i{+}1} \le \eta_{i{+}1}\le C_w$  (see Lemma \ref{prop-control-on-ela}), \EEE  we  obtain  \eqref{firtmain1}   for all $i \in \lbrace p,\ldots,k{-}1\rbrace$.

    \emph{Step 3: Proof of estimate \eqref{projection bound}, first step.} To show \eqref{projection bound}, our idea is to  first prove 
 \begin{align}\label{Pip}
\mathcal{H}^{d{-}1}(\Pi_p) \le C\frac{\eps}{k}r^{d{-}1}  + C\delta^{p}(Q),
\end{align}   
    and then \ZZZZ to \EEE  use an induction argument  to \ZZZZ prove \EEE that 
\begin{align}\label{projection bound-induction}
\mathcal{H}^{d{-}1}(\Pi_i) \le  \mathcal{H}^{d{-}1}(\Pi_p) \prod_{l=p{+}1}^{i}  (1+\widehat{C}\sigma_l) + \prod_{l=p{+}1}^{i}  (1+\widehat{C}\sigma_l) \cdot \widehat{C} \sum_{l=p{+}1}^{i}  \big(\delta^{p}(Q)\sigma_l+ \delta^{l}(Q) + \tfrac{\eps}{k}r^{d{-}1} \big)
\end{align} 
 for $p \le i \le k$,  for a constant $\widehat{C}>0$ large enough specified below in \eqref{tospecify}.  (Here, by convention we set $\prod_{l=p{+}1}^{p} = 1 $ and $\sum_{l=p{+}1}^{p} = 0$.) 
Once this is achieved,  \eqref{projection bound} indeed follows by   $\ZZZZ \sum_{l=p{+}1}^k \EEE  \sigma_l \le \eta_k \le  C_w$   (see Lemma \ref{prop-control-on-ela}) \EEE and \eqref{Pip}, along with the fact that $\log(1+x)\leq x$ for $x\geq 0$ implies $\prod_{l=p{+}1}^{i}  (1+\widehat{C}\sigma_l)\leq e^{\widehat{C} C_w\EEE}$.

In this step, we check that  \eqref{Pip} holds true.  This is achieved by the following projection \ZZZZ argument: \EEE 
\ZZZZ considering the functions \EEE $\ZZZZ {f}_n \EEE  \colon P^\nu(Q) \to \N_0$ defined by ${f}_n(y)= \# (K_n(s_p)\cap Q)^{\nu}_y$, we   compute by Fatou's Lemma, the Area Formula, \BBB   \eqref{firstgammaproperty},  \ZZZZ \eqref{bgammabound}, \EEE \eqref{0102261700},   as well as  \eqref{2prop}(ii)  that
 \[
 \begin{split}
&\hn\big( \{ y \in P^\nu(Q)  \colon \liminf\nolimits_{n\to +\infty} \#\big((K_n(s_p) \cap Q)^\nu_y\big) \ge 2\}\big) \ZZZZ \leq   \liminf_{n\to \infty}\int_{P^\nu(Q)}  \chi_{\lbrace f_n \ge 2 \rbrace } \big( {f}_n  (y) -1\big)  \,\dhn(y) \EEE \\
& \leq   \liminf_{n\to \infty}\int_{P^\nu(Q)}  \big(\ZZZZ {f}_n \EEE (y) -1\big)  \,\dhn(y) + \limsup_{n\to \infty} \mathcal{H}^{d{-}1}( \ZZZZ \lbrace \EEE  y \in P^\nu(Q) \colon \, \ZZZZ {f}_n \EEE (y) = 0 \rbrace) \\ & \le \liminf_{n\to \infty} \mathcal{H}^{d{-}1}(K_n(s_p) \cap Q)- r^{d{-}1} +  C\frac{\eps}{k}r^{d{-}1} \le  C\frac{\eps}{k}r^{d{-}1}+ C\delta^p(Q).
 \end{split}
 \] \EEE
\ZZZZ Recalling the definition in \eqref{defPi}, \EEE  \eqref{Pip} is confirmed.

    \emph{Step 4: Proof of estimate \eqref{projection bound}, induction step.} Clearly, property \eqref{projection bound-induction} holds true for $i = p$.   We now assume that \eqref{projection bound-induction} holds for some  $p \le i \le k{-}1$, and pass to $i{+}1$. The main idea consists in showing that, up to small errors, $\mu^{i{+}1}_{\rm conv}(Q) - \mu_{\rm conv}^{i}(Q)$ can  be \EEE bounded from below by the additional crack energy of the limit   $\mu^{i{+}1}_-(Q)$, 
    plus $\mathcal{H}^{d{-}1}(\Pi_{i{+}1} \setminus \Pi_{i})$. Then,  $\mathcal{H}^{d{-}1}(\Pi_{i{+}1} \setminus \Pi_{i})$ can be controlled by employing \eqref{countmist}.

We start by letting  $({v}_{n}(s_{i{+}1}))_n$   be a recovery sequence  of $\ol v(s_{i{+}1})$ (see \eqref{vs}) for $\ZZZZ \H^-_{n}(s_{i{+}1}) \EEE(\cdot, Q)$, thus \EEE  satisfying \eqref{nochmali} for $i{+}1$ in place of $i$.  Arguing as in Step~2, we may assume $v_n(s_{i{+}1})=\ol v(s_{i{+}1})$ on $\partial Q$ and $J_{{v}_{n}(s_{i{+}1})} \subset \Omega \cup \partial_D \Omega$. \EEE  In \EEE view of \eqref{defmu} \ZZZZ and \eqref{nochmali} for $i{+}1$, \EEE we can estimate, \ZZZZ for $n$ sufficiently large,  
  \begin{align*}
\mu^{i{+}1}_n(Q) - \mu_n^{i}(Q) & \ge \int_{K_n(s_{i{+}1}) \cap Q}  \big( \psi_{n}(s_{i{+}1}) \vee g(|[{v}_{n}(s_{i{+}1})]|) \big) \EEE \,{\rm d} \mathcal{H}^{d{-}1} +  \int_{J_{{v}_{n}(s_{i{+}1})} \setminus K_n(s_{i{+}1})}   g(|[{v}_{n}(s_{i{+}1})]|) \,{\rm d} \mathcal{H}^{d{-}1} \\ 
& \quad  - \int_{K_n(s_{i}) \cap Q}  \psi_{n}(s_{i}) \,{\rm d} \mathcal{H}^{d{-}1} - \ZZZZ 2 \EEE \frac{\eps}{k}r^{d{-}1}.
\end{align*}
Then,    using   $K_n(s_i) \subset K_n(s_{i{+}1})$ and the convention introduced in \eqref{convention}, we shortly  write \EEE
  \begin{equation}\label{0302261109}
\mu^{i{+}1}_n(Q) - \mu_n^{i}(Q)  \geq   \int_{(K_n(s_{i{+}1})  \cup J_{{v}_{n}(s_{i{+}1})}) \cap Q} \Big(   \big( \psi_{n}(s_{i{+}1}) \vee  g(|[{v}_{n}(s_{i{+}1})]|) \big) \EEE - \psi_n(s_i)\Big) \,{\rm d} \mathcal{H}^{d{-}1}  - \ZZZZ 2 \EEE  \frac{\eps}{k}r^{d{-}1}.
\end{equation}
 Recalling \ZZZZ \eqref{0302260850}, \EEE we claim that for $\mathcal{H}^{d{-}1}$-a.e.\ $y \in  P^\nu(Q)  \EEE $  and for $n\geq \ol n$ (depending on $\nu$, $y$) it holds \EEE
\begin{align}\label{0502261734}
 &  a_n^1(y):= \EEE \chi_{\Pi_{i{+}1} \setminus \Pi_i^n}(y) \EEE \sum_{t \in (K_{n}(s_{i{+}1}) \cup J_{{v}_{n}(s_{i{+}1})})^\nu_y} \big( (\psi_n(s_{i{+}1}))^\nu_y \vee g(|[({v}_{n}(s_{i{+}1}))^\nu_y]|)\big) \EEE (t)   - (\psi_n(s_{i}))^\nu_y(t)  \notag \\ &
\ge  a_n^2(y):= \EEE\chi_{\Pi_{i{+}1} \setminus \Pi_i^n}(y) \EEE  \Big(\big(g(\eta_{i{+}1})-\psi_y^{n,i}\big)^+ + c_g \Big),\quad \quad   \text{for }\psi_{n,i}^{y} := (\psi_{n}(s_{i}) \ZZZZ )^\nu_y \EEE (\ol t_n^y). 
\end{align}
In fact, by the  definition in \eqref{defPi} \ZZZZ and \eqref{0302260850}, \EEE given $y \in \Pi_{i{+}1} \setminus \ZZZZ \Pi^n_i \EEE $, 
 for $n\geq \ol n$ (depending on $\nu$, $y$) we have \EEE
that $\#(K_n(s_i) \cap Q)^\nu_y =1$ and $\#(K_n(s_{i{+}1}) \cap Q)^\nu_y \ZZZZ \ge 2$.   In particular, $\sum_{t \in (K_{n}(s_{i{+}1}) \cup J_{{v}_{n}(s_{i{+}1})})^\nu_y} (\psi_n(s_{i}))^\nu_y(t) = \psi_{n,i}^{y}$.     Moreover, $\# (J_{v_n(s_{i{+}1})})^\nu_y \geq 1$ as  $v_n(s_{i{+}1}) \in \PC(Q)$ and  $v_n(s_{i{+}1})=\ol v(s_{i{+}1})$ on $\partial Q$. 
  As \EEE $(\psi_n(s_i))^\nu_y(t)=0$ for $t \neq \ol t^y_n$,  $\psi_n(s_{i{+}1})(\ol t^y_n)\geq \psi_n(s_i)(\ol t^y_n)$, \EEE  and $\psi_n(s_j)\geq c_g$ on $K_n(s_j)$ for all $j$
by \ZZZZ Theorem \ref{sigma-comp}, \EEE 
  we get that the left-hand side in \eqref{0502261734} is at least $c_g$. Thus, \BBB to show \eqref{0502261734}, \EEE it is not restrictive to assume  \EEE
$g(\eta_{i{+}1})> \psi_y^{n,i}$.   
If $\#(J_{v_n(s_{i{+}1})})^\nu_y =1$,  let $(J_{v_n(s_{i{+}1})})^\nu_y = \{\tilde{t}_n^y\}$ and observe that $\#(K_n(s_{i{+}1}) \sm J_{{v}_{n}(s_{i{+}1})})^\nu_y \geq 1$, $[({v}_{n}(s_{i{+}1}))^\nu_y](\tilde{t}^y_n) = \eta_{i{+}1}$,   so that 
\begin{equation*}
\sum_{t \in (K_{n}(s_{i{+}1}) \cup J_{{v}_{n}(s_{i{+}1})})^\nu_y}\hspace{-3em}\big( (\psi_n(s_{i{+}1}))^\nu_y \vee g(|[({v}_{n}(s_{i{+}1}))^\nu_y]|)\big) \EEE(t) \geq g(\eta_{i{+}1}) + c_g \# (K_n(s_{i{+}1}) \sm J_{{v}_{n}(s_{i{+}1})})^\nu_y \geq g(\eta_{i{+}1}) + c_g.
\end{equation*} 
Otherwise, that is if $\#(J_{v_n(s_{i{+}1})})^\nu_y \geq 2$, the   second property  in  \EEE \eqref{gproperty}  applied to the addends of
 $\eta_{i{+}1} = \sum_{t \in (J_{v_n(s_{i{+}1})})^\nu_y}[({v}_{n}(s_{i{+}1}))^\nu_y](t)$ 
gives   
$
\sum_{t \in (J_{v_n(s_{i{+}1})})^\nu_y}g(|[({v}_{n}(s_{i{+}1}))^\nu_y](t) \ZZZZ |) \EEE  \geq g(\eta_{i{+}1}) +c_g$.
Thus, \eqref{0502261734} is confirmed. \EEE

We   now \EEE pass to the limit in \eqref{0302261109},  to get \EEE
\begin{align}\label{combon2.5}
\mu^{i{+}1}_{\rm conv}(Q) - \mu^{i}_{\rm conv}(Q) &  \ge \liminf_{n \to \infty}  \bigg( \EEE \int_{\Gamma_n   \setminus {B}^n_{i{+}1}}   \hspace{-1em} \big(  g(\eta_{i{+}1}) - \psi_{n}(s_{i})\big)^+ \,{\rm d} \mathcal{H}^{d{-}1} +  \int_{\Pi_{i{+}1} \setminus \ZZZZ \Pi^n_{i} \EEE }  \hspace{-1em} \big( g(\eta_{i{+}1}) - \psi_{n,i}^{y} \big)^+  \, {\rm d}\mathcal{H}^{d{-}1}(y) \bigg) \EEE  \notag\\ & \ \ \    + c_g \Hd(\Pi_{i{+}1} \setminus \Pi_{i}) -  C \EEE \frac{\eps}{k}r^{d{-}1}.
\end{align}
 To see this,  we observe that the left-hand \ZZZZ side converges \EEE by \eqref{1910252147} and $\mu^k_{\rm conv}(\partial Q) = 0$, and \EEE
  we split the integral on the right-hand side  of \eqref{0302261109} \EEE into the contributions  in \EEE $K_n(s_{i{+}1}) \sm  {B}_{i{+}1}$,  $K_n(s_{i{+}1}) \cap  ({B}_{i{+}1} \setminus \ZZZZ  {B}^n_{i}) \EEE $,  $K_n(s_{i{+}1}) \cap  \ZZZZ  {B}^n_{i} \EEE $, \ZZZZ and \EEE $J_{{v}_{n}(s_{i{+}1})} \sm K_n(s_{i{+}1})$. \EEE We   note that by \eqref{defmu}--\eqref{defmu-mono} the  integrand is nonnegative, and thus  the \ZZZZ contributions \EEE in $K_n(s_{i{+}1}) \cap \ZZZZ B^n_i \EEE $ \ZZZZ and  $J_{{v}_{n}(s_{i{+}1})} \sm K_n(s_{i{+}1})$ \EEE  can be neglected on the right-hand side of \eqref{combon2.5}. 
 We also notice that the contribution on $K_n(s_{i{+}1}) \cap  ({B}_{i{+}1} \setminus  {B}^n_{i})$ may be controlled by an integral over  $P^\nu(Q)$ \EEE
 in view of \eqref{0502261734}. In order to control in \eqref{combon2.5} the liminf of the sum, we bound also the first contribution in terms of an integral   on  $P^\nu(Q)$: \EEE by using that $B_{i{+}1}\subset  {B}_{i{+}1}^n$,  and by \eqref{0302260850} (for $i{+}1$ in place of $i$, giving in particular $(K_n(s_{i{+}1}) \sm {B}_{i{+}1}^n) \cap Q = J_{{v}_{n}(s_{i{+}1})} \sm {B}_{i{+}1}^n = \Gamma_n \sm {B}_{i{+}1}^n$), using the Area Formula we get  that
the contribution  of the integral in \eqref{0302261109} \EEE on $K_n(s_{i{+}1}) \sm  {B}_{i{+}1}$
is  bounded from below by \EEE
\[
\int_{\Gamma_n   \setminus {B}^n_{i{+}1}} \hspace{-1em} \big( g(\eta_{i{+}1}) - \psi_{n}(s_{i})\big)^+ \,{\rm d} \mathcal{H}^{d{-}1}= \int_{ P^\nu(Q)} \hspace{-1.5em} a_n(y) \, \dhn(y), \ \   a_n(y):= \chi_{ P^\nu(Q) \EEE \sm \Pi_{i{+}1}^n}\big( g(\eta_{i{+}1})- \psi^y_{n,i} \big)^+ |\langle \nu  , \nu(\ol t^y_n) \rangle|^{-1},
\]
 where we \EEE  recall the definition of $\psi^y_{n,i}$ from \eqref{0502261734} and  denote \EEE by $\nu(\ol t^y_n)$ the normal to the set $\Gamma_n$ at $y + \ol t^y_n \nu$. 
By \eqref{0502261734}, $d_n(y):= a_n(y) + a_n^1(y)\geq e_n(y):= a_n(y) + a_n^2(y)$ for $n\geq  \ol n \EEE$, and Fatou's lemma for $d_n-e_n\geq 0$ gives that
$\liminf_n \int_{ P^\nu(Q)} e_n(y) \,\dhn(y)\leq \limsup_n \int_{ P^\nu(Q)} d_n(y) \,\dhn(y)$. \EEE We thus conclude \eqref{combon2.5} using that  $(\Pi_{i{+}1} \setminus \Pi_i^n) \triangle (\Pi_{i{+}1} \setminus \Pi_i) \subset \Pi_{i}^n \setminus \Pi_i$ \EEE along with \eqref{similarly-neu}.
\EEE


Further, we claim that 
\begin{equation}\label{0502261916}
 \Psi{:=} \int_{\ZZZZ \Gamma_n  \EEE  \cap (B_{i{+}1}\sm \ZZZZ B^n_i) \EEE } \hspace{-3em}\big(g(\eta_{i{+}1}) - \psi_n(s_i)\big)^+  \dhn  {-} \int_{\Pi_{i{+}1} \sm \ZZZZ \Pi^n_i \EEE} \hspace{-2em}\big( g(\eta_{i{+}1}) - \psi_{n,i}^{y}\big)^+  \dhn(y) 
\leq C  \frac{\varepsilon}{k}r^{d{-}1} +  \ZZZZ C \EEE \sigma_{i{+}1} \delta^p(Q).
\end{equation}
\ZZZZ Indeed, \EEE for each $y \in \Pi_{i{+}1} \sm \ZZZZ \Pi^n_i \EEE $ we let  $f_{ n \EEE}(y):= \big(g(\eta_{i{+}1}) - \psi_{n,i}^{y}\big)^+$, $\widehat{f}_{ n \EEE}(y):= \big(g(\eta_{i}) - \psi_{n,i}^{y}\big)^+$. 
 For \EEE $\nu(\ol t^y_n)$ defined above,
 $0\leq f_{ n \EEE}(y) \big(|\langle \nu, \nu(\ol t^y_n) \rangle|^{-1}-1\big) \leq (g(\eta_{i{+}1})-g(\eta_i) + \widehat{f}_{n \EEE}(y))  \big(|\langle \nu, \nu(\ol t^y_n) \rangle|^{-1}-1\big)$, so
recalling  the definition $\psi_{n,i}^{y} = (\psi_{n}(s_{i})\ZZZZ )^\nu_y\EEE(\ol t_n^y)$, integrating on $\Pi_{i{+}1}\sm \ZZZZ \Pi^n_i \EEE $, and using the Area Formula we get
 \begin{equation*}\label{0502261916-000}
\begin{split}
 \Psi & =   \int_{\Pi_{i{+}1}\sm \ZZZZ \Pi^n_i \EEE } \hspace{-1em} {f}_{ n \EEE}(y) \big(|\langle \nu, \nu(\ol t^y_n) \rangle|^{-1}{-}1\big)  \, \dhn(y) 
\leq  \int_{\Pi_{i{+}1}\sm \ZZZZ \Pi^n_i \EEE} \hspace{-1em} \big(g(\eta_{i{+}1})-g(\eta_i) + \widehat{f}_{ n \EEE}(y)\big)  \big(|\langle \nu, \nu(\ol t^y_n) \rangle|^{-1}{-}1\big)  \, \dhn(y) \\&
= \int_{´\ZZZZ \Gamma_n \EEE  \cap (B_{i{+}1}\sm \ZZZZ B^n_i \EEE )} \big(g(\eta_{i}) - \psi_n(s_i)\big)^+ \, \dhn - \int_{\Pi_{i{+}1} \sm \ZZZZ \Pi^n_i \EEE } \big(g(\eta_{i}) - \psi_{n,i}^{y}\big)^+  \, \dhn(y) \\& \hspace{1em}+ (g(\eta_{i{+}1})-g(\eta_i)) \Big(\ZZZZ \mathcal{H}^{d{-}1}(\Gamma_n \EEE \cap (B_{i{+}1}\sm \ZZZZ B^n_i)) \EEE - \hn(\Pi_{i{+}1} \sm \ZZZZ \Pi^n_i) \EEE \Big).
\end{split}
\end{equation*} 
 The contribution \EEE in the last line can be controlled by  $C \frac{\varepsilon}{k}r^{d{-}1} + C_g \sigma_{i{+}1} \delta^p(Q)$, where we use \eqref{gproperty2}, \ZZZZ $\sigma_{i{+}1} \le C_w$, \EEE \eqref{basiccracklength}, and  Lemma~\ref{le:projection} applied to \BBB $\Psi=\Gamma_n$ \EEE and $A=B_{i{+}1}\sm  \ZZZZ B^n_i \EEE $.  For \EEE those in the penultimate line we use \eqref{25061443} with $\Phi_n=\Gamma_n$ for the first  term \EEE and \EEE
 we neglect the   second  term, \EEE as the integrand is nonnegative.  This gives \eqref{0502261916}. \EEE
 
Next, by using \eqref{isneededlater} for $\Phi_n = \Gamma_n \cap (B^n_{i{+}1} \setminus B_{i{+}1})$
\begin{align}\label{newcollect-forlater}
\limsup_{n \to \infty}  \int_{\Gamma_n \cap (B^n_{i{+}1}   \setminus (B_i^n \cup {B}_{i{+}1} ))   }  \hspace{-1cm} \big(  g(\eta_{i{+}1}) - \psi_{n}(s_{i})\big)^+ \,{\rm d} \mathcal{H}^{d{-}1} &\le \limsup_{n \to \infty}  \int_{\Gamma_n \cap (B^n_{i{+}1} \setminus B_{i{+}1})}  \hspace{-1cm} \big( g( \eta_{i{+}1}) - g( \eta_i) \big) \,{\rm d} \mathcal{H}^{d{-}1} + C\frac{\varepsilon}{k} r^{d{-}1}
\notag \\& \le  C \frac{\varepsilon}{k} r^{d{-}1}  + C\delta^{p}(Q)\sigma_{i{+}1},
\end{align}
where the latter estimate follows by \eqref{similarly} and since $g( \eta_{i{+}1} ) -   g(\eta_i) \le C_g(\eta_{i{+}1}-\eta_i ) = C_g \sigma_{i{+}1}$ by \eqref{gproperty2}, as well as $\sigma_{i{+}1} \le C_w$.
  Collecting \eqref{combon2.5}, \eqref{0502261916}, and \eqref{newcollect-forlater}, \EEE  and employing  \ZZZZ  \eqref{similarly-neu}, \EEE       we obtain   
\begin{equation}\label{combon2}
\begin{split}
\mu^{i{+}1}_{\rm conv}(Q) {-} \mu^{i}_{\rm conv}(Q) &  \ge \liminf_{n \to \infty}  \int_{\Gamma_n   \setminus \ZZZZ B^n_{i} \EEE }   \hspace{-1em}\big(  g(\eta_{i{+}1}) {-} \psi_{n}(s_{i})\big)^+ {\rm d} \mathcal{H}^{d{-}1} {+}   c_g   \mathcal{H}^{d{-}1}(\Pi_{i{+}1} \setminus \ZZZZ \Pi_{i}) \EEE {-} C\frac{\eps}{k}r^{d{-}1}  {-}  C\sigma_{i{+}1}\delta^{p}(Q).
\end{split} 
\end{equation}
  Then,  combining  \ZZZZ \eqref{combon1''} \EEE  and  \eqref{combon4},   from \eqref{combon2} \ZZZZ and \eqref{countmist} \EEE we derive 
$${  \mathcal{H}^{d{-}1}(\Pi_{i{+}1} \setminus \Pi_{i})  \le C\Big( \delta^{i{+}1}(Q) + \frac{\eps}{k} r^{d{-}1} +  \big(\delta^{p}(Q)  + \mathcal{H}^{d{-}1}(\Pi_i)\big) \sigma_{i{+}1} \Big). } $$
  Eventually, using  the induction assumption \EEE \eqref{projection bound-induction} 
 we get 
\begin{align}\label{tospecify}
\mathcal{H}^{d{-}1}(\Pi_{i{+}1}) & = \mathcal{H}^{d{-}1}(\Pi_i)  +  \mathcal{H}^{d{-}1}(\Pi_{i{+}1} \setminus \Pi_{i})  \le (1+C\sigma_{i{+}1})  \mathcal{H}^{d{-}1}(\Pi_i) +  C  \delta^{i{+}1}(Q) + C\frac{\eps}{k} r^{d{-}1} +  C \delta^{p}(Q)  \sigma_{i{+}1} \notag \\
& \le   \mathcal{H}^{d{-}1}(\Pi_p) \prod_{l=p{+}1}^{i{+}1}  (1+\widehat{C}\sigma_l) + \prod_{l=p{+}1}^{i{+}1} (1+\widehat{C}\sigma_l) \cdot \widehat{C} \sum_{l=p{+}1}^{i}  \big(\delta^{p}(Q)\sigma_l+ \delta^{l}(Q) + \tfrac{\eps}{k}r^{d{-}1} \big) \notag\\ & \ \ \ \  +  C  \delta^{i{+}1}(Q) + C\frac{\eps}{k} r^{d{-}1} +  C \delta^{p}(Q)  \sigma_{i{+}1} \notag\\
& \le   \mathcal{H}^{d{-}1}(\Pi_p) \prod_{l=p{+}1}^{i{+}1}  (1+\widehat{C}\sigma_l) + \prod_{l=p{+}1}^{i{+}1}  (1+\widehat{C}\sigma_l) \cdot \widehat{C} \sum_{l=p{+}1}^{i{+}1}  \big(\delta^{p}(Q)\sigma_l+ \delta^{l}(Q) + \tfrac{\eps}{k}r^{d{-}1}\big),
\end{align}
where we choose $\widehat{C}$ larger than the generic constant $C$. This shows \eqref{projection bound-induction} at time step $i{+}1$ and concludes the proof.  
\end{proof}

 \section*{Acknowledgements} 
 This research was funded by the Deutsche Forschungsgemeinschaft (DFG, German Research Foundation) - 377472739/GRK 2423/2-2023.  
 
  VC is member of the Gruppo Nazionale per l'Analisi Matematica, la Probabilit\`a e le loro Applicazioni (GNAMPA) of the Istituto Nazionale di Alta Matematica (INdAM).

\EEE
\begin{appendices}

\section{On the integral representation for functionals  with  precrack\EEE}\label{sec:App}

In this appendix,  we include a direct proof of the $\Gamma$-convergence and  integral representation results used in the proof of Theorem~\ref{thm:convpreliminaries}. In particular, we will only resort to basic integral representation results for Sobolev \cite{ButDM85}, piecewise constant \cite{BraChP96}, and $\SBV$ functions \cite{BouFonLeoMas02}, avoiding more complicated vectorial and generalized settings treated in \cite{Sto lavoro GSBV, FM}. Although having Theorem~\ref{thm:convpreliminaries} in mind, we   work in a more general setting of varying bulk and surface energy densities as this may be convenient for future work. \EEE This could be further generalized, see Remark \ref{rem:continuityf}. 

\BBB  Given \EEE $0 < \alpha \le \beta <+ \infty$, let   $f_n \colon \Omega' \times \R^{d} \to [0,+\infty)$   be Carath\'eodory  functions such that for \BBB every \EEE  $x\in \Omega'$ and all  $\zeta \in \R^{d}$  we have
\begin{align}\label{eq: general bound}
\alpha | \zeta  |^p \le f_n(x,\zeta) \le  \beta (1+ | \zeta  |^p).
\end{align}
\BBB Let \EEE  $g_n \colon \Omega'  \times \R \times  \mathbb{S}^{d{-}1} \to [0,+\infty)$ be Borel \BBB functions \EEE  satisfying \EEE $g_n(x,0,\nu)= 0$ and
\begin{itemize}
\item[(g1)] 
$\alpha  \le g_n(x,\xi,\nu) \le \beta\, (1+|\xi|)$   for every  \BBB $x \in \Omega'$, \EEE  $\xi \in \R\sm \{0\}$,   $\nu \in \mathbb{S}^{d{-}1}$;
\item[(g2)] $g_n(x,\cdot, \nu)$ is nondecreasing in $[0,+\infty)$ for every $x \in \Omega'$ and $\nu \in \Sn$;
\item[(g3)] $g_n(x, \xi, \nu)=g_n(x,-\xi,-\nu)$ for every $x \in \Omega'$, $\xi \in \R$, and $\nu \in \Sn$;
\item[(g4)]
there exists a modulus of continuity $\omega\colon [0,\infty)\to [0,\infty)$ such that
\begin{equation}\label{gncont}
{ | \EEE g_n(x, \xi_2, \nu) - g_n(x, \xi_1, \nu)  | \EEE \leq \omega(\xi_2-\xi_1) \BBB \big(  g_n(x, \xi_1, \nu) + g_n(x, \xi_2, \nu)  \big) \EEE  \quad \text{ for }0\ZZZZ < \EEE \xi_1 < \xi_2.}
\end{equation}
\end{itemize}
Let  us \EEE  also introduce the   alternative  \BBB conditions \EEE 
\begin{itemize}
\item[(g1')] $\alpha \, (1+|\xi|) \le g_n(x,\xi,\nu) \le \beta\, (1+|\xi|)$ for every \BBB $x \in \Omega'$, \EEE  $\xi \in \R\sm \{0\}$,   $\nu \in \mathbb{S}^{d{-}1}$;
\item[(g1'')]  \BBB  $\alpha   \le g_n(x,\xi,\nu) \le \beta$ for every $x \in \Omega'$,   $\xi \in \R\sm \{0\}$,   $\nu \in \mathbb{S}^{d{-}1}$. \EEE
\end{itemize}
In the following, we  always assume (g3)--(g4), and specify if we  assume either (g1)--(g2), or (g1'), \BBB or  (g1''). \EEE
 We emphasize that we do not assume the conditions \eqref{gproperty}--\eqref{gproperty2} in this section. \EEE

We consider the functionals $\F_n$, $\G_n$, $\E_n \colon L^1(\Omega') \times \A(\Omega') \to [0,+\infty]$ given for any $A \in \A(\Omega')$ by 
\begin{align}\label{fg}
\mathcal{F}_n(u, A) &:=  \int_A f_n(x,\nabla u(x)) \, {\rm d}x \quad  \text{if } \ZZZZ u|_A \in W^{1,p}(A), \EEE \notag \\ \mathcal{G}_n(u, A) & :=  \int_{ A \cap J_u \EEE} g_n(\cdot,[u],\nu_u) \,\dhn  \quad  \text{if } \ZZZZ u|_A \in \PC(A) \EEE \notag\\
 \E_n(u,A) &:= \int_A f_n(x,\nabla u(x)) \, {\rm d}x + \int_{ A \cap J_u \EEE} g_n(\cdot,[u],\nu_u) \,\dhn \quad  \text{if } \ZZZZ u|_A \in \SBV^p(A) \EEE
\end{align}
and $+\infty$ otherwise. Moreover, given   a sequence of rectifiable sets $K_n$ in  $\Omega'$  and  Borel functions   $\varphi_n \colon K_n \to [0,+\infty)$ such that  $\Hd(K_n) \le C_0$, 
we let  
\begin{align}\label{gnm}
\Gnm(u, A) =    \int_{J_u \cap K_n \cap  A} \big(  g_n(\cdot,[u],\nu_u) - g_n(\cdot,\varphi_n,\nu_u) \big)^+ \, {\rm d}\Hd + \int_{(J_u \sm K_n) \cap  A} g_n(\cdot,[u],\nu_u) \, {\rm d}\Hd    
\end{align}
\ZZZZ if $u|_A \in \PC(A)$ \EEE and $\Gnm(u, A) =  +\infty$ otherwise.   Similarly to \eqref{convention}, \EEE for notational convenience, extending $\varphi_n = 0$ outside of $K_n$, and using $g_n(x,0,\nu) =0$, the functional in \eqref{gnm} can be written shortly as   
\begin{align*}
\Gnm(u, A) =    \int_{J_u  \cap  A} \big(  g_n(\cdot,[u],\nu_u) - g_n(\cdot,\varphi_n,\nu_u) \big)^+ \, {\rm d}\Hd .
\end{align*}

 It is known (see \cite{ButDM85, BouFonLeoMas02, GiacPonsi}) that $\mathcal{F}_n$ $\ol \Gamma$-converge  (up to a subsequence) \EEE  to $\mathcal{F} $,  where \EEE
 \begin{equation}\label{flimit}
 \F(u,A)= \int_A f(x, \nabla u(x)) \dx \quad \text{for }u\in W^{1,p}(\Omega'),
\end{equation}
 with $f$ characterized by $f(x, \zeta)= \limsup_{\varrho\to 0^+} \BBB \varrho^{-d} \EEE \mathbf{m}_{\mathcal{F}}^{W^{1,p}}(l_{\zeta,x}, Q^\nu_\varrho(x)) $ for all $x \in \Omega'$ \BBB and \EEE $\zeta \in \R^d$,  where $l_{\zeta,x}(y) \BBB := \EEE  \zeta(y-x)$ for $y \in \R^d$. \EEE
 Our goal is to \EEE prove that an analogous property holds for $\mathcal G_n$ and $\mathcal G$.

\subsection{Fundamental estimate} 
 As a preparation, \EEE we first show that a \emph{fundamental estimate} holds uniformly for the  sequence $(\mathcal{G}_n)_n$.

\begin{lemma}\label{le:fundestGn} 
Assume    {\rm (g3)}--{\rm (g4)} \BBB  and either {\rm (g1')} or {\rm (g1'')}. \EEE 
Let $\eta>0$, let  $A'$, $A$, $B \in \mathcal{A}(\Omega')$ with $A' \subset \subset A$, and let $S:= (A\sm A') \cap B$. Then, there \BBB exists \EEE  $\Lambda>0$ such that for every $n \in \N$,  for every $\sigma \in (0,1)$ \BBB  with $\omega(2\sigma) \le \frac{1}{2}$, \EEE   and $u \in \PC(A)$, $v\in \PC(B)$ there exist $\phi \in C_c^\infty(A;  [0,1] \EEE )$ with $\phi=1$    in a neighborhood of $\ol{A'}$ and $w \in \PC( \BBB A' \EEE \cup B)$ such that $\hn(J_{w} \sm (J_u \cup J_v \cup S))=0$ and   
\begin{itemize}
\item[(i)]  $ \|\phi u+(1-\phi)v-w\|_{L^\infty( \BBB A' \EEE \cup B)}  \le \sigma$, and
\item[(ii)] $\mathcal{G}_n(w, A'\cup B)  \leq (1 + \BBB 2 \EEE  \omega(2\sigma))^2(1+\eta) \big( \mathcal{G}_n(u,A) + \mathcal{G}_n(v,B)\big) + \frac{\Lambda}{\sigma}\|u-v\|_{L^1(S)}$.
\end{itemize}\BBB 
Under the additional assumption that $u$ takes values in $[0,\xi]$ and $v$ takes values in $\lbrace 0, \xi \rbrace$ \ZZZZ for some $\xi >0$, \EEE the function $w$ can be chosen such that $w = v$ on $B \setminus A$.  \EEE

\end{lemma}
 Here, \EEE $u$, $v$ are extended arbitrarily outside $A$ and $B$, respectively. 
\begin{proof}
For $k\in \N$ with $k\geq \frac{\beta}{\alpha \,\eta}$, let $A_1,\dots, A_{k{+}1}\subset \Omega'$ with $A'\subset \subset A_1 \subset \subset \dots \subset \subset A_{k{+}1} \subset \subset A$,  and \EEE $S_i:= B \cap (A_{i{+}1}\sm \ol {A_i})$ \ZZZZ for $i \in \{1,\dots, k\}$. \EEE
Then,  there exists $i_0 \in \{1,\dots, k\}$ such that  
\begin{equation}\label{etaound}
\mathcal{G}_n(u, S_{i_0})+ \mathcal{G}_n(v, S_{i_0})\leq  \frac{1}{k} \EEE \big( \mathcal{G}_n(u, S) + \mathcal{G}_n(v, S)\big) \leq  \frac{\alpha \eta}{\beta}  \EEE \big( \mathcal{G}_n(u, A) + \mathcal{G}_n(v, B)\big).
\end{equation}
For $\phi \in C_c^\infty(A_{i_0+1})$ with $0\leq \phi\leq 1$ and $\phi=1$ on a neighborhood   of $\ol {A_{i_0}}$, the function   
\BBB $\widehat{w}:=\phi u + (1-\phi) v$ \EEE 
 satisfies
 $\widehat{w}\in \SBV( \BBB A' \EEE  \cup B)$ with $J_{\widehat{w}}\subset J_u  \cup \EEE J_v$, 
$\nabla \widehat{w}=0$ a.e.\ in $( \BBB A' \EEE  \cup B) \sm \{0<\phi<1\} \supset ( \BBB A' \EEE  \cup B) \sm S_{i_0}$, 
and
\begin{equation}\label{3103261226'}
\int_{ \BBB A' \EEE  \cup B}|\nabla \widehat{w}|\dx = \int_{S_{i_0}}|\nabla \widehat{w}|\dx \leq \|\nabla \phi\|_{\infty}\, \|u-v\|_{L^1(S_{i_0})}.
\end{equation}
By the Coarea Formula, 
$
\int_E |\nabla \widehat{w}|\dx= |\mathrm{D}\widehat{w}|(E \sm J_{\widehat{w}})= \int_{\R} \hn(\partial^* \{ \widehat{w}>t\} \cap (E \sm J_{\widehat{w}}))\dt
$
for every Borel set $E \subset \BBB A' \EEE  \cup B$.   Thus,  by the Mean Value Theorem, for every $l \in \Z$ 
there exists $t_l \in  (  \frac{\sigma}{2} \EEE  \, l,  \frac{\sigma}{2} \EEE  (l+1))$ such that
\begin{equation}\label{3103261227'}
\ZZZZ \int_{A' \cup B} \EEE |\nabla \widehat{w}|\dx \geq  \frac{\sigma}{2} \EEE  \sum_{l\in \Z} \hn\big(\partial^* \{ \widehat{w}>t_l\}  \ZZZZ \cap ((A' \cup B)\sm J_{\widehat{w}}) \EEE \big).
\end{equation}
Defining 
\begin{equation}\label{3103261103'}
w:= \sum_{l \in \Z} t_l \chi_{\{t_{l-1}< \widehat{w} \leq t_l\}},
\end{equation}
we have that $w \in \PC( \BBB A' \EEE  \cup B)$  with $J_w = J_u$ on $A_{i_0} \BBB \cap (A' \cup B)\EEE $ and $J_w = J_v$ on $B \setminus A_{i_0+1}$.  In particular, this shows $\hn(J_{w} \sm (J_u \cup J_v \cup S))=0$. \EEE  Moreover,  \EEE 
$\|w-\widehat{w}\|_\infty \leq \sigma$,  as well as  $|[w]-[\widehat{w}]|\leq 2 \sigma $ on $J_{\widehat{w}}$,  and \EEE  $|[w]|\leq  \sigma$ on $J_w \sm J_{\widehat{w}}$. Therefore, by the definition of $\widehat{w}$, (i) follows. \BBB By (g4) we get  
$$ \G_n(w, A'\cup B) \le \widehat{\G}_n(\widehat{w}, A'\cup B) + \omega(2\sigma)\big( \G_n(w, A'\cup B) +   \widehat{\G}_n(\widehat{w}, A'\cup B)   \big) + \int_{J_w \sm J_{\widehat{w}}} g_n(\cdot, [w], \nu_w) \,\dhn.  $$ 
\ZZZZ Here, $\widehat{\G}_n$ corresponds to the surface energy in \eqref{fg} which \ZZZZ is \EEE also defined for $\SBV$ functions. \EEE
Since $\omega(2\sigma) \le \frac{1}{2}$, we have $(1-\omega(2\sigma))^{-1} \le 1 + 2\omega(2\sigma) \le 2$, and  rearranging yields
$$ \G_n(w, A'\cup B) \le (1 + 2\omega(2\sigma))^2 \widehat{\G}_n(\widehat{w}, A'\cup B)  +  2\int_{J_w \sm J_{\widehat{w}}} g_n(\cdot, [w], \nu_w) \,\dhn.  $$ \EEE
By \eqref{3103261226'}, \eqref{3103261227'}, and \eqref{3103261103'} we further get 
$\hn(J_w \sm J_{\widehat{w}})\leq \frac{ 2 \EEE \|\nabla \phi\|_{\infty}}{\sigma}\|u-v\|_{L^1(S_{i_0})}$. Then, using the 
growth conditions  from above and below \EEE given by (g1') or (g1''), respectively, as well as  \eqref{etaound} we derive  \EEE
\begin{equation*}
\begin{split}
\G_n(w, A'\cup B) &  
\leq  (1+2\omega(2\sigma))^2 \Big(\G_n(u, A)+ \G_n(v, B) +  \frac{\beta}{\alpha} (\G_n(u, S_{i_0})+  \G_n(v, S_{i_0}))\Big) +   2\beta(1 +\sigma) \EEE \, \hn(J_w \sm J_{\widehat{w}})
\\&
\leq (1+2\omega(2\sigma))^2 (1+\eta) \Big(\G_n(u, A)+ \G_n(v, B) \Big) + \frac{\Lambda}{\sigma}\|u-v\|_{L^1(S_{i_0})}
\end{split}
\end{equation*}
 for $\Lambda:= \BBB 8 \EEE \beta\, \|\nabla \phi\|_{\infty}$.  The energy estimate  (ii) \EEE is concluded by using  $S_{i_0}\subset S$. 

\BBB
We now prove the additional statement under the assumption that   $u \in \PC(A;[0,\xi])$  and  $v \in \PC(B;\lbrace 0, \xi\rbrace)$ for some $\xi >0$. To this end, by the monotonocity of $\omega$, it is not restrictive to assume that $\sigma \le \frac{\xi}{4} \wedge 1$. The idea consists in slightly changing the construction  of $w$ in \eqref{3103261103'}: defining $C_{\sigma}:=\lfloor \frac{ 2 \EEE \xi}{\sigma} \rfloor$, \BBB we let $t_l  \in  (  \frac{\sigma}{2} \EEE  \, l,  \frac{\sigma}{2} \EEE  (l+1))$ for  $0\leq l\leq C_{\sigma}-1$  be the values chosen before \eqref{3103261227'}, and set 
\begin{equation*}
\ZZZZ w:= \EEE \sum_{l=1}^{C_\sigma-1} t_l \chi_{\{t_{l-1}< \widehat{w} \leq t_l\}} +  \xi \chi_{\{\widehat{w}> t_{C_\sigma-1}\}}.
\end{equation*}
\ZZZZ This implies that $w(x)=\widehat{w}(x)$ for all $x \in A' \cup B$ with $\widehat{w}(x) \in \lbrace 0, \xi \rbrace$. Moreover, we recall that $v = \widehat{w}$ on $B \sm A$ and $v \in \PC(B;\lbrace 0, \xi\rbrace)$.    This shows \EEE   $w=\widehat{w}=v$ on $B\sm A$.  
\end{proof}

\begin{proof}[Proof of Lemma~\ref{le:fundestHnm}]
\BBB We first observe that by \eqref{gproperty}--\eqref{gproperty2} and the definition in \eqref{HHNN} the functional $\H_n^{\varepsilon,-}$ satisfies the conditions  {\rm (g1'')} and   {\rm (g3)}--{\rm (g4)}, where  the constants $\alpha$ and $\beta$ depend on $\eps$, and a modulus of continuity $\omega$ can be given explicitly as $\omega(t) = C_g \eps^{-1} t$.  Then,  the statement follows by Lemma \ref{le:fundestGn}, where we set $\omega_g^\eps(t) =  (1+2\omega(2t))^2  -1$. 
   \end{proof}

 \subsection{$\Gamma$-convergence of $\mathcal G_n$}

We now 
 study \EEE the asymptotic behavior of the sequence $(\mathcal G_n)_n$. \EEE

\begin{proposition}\label{gammaconvgn}
Assume  {\rm (g1')}, {\rm (g3)}, {\rm (g4)} \EEE 
and let $\G_n$ be defined as  in \eqref{fg}. Then,  up to a subsequence,   $\G_n$ \BBB $\ol \Gamma$-converge \EEE   to $\G $, where $\G(u,A)=\int_{ A \cap J_u \EEE} g(\cdot,[u], \nu_u) \,\dhn$  for $u \in \PC(\Omega')$ \EEE and $g$  is \EEE  characterized by
\begin{equation}\label{3103262023}
\begin{split}
g(x,\xi,\nu) = \limsup_{\varrho \to 0^+} \frac{\mathbf{m}_{\G}^{\PC}(\ol u_{x,\xi,\nu}, Q^\nu_\varrho(x))}{\varrho^{d{-}1}} \quad \text{for all $x \in \Omega'$, $\xi \in \R$, and $\nu \in \Sn$}.
\end{split}
\end{equation}
\end{proposition}

\begin{proof}
By \cite[Theorem~16.9]{DMLibro},  we have that $\G_n$ \BBB $\ol \Gamma$-converge \EEE to a functional $\G$, up to a subsequence.  As  $\G_n(u, \cdot)$ \BBB is a  measure \EEE for all $u \in L^1(\Omega')$,  \BBB in view of \EEE  \cite[Theorem~18.5]{DMLibro}, \BBB  it suffices to prove that,  for any $u \in \PC(\Omega')$, $\G(u, \cdot)$ is subadditive  \EEE in order to deduce that  it is a measure. Moreover, by the form of $\G_n$, it can be shown with standard arguments (see e.g.\ \cite[Remark~16.3 and Proposition~16.15]{DMLibro}) that ${\G}$ is also  {local} on $\mathcal{A}(\Omega')$, that $\G(\cdot, A)$ is lower semicontinuous with respect to $L^1(\Omega')$, and that 
 $\alpha \int_{ A \cap J_u \EEE}(1+|[u]|)\,\dhn \leq \G(u,A)\leq \beta \int_{ A \cap J_u \EEE}(1+|[u]|)\,\dhn$ for every $A \in \mathcal{A}(\Omega')$  and $u \in \PC(\Omega')$. \EEE 
Therefore,  once the subadditivity of $\G(u, \cdot)$ for any $u \in  \PC(\Omega')$ \EEE is shown, \BBB the desired integral representation  on $\PC(\Omega'){\times} \mathcal{A}(\Omega')$ \EEE follows from \EEE  \cite[Theorem~3.2]{BraChP96}.  

We are now left to show that $\G(u, \cdot)$ is subadditive for any fixed $u \in  \PC(\Omega')$. \EEE
 In particular, by  following the proof of \EEE \cite[Proposition~18.4]{DMLibro}, it is enough to show that   for any $u \in \PC(\Omega')$ \EEE
\begin{equation}\label{3103261311}
\G''(u, A'\cup B)\leq \G''(u, A)+ \G''(u,B)
\end{equation}
for every $A'$, $A$, $B\in \mathcal{A}(\Omega')$ with $A'\subset \subset A$,  where \EEE $\G''(\cdot, \widetilde{A}):=\Gamma$-$\limsup_{n\to \infty} \G_n(\cdot, \widetilde{A})$ for every $\widetilde{A}\in \mathcal{A}(\Omega')$. To prove \eqref{3103261311}, we follow the lines of \cite[Proposition~18.3]{DMLibro}:
let $u_n$, $v_n \in \PC(\Omega')$ converge in $L^1(\Omega')$ to $u  \in \PC(\Omega')$ \EEE such that
\begin{equation*}
\G''(u, A)=\limsup_{n\to \infty}\G_n(u_n, A), \quad \G''(u, B)=\limsup_{n\to \infty}\G_n(v_n, B). 
\end{equation*} 
Let  $\eta>0$. By Lemma~\ref{le:fundestGn}, taking  $\sigma_n= \|u_n-v_n\|_{L^1(\Omega')}^{1/2}$ for every $n\in \N$, there exist $\Lambda>0$, $\phi_n \in C_c^\infty(A;  [0,1]) \EEE $ with $\phi_n=1$ in a neighborhood of $\ol{A'}$, and $w_n \in \PC(\BBB A' \EEE \cup B)$ such that 
\begin{equation*}
\mathcal{G}_n(w_n, A'\cup B) \leq \BBB (1 + 2\omega(2\sigma_n))^2 \EEE (1+\eta) \big( \mathcal{G}_n(u_n,A) + \mathcal{G}_n(v_n,B)\big) + \frac{\Lambda}{\sigma_n}\|u_n-v_n\|_{L^1(S)}
\end{equation*}
and
\[
 {\|\phi_n u_n+(1-\phi_n)v_n-w_n\|_{L^\infty(\BBB A' \EEE \cup B)} \le \sigma_n,}
\]
for $S:= (A\sm A') \cap B$.
As \BBB $ u_n$ and $v_n$ converge to $u$ in $L^1(\Omega')$, we get   $w_n \to u$ in $L^1(\Omega')$.  Moreover, \EEE since  $\|u_n-v_n\|_{L^1( \Omega' \EEE )}\to 0$ as $n\to \infty$, it follows that $\sigma_n\to 0$, and therefore
\begin{equation*}
\G''(u, A'\cup B) \leq \limsup_{n\to \infty} \mathcal{G}_n(w_n, A'\cup B) \leq (1+\eta) (\G''(u, A)+\G''(u, B)).
\end{equation*}
By the arbitrariness of $\eta>0$, we confirm \eqref{3103261311}. 
\end{proof}

By a truncation argument we  can show that the proposition above also \BBB holds \EEE  for   (g1)--(g2) \EEE in place of (g1').  \EEE 
\begin{proposition}\label{prop:gammaconvgnmod}
Suppose that {\rm (g1)}--{\rm (g4)} hold and that $\G_n$  are defined as  in \eqref{fg}. Then,  up to a subsequence,   $\G_n$ $\ol \Gamma$-converge   to $\G$, where $\G(u,A)=\int_{ A \cap J_u \EEE} g(\cdot,[u], \nu_u) \,\dhn$  for $u \in \PC(\Omega')$ \EEE and $g$  is \EEE characterized by
 \eqref{3103262023}.
\end{proposition}

\begin{proof}
For every $L \ZZZZ \in  \N \EEE $ and $n\in \N$ \ZZZZ we define \EEE the functionals  $\mathcal{G}_n^L(u, A):= \mathcal{G}_n(u^L, A)$ if $u|_A \in \PC(A)$ and $ \ZZZZ \mathcal{G}^L_n(u, A)= \EEE +\infty$ otherwise, \EEE where $u^{L} := (u \wedge L) \vee (-L)$.  By Proposition~\ref{gammaconvgn} \ZZZZ and a diagonal argument, we find a subsequence (not relabeled) such that for all $L   \in \N   $ it holds that  $\G^L_n $ $\ol \Gamma$-converge  to a limit $\G^L$ which on $\PC(\Omega')$ is given by a surface energy with  density $g^L$ characterized by \eqref{3103262023}. In fact, we may ensure (g1'), where  the constant in the lower bound depends also on $L$,  additionally to  the constant in the lower bound in (g1). Due to monotonicity of $\mathcal{G}_n^L$ in $L$, we observe that the densities $g^L$ are increasing in $L$, i.e., $g := \sup_{L \in \N} g^L$ is well defined. We now deduce that  $\G_n $ $\ol \Gamma$-converge to a functional $\G$ with  $\G(u,A)=\int_{ A \cap J_u \EEE} g(\cdot,[u], \nu_u) \,\dhn$ for all $u \in \PC(\Omega')$. \EEE

In fact, for fixed $A \in \mathcal{A}(\Omega')$, \ZZZZ we first observe that by the  Monotone Convergence Theorem it holds that
\begin{equation}\label{3103261722}
\mathcal{G}(u,A)=\sup_{L >1} \mathcal{G}^L(u,A)= \lim_{L\to \infty}\mathcal{G}^L(u,A) \quad \text{for any $u \in \PC(\Omega')$}.
\end{equation}
Now,   if $\ZZZZ (u_n)_n \subset \EEE \PC(\Omega')$ converge in $L^1(\Omega')$ to $u \in \PC(\Omega')$, for any $\varepsilon>0$ there exists $\ZZZZ L \in \N \EEE $ such that $\G(u,A)-\varepsilon \ZZZZ \le \EEE   \G^L(u,A) \leq \liminf_n \mathcal{G}_n^L(u_n,A) \leq \liminf_n \mathcal{G}_n(u_n,A)$. \ZZZZ By arbitrariness of $\eps$ this shows the $\Gamma$-liminf inequality. Moreover, for fixed $u \in \PC(\Omega')$ and \EEE   for any \ZZZZ $L \in \N$ \EEE there exists $\ZZZZ (u_{n,L})_n \subset \EEE  \PC(\Omega')$ with $u_{n,L}\to u$ \ZZZZ (and thus  $u_{n,L}^L\to u^L$) \EEE such that $\lim_n  \mathcal{G}_n   (u_{n,L}^L,A) \ZZZZ  = \lim_n  \mathcal{G}^L_n   (u_{n,L},A)  \EEE =  \G^L(\ZZZZ u, \EEE A)$. \ZZZZ Therefore, by \eqref{3103261722} and a diagonal argument we conclude the existence of a recovery sequence for $u$. \EEE  
\end{proof}

\subsection{Integral representation for functionals   with  precrack}\label{sec: ppi3}
 In this subsection we  formulate and prove the $\Gamma$-convergence and integral representation results used in the proof of Theorem \ref{thm:convpreliminaries}. \EEE   As a preparation, we give the proof of Lemma~\ref{extra-lemma}. \EEE

 \begin{proof}[Proof of Lemma \ref{extra-lemma}]
Let us set $\mu_n:= \hn\mres_{K_n}$. Since $\hn(K_n)\leq C_0$, there exists $ \mu \in \M_b^+(\Omega')$ such that, up to a subsequence, $\mu_n \weaklystar \mu$ in $\ZZZZ \M^+_b \EEE (\Omega')$. By \ZZZZ \eqref{HHNN} \EEE and general properties of $\Gamma$-convergence we deduce that,  \ZZZZ for all $A \in \mathcal{A}(\Omega')$, \EEE
\begin{equation}\label{1004262009-neu}
\H^-(u,A) \leq \H^{\varepsilon,-}(u,A) \leq \H^-(u,A) + \varepsilon \,\mu(\ol A).
\end{equation}
Since $\mu(\Omega')\leq C_0$ and $h_\varepsilon^-$ are monotone increasing in $\varepsilon$ (thus converging as $\varepsilon\to 0$),
it holds that $\H^{\varepsilon,-}$   converge  pointwise  to $\H^-$ which \ZZZZ for  $ u \in \EEE \PC(\Omega')$ \EEE  is  represented by 
\begin{equation*}
\H^{-}(u,A)= \int_{A\cap J_u} h_\infty^{-}(\cdot, [u],\nu_u) \,\dhn,  
\end{equation*}
where  $h_\infty^-(x, \xi,\nu):= \lim_{\varepsilon\to 0} h^{\varepsilon,-}(x,\xi,\nu)$ \BBB for every  $x \in \Omega'$, $\xi \in \R$, and $\nu \in \mathbb{S}^{d{-}1}$. \EEE  Recalling the definition of  $h^-$  in \eqref{1004261901}, it now \EEE  holds that $h_\infty^-(x, \xi,\nu)= h^-(x,\xi,\nu)$
for every $\xi \in \ZZZZ \R \EEE$, $\nu \in \Sn$,  and for every $x \in \Omega'$ such that $H(x):= \limsup_{\varrho\to 0^+}  \varrho^{-(d{-}1)} \mu(\ol{B_\varrho(x)})<+\infty$.
In fact, from \eqref{1004262009-neu} it follows that $\mathbf{m}_{\H^-}^{\PC}(\ol u_{x,\xi,\nu}, Q^\nu_\varrho(x)) \leq \mathbf{m}_{\H^{\varepsilon,-}}^{\PC}(\ol u_{x,\xi,\nu}, Q^\nu_\varrho(x)) \leq \mathbf{m}_{\H^-}^{\PC}(\ol u_{x,\xi,\nu}, Q^\nu_\varrho(x)) + \varepsilon \mu(\ol Q^\nu_\varrho(x))$. Since $\mu(\ol Q^\nu_\varrho(x)) \leq \mu(\ol {B_{\sqrt{d}\varrho}(x)})$,    we conclude by dividing by $\varrho^{d{-}1}$ and taking the $\limsup$ on $\varrho\to 0^+$, \ZZZZ where we also use $h^{\varepsilon,-}(x, \xi,\nu)=\limsup_{\varrho \to 0^+}  \varrho^{-(d{-}1)} \mathbf{m}_{\H^{\varepsilon,-}}^{\PC}(\ol u_{x,\xi,\nu}, Q^\nu_\varrho(x))$. \EEE
 Since $H(x)<+\infty$ for $\hn$-a.e.\  $x \in \Omega'$ (use $\mu(\Omega')\leq C_0$ and  \cite[Theorem~2.56]{Ambrosio-Fusco-Pallara:2000}), the claim is concluded. 
 \end{proof}

  We now first come to  the  pure surface energy. \EEE 
\begin{proposition}\label{prop:Gammaconvgnminus}
\BBB Assume that {\rm (g1)}--{\rm (g4)} hold. \EEE 
There exists $\G^-\colon L^1(\Omega') \times \mathcal{A}(\Omega') \to [0,\infty]$ such that, up to a subsequence \ZZZZ (not relabeled), \EEE $\G_n^-$ defined in \eqref{gnm} $\ol \Gamma$-converge  to $\G^- $, with
\begin{equation*}
\G^-(u,A)= \int_{A \cap J_u} g^-(\cdot, [u],\nu_u)\,\dhn  \quad  \text{for } u \in \PC(\Omega'), \EEE
\end{equation*}
where
\begin{equation}\label{3103262002-neu}
{g^-(x, \xi, \nu)=\limsup_{\varrho \to 0^+} \frac{\mathbf{m}_{\mathcal G^-}^{\PC}(\ol u_{x,\xi,\nu}, Q^\nu_\varrho(x))}{\varrho^{d{-}1}} \quad  \text{ for all $x \in \Omega'$, $\xi \in \R$, and $\nu \in \mathbb{S}^{d{-}1}$.} \EEE  }
\end{equation}
\end{proposition}

\begin{proof}
 The statement follows by combining \BBB an abstract compactness result for $\Gamma$-convergence (see  \EEE \cite[Theorem 16.9]{DMLibro}) with  Proposition \ref{prop:gammaconvgnmod} and Lemma \ref{extra-lemma}   (with $\G_n^-$, $\G^-$ in place of $\H_n^-$, $\H^-$, the proof being analogous).  
\end{proof}

\begin{remark}\label{rem:propglimit}
By the asymptotic representation \ZZZZ formula \EEE  for $g$ in \eqref{3103262023}, arguing on the approximating cell problems as in \cite[Lemma~A.7]{Sto lavoro GSBV}, it can be proven that also $g$ satisfies (g1)--(g4) with the same $\alpha$, $\beta$,  and $\omega$. By the monotone approximation in the proof of  Lemma \ref{extra-lemma}, \EEE $g^-$ satisfies (g2), (g3), and (g4) with the same $\omega$.   Moreover, arguing as in the proof of Theorem~\ref{thm:convpreliminaries} (proof of second inequality in \eqref{12040901-eps} \ZZZZ and of \eqref{12040854}), \EEE it holds that 
\begin{equation}\label{15060711}
g^-(x, \xi, \nu)\leq \limsup_{\varrho\to 0^+}\liminf_{n\to \infty} \frac{\mathbf{m}_{\mathcal G_n^-}^{\PC}(\ol u_{x,\xi,\nu}, Q^\nu_\varrho(x))}{\varrho^{d{-}1}} \quad \text{ for all $x \in \Omega'$, $\xi \in \R$, and $\nu \in \mathbb{S}^{d{-}1}$.}
\end{equation} \EEE
\end{remark}

 We now come to the representation  of the bulk-surface energy. \EEE The proof follows the lines of \cite[Theorem~5.1]{GiacPonsi},  additionally \EEE taking into account the dependence on the jump amplitude of the surface energy. 
 \begin{theorem}\label{generalgammacon}
\BBB Assume that {\rm (g1)}--{\rm (g4)} hold. \EEE 
 Assume that \EEE the functionals $\mathcal{F}_n$ and $\mathcal{G}^-_n$ defined in   \eqref{fg}--\eqref{gnm} \EEE  $\ol \Gamma$-converge   to the functionals  
$\mathcal{F}$ and $\mathcal{G}^-$ given in \eqref{flimit} and  Proposition~\ref{prop:Gammaconvgnminus}, respectively. \EEE  Then,
 the functionals $\mathcal{E}^-_n $ defined by 
$$\mathcal{E}^-_n(u, A)  =   \int_A f_n(x,\nabla u(x)) \, {\rm d}x  + \int_{A \cap J_u} \big( g_n(\cdot,[u],\nu_u) - g_n(\cdot,\varphi_n,\nu_u) \big)^+ \, {\rm d}\Hd   $$
\BBB if  $u|_A \EEE \in \SBV^p(A)$ and   $+\infty$ otherwise in $L^1(\Omega')$, $\ol \Gamma$-converge  to $\mathcal{E}^- $ which satisfies 
\begin{equation}\label{14042304}\mathcal{E}^-(u, A)  =    \int_A f(x,\nabla u(x)) \, {\rm d}x  +  \int_{ A \cap J_u \EEE} g^-(\cdot,[u],\nu_u) \, {\rm d}\Hd  \quad  \text{for } \EEE u \in \SBV^p(\Omega').
\end{equation}

 \end{theorem}  
   We emphasize that we use the same notation $\mathcal{E}^-_n$ as in  \eqref{samenotation}, although the functionals are in general different. In fact, here we consider general sequences of densities $(g_n)_n$. \BBB Note that, in this formulation, we can also  encompass the penalization term on $\Omega' \setminus \overline{\Omega}$. \EEE

 \begin{proof}
By abstract theory of $\Gamma$-convergence (see \cite[Theorem~16.9]{DMLibro}), up to a subsequence \ZZZZ (not relabeled) \EEE the functionals $\mathcal{E}_n^- $ $\ol \Gamma$-converge to a functional $\mathcal{E}^- $. Moreover, it holds that $\mathcal{E}^-(\cdot, A)$ is finite on $\SBV^p$  for $A \in \A(\Omega')$  by the growth assumptions on $f_n$ and $g_n$  (see also Remark \ref{rem:propglimit}), \EEE with
\begin{equation}\label{1304260714}
\E^-(u,A) \leq \F(u,A)\quad \text{for }u \in W^{1,p}(A), \qquad \E^-(u,A) \leq \beta |A| + \G^-(u,A) \quad \text{for }u \in \PC(A).
\end{equation}
We notice that, even if we  work on a suitable subsequence, as we are going to show that the  $\ol\Gamma$-limit \EEE has the form \eqref{14042304}, \ZZZZ by \EEE  Urysohn's lemma \EEE the whole sequence  \ZZZZ  $( \E^-_n)_n$ \EEE $\ol\Gamma$-converges. \EEE
 
We proceed as in \cite[(4.23)]{GiacPonsi}: for   $\varepsilon>0$ we introduce the functionals $\mathcal{E}^{\varepsilon,-} \colon L^1(\Omega') \times \mathcal{A}(\Omega') \to [0,+\infty]$    by
\begin{equation*}
\mathcal{E}^{\varepsilon,-}(u,A):=\mathcal{E}^-(u,A) + \varepsilon \int_{A \cap J_u} (1+|[u]|) \,\dhn \quad \text{if } \BBB u|_A \EEE  \in \SBV^p(A)
\end{equation*}
and $+\infty$ otherwise.
Then
$\mathcal{E}^{\varepsilon,-}$ converge to $\mathcal{E}^-$ pointwise as $\varepsilon\to 0$ \ZZZZ for all $u \in \SBV^p(\Omega')$. \EEE
%
 By the integral representation result 
\cite[Theorem~1]{BouFonLeoMas02},
 for $u \in \SBV^p(\Omega')$ 
 it holds that
 \begin{equation*}
 \E^{\varepsilon,-}(u,A)= \int_A f^\varepsilon_\infty(x, \nabla u(x)) \dx+ \int_{A \cap J_u} g^\varepsilon_\infty(\cdot, u^-, u^+, \nu_u) \,\dhn
\end{equation*}  
 with, \ZZZZ  for all  $x \in \Omega'$, $\zeta \in \R^d$, $\BBB \xi^-, \EEE \xi^+ \in \R$, and $\nu \in \mathbb{S}^{d{-}1}$, \EEE
 \begin{equation}\label{0104261748}
 f^\varepsilon_\infty(x,\zeta)=\limsup_{\varrho\to 0^+} \frac{\mathbf{m}_{\E^{\varepsilon,-}}(l_{\zeta,x}, Q^\nu_\varrho(x))}{\varrho^{d{-}1}}, \qquad g^\varepsilon_\infty(x,\xi^-, \xi^+,\nu)=\limsup_{\varrho\to 0^+} \frac{\mathbf{m}_{\E^{\varepsilon,-}}(\ol u_{x,\xi^-, \xi^+,\nu}, Q^\nu_\varrho(x))}{\varrho^{d{-}1}},
 \end{equation}
where $\mathbf{m}_{\E^{\varepsilon,-}}\equiv \mathbf{m}_{\E^{\varepsilon,-}}^{L^1}$ and  $\ol u_{x,\xi^-, \xi^+,\nu}:=  \xi^-    +   (\xi^+ - \xi^-) \chi_{Q^{\nu,+}_\varrho(x)}$.  
Passing to the limit in $\varepsilon$, by the pointwise convergence of $\mathcal{E}^{\varepsilon,-}(u,A)$ to $\mathcal{E}^-(u,A)$ and since the formulas in \eqref{0104261748} are increasing in $\varepsilon$ (thus converging as $\varepsilon\to 0$),  we then get \EEE
\begin{equation*}
 \E^-(u,A)= \int_A f_\infty(x, \nabla u(x)) \dx+ \int_{A \cap J_u} g_\infty( \ZZZZ\cdot, \EEE u^-, u^+, \nu_u) \,\dhn,
\end{equation*}
with
\begin{equation}\label{ginftyeps}
f_\infty(x, \zeta)=\lim_{\varepsilon\to 0} f^\varepsilon_\infty(x,\zeta), \quad \quad  g_\infty(x,\xi^-,\xi^+, \nu)=\lim_{\varepsilon\to 0} g^\varepsilon_\infty(x,\xi^-, \xi^+, \nu)
\end{equation}
 for $x \in \Omega'$, $\zeta \in \R^d$, $\BBB \xi^-, \EEE \xi^+ \in \R$, and $\nu \in \mathbb{S}^{d{-}1}$. \EEE Therefore, as in \cite[Theorem 5.1]{GiacPonsi}, the goal is to prove that 
\begin{equation}\label{0104261834}
f_\infty(x, \nabla u(x))=f(x, \nabla u(x)) \text{ for a.e.\ }x\in \Omega', 
\end{equation}
and
\begin{equation}\label{0104261835}
g_\infty(x,u^-(x),u^+(x),\nu_u(x))= g^-(x,[u](x),\nu_u(x)) \text{ for }\hn\text{-a.e.\ }x \in J_u
\end{equation}
for every $u \in \SBV^p(\Omega')$.
The identity \eqref{0104261834} follows  exactly as in \EEE  \cite[Steps 1,2 in Theorem 5.1]{GiacPonsi}: indeed, in this part of the proof, the only relevant properties on the surface energy are  $\hn(K_n)\leq C_0$
and the nonnegativity of the surface density of $\mathcal{E}_n^-$. 
To prove \eqref{0104261835}, we show both inequalities separately.

\emph{Inequality $g_\infty(x,u^-(x),u^+(x),\nu_u(x))\leq g^-(x,[u](x),\nu_u(x))$ for $\mathcal{H}^{d{-}1}$-a.e.\ $x \in J_u$}: This can be shown as in \cite[Step 3 in Theorem 5.1]{GiacPonsi}. We report here the argument for convenience \BBB of the reader. \EEE  Denoting $\mu_n:= \hn\mres_{K_n}$, by $\hn(K_n)\leq C_0$, there exists
$\mu\in \M_b^+(\Omega')$ such that, up to a subsequence, $\mu_n \weaklystar \mu$ in $\ZZZZ \M^+_b \EEE (\Omega')$.
For every $u \in \SBV^p(\Omega')$ and $A \in \A(\Omega')$ we have, using (g1),
\begin{equation*}
\alpha \,  \hn(A \cap (J_u\sm K_n))\leq \G_n^-(u,A), 
\end{equation*}
so that $\alpha \,  \hn(A \cap J_u)\leq \G_n^-(u,A) + \alpha \, \mu_n(A)$. Passing to the $\Gamma$-limit of $\G_n^-(\cdot,A)$ and using the weak$^*$-convergence of $\mu_n$ to $\mu$, we get
\begin{equation}\label{0104261904}
\alpha \, \hn(A \cap J_u)\leq \G^-(u,A) + \alpha \,\mu(\ol A).
\end{equation}
For brevity, we write $\nu$, $u^\pm$ in place of $\nu_u(x)$, $u^\pm(x)$ in the sequel, as well as $[u] = u^+ - u^-$. 
Given $\varepsilon>0$, for every $\varrho>0$ let $u_{\varepsilon,\varrho} \in \PC(Q_\varrho^\nu(x))$ be  such that  $u_{\varepsilon,\varrho} =  \ol u_{x,u^-, u^+,\nu} \EEE $ near  $\partial Q_\varrho^\nu(x)$ and  \EEE
\begin{equation}\label{1304260723}
\G^-(u_{\varepsilon,\varrho}, Q_\varrho^\nu(x)) \leq \mathbf{m}_{\G^-}^{\PC}(\ol u_{x, [u], \nu}, Q_\varrho^\nu(x)) + \varepsilon \varrho^{d{-}1}.
\end{equation}
 Here, we  used \EEE that $\G^-$ is translation invariant. We observe that, by truncation and (g2),   we may assume that $u^-\wedge u^+\leq u_{\varepsilon,\varrho} \leq u^-\vee u^+$. Thus, 
$\E^{\varepsilon,-}(u_{\varepsilon,\varrho},Q_\varrho^\nu(x))\leq  \E^{-}(u_{\varepsilon,\varrho},Q_\varrho^\nu(x))+ \varepsilon(1+|[u]|) \hn(J_{u_{\varepsilon,\varrho}} \cap Q_\varrho^\nu(x))$.
By \eqref{1304260714}, \eqref{0104261904}, and \eqref{1304260723}  we get 
\begin{equation*}
\begin{split}
\E^{\varepsilon,-}(u_{\varepsilon,\varrho},Q_\varrho^\nu(x))&\leq \beta \varrho^d +  \G^-(u_{\varepsilon,\varrho}, Q_\varrho^\nu(x)) + \varepsilon (1+|[u]|)\hn(J_{u_{\varepsilon,\varrho}} \cap Q_\varrho^\nu(x))
\\&
\leq \beta \varrho^d + \Big(1+\frac{\varepsilon}{\alpha}(1+|[u]|)\Big) \big(\mathbf{m}_{\G^-}(\ol u_{x,[u],\nu}, Q_\varrho^\nu(x)) + \varepsilon \varrho^{d{-}1}\big) + \varepsilon (1+|[u]|) \mu(\ol{Q_\varrho^\nu(x)}).
\end{split}
\end{equation*}
Thus, for $x$ such that 
$H(x):=\limsup_{\varrho\to 0^+} \varrho^{-(d{-}1)} \mu(\ol{B_\varrho(x)})<+\infty$   (which holds for $\hn$-a.e.\ $x\in \Omega'$, see e.g.\ \cite[Theorem~2.56]{Ambrosio-Fusco-Pallara:2000}), \BBB dividing by \ZZZZ $\varrho^{d{-}1}$ \EEE and passing to the limit $\ZZZZ \varrho \EEE \to 0$,  by \eqref{3103262002-neu} \EEE and \eqref{0104261748} we deduce
\begin{equation*}
g^\varepsilon_\infty(x,u^-, u^+,\nu) \leq \Big(1+\frac{\varepsilon}{\alpha}(1+|[u]|) \Big) \big(g^-(x, [u], \nu) + \varepsilon \big) + \varepsilon (1+|[u]|)  \sqrt{d}^{d{-}1} \EEE  H(x).
\end{equation*}
We  then conclude by \eqref{ginftyeps} and the arbitrariness of $\varepsilon>0$ and since $|[u]|(x)<+\infty$ for $\hn$-a.e.\ $x \in J_u$.

\emph{Inequality $g_\infty(x,u^-(x),u^+(x),\nu_u(x))\geq g^-(x,[u](x),\nu_u(x))$ for $\mathcal{H}^{d{-}1}$-a.e.\ $x \in J_u$}:
Since $u \in \SBV^p(\Omega')$ and thus $\E^-(u,\Omega')<+\infty$, for $\hn$-a.e.\ $x \in J_u$ it holds that
\begin{equation}\label{01041949}
g_\infty(x,u^-(x),u^+(x),\nu_u(x))=\lim_{\varrho\to 0^+} \frac{\E^-(u, Q_\varrho^{\nu_u(x)}(x))}{\varrho^{d{-}1}} < + \infty.
\end{equation}
Moreover, for \BBB $(u_n)_n \subset \SBV^p(\Omega')$ \EEE such that $\E^-(u,\Omega')=\lim_{n \to \infty} \E_n^-(u_n,\Omega')$, we define $\mu_n:=\hn\mres_{J_{u_n}}$ and we notice that, by  (g1) \EEE
along with $\hn(K_n)\leq C_0$, $\mu_n(\Omega')$ are \BBB uniformly \EEE bounded. Thus,  up to a subsequence, $\mu_n \weaklystar \mu$ in $\ZZZZ \M^+_b \EEE (\Omega')$ for $\mu \in \M_b^+(\Omega')$.
 Then $H(x):=\limsup_{\varrho\to 0^+} \varrho^{-(d{-}1)}\mu(\ol{{Q}_\varrho (x)}) <+\infty$ for $\hn$-a.e.\ $x \in \Omega'$.

Fix $x$ satisfying \eqref{01041949}, such that $u^\varrho(y):=u(x+\varrho y)  - u^-  \to \ol u_{x,[u],\nu} \EEE $ in $L^1(Q^\nu_1)$ (for $Q^\nu_1=Q^\nu_1(0)$) and $H(x)<+\infty$ \ZZZZ holds, \EEE again using $u^\pm$, $[u]$, and $\nu$ as shorthand notation. (Notice that these hold for $\hn$-a.e.\ $x \in J_u$.)  By \eqref{15060711}, for every $i \in \N$ there exist $\varrho_i\to 0^+$ and $n_i \in \N$ such that 
\begin{equation}\label{0104262330}
g^-(x,[u],\nu)\leq\liminf_{i\to \infty} \frac{\mathbf{m}^{\PC}_{\G^-_{n_i}}(\ol u_{x,[u],\nu}, Q_{\varrho_i}^\nu(x))}{\varrho_i^{d{-}1}}.
\end{equation} 
 Moreover,  we may also assume that
$\E^-(u, \partial Q_{\varrho_i}^\nu(x)))=0$. Then,
 \EEE in view of \cite[Remark 4.6]{GiacPonsi}, by the choice of $x$, \ZZZZ and by a change of variables, we get 
\begin{align}\label{0104262326}
\frac{\E^-(u, Q_{\varrho_i}^\nu(x))}{\varrho_i^{d{-}1}}&\geq\frac{\E_{n_i}^-(u_{n_i}, Q_{\varrho_i}^\nu(x))}{\varrho_i^{d{-}1}} -\frac{1}{i} \geq  \frac{\G_{n_i}^-(u_{n_i}, Q_{\varrho_i}^\nu(x))}{\varrho_i^{d{-}1}} -\frac{1}{i}
\\&
= \int_{Q_1^\nu \cap J_{v_i} }  \big(g_{n_i}(x + \varrho_i y, [v_i](y), \nu_{v_i}(y)) - g_{n_i}(x+\varrho_i y, \widetilde{\varphi}_i \BBB (y), \EEE \nu_{v_i}(y)) \big)^+ \,\dhn(y) -\frac{1}{i},\notag
\end{align}
and 
\begin{equation}\label{12041849}
v_i \to  \ZZZZ \ol u_{0,[u],\nu} \EEE \quad\text{in }L^1(Q^\nu_1),
\end{equation} where 
$v_i(y):=u_{n_i}(x+\varrho_i y)  - u^- \EEE $, $\widetilde{\varphi}_i(y):=\varphi_{n_i}(x+\varrho_i y)$  for  $y \in Q_1^\nu$  and \EEE $\widetilde{K}_i:= \frac{(K_{n_i} \cap Q_{\varrho_i}^\nu(x))-x}{\varrho_i}$. \ZZZZ Without restriction we assume that $  [u] \EEE >0$ and, by a truncation argument, that $v_i$ maps into $[0,[u]]$. \EEE By \eqref{eq: general bound} and a change of variables we find
\begin{equation*}
\int_{Q_1^\nu} |\nabla v_i|^p \,\mathrm{d}y= \varrho_i^{p-d} \int_{Q_{\varrho_i}^\nu(x)} |\nabla u_{n_i}|^p \dx\leq \frac{\varrho_i^{p-1}}{\alpha} \frac{\E^-_{n_i}(u_{n_i}, Q_{\varrho_i}^\nu(x))}{\varrho_i^{d{-}1}} \leq C \varrho_i^{p-1},
\end{equation*}
where we used \eqref{01041949}  and \eqref{0104262326} \EEE in the last estimate. Consequently, by H\"older's inequality  
\begin{equation}\label{01042233}
\int_{Q_1^\nu} |\nabla v_i| \,\mathrm{d}y\leq C \varrho_i^{1-1/p}=  C \EEE \sigma_i^2, \quad  \text{where } \EEE  \sigma_i:= \varrho_i^{\frac{1}{2}(1-1/p)}.
\end{equation}
We now 
argue as in the proof of Lemma~\ref{le:fundestGn}   
to deduce  
\begin{equation}
\int_{Q_1^\nu} |\nabla v_i| \,\mathrm{d}y  \geq \EEE   \ZZZZ \int_{0}^{[u]} \EEE \hn\big(\partial^*\{v_i>t\} \ZZZZ \cap (Q_1^\nu \sm J_{v_i}) \EEE  \big)\,\mathrm{d}t\geq \sigma_i  \sum_{l=0}^{C_{\sigma_i}-1} \EEE \hn \big(\partial^*\{v_i>t_l\}  \ZZZZ \cap (Q_1^\nu \sm J_{v_i}) \EEE  \big)
\end{equation}
for $C_{\sigma_i}= \lfloor\frac{ [u] \EEE }{\sigma_i} \rfloor$ and suitable $\BBB t_l \in \EEE     (\sigma_i\, l, \sigma_i (l+1))$ for $l \in \{0,\dots, C_{\sigma_i}-1\}$.
Then, we  set  
\BBB 
\begin{equation}\label{0104262246}
w_i:=  \sum_{l=1}^{C_\sigma-1} t_l \chi_{\{t_{l-1}< v_i \leq t_l\}} +  [u] \chi_{\{v_i> t_{C_\sigma-1}\}}.
\end{equation}
\EEE Therefore, $w_i \in \PC(Q_1^\nu; \BBB [0,[u]])\EEE$,  
$\|w_i-v_i\|_\infty\leq 2 \sigma_i$, 
$\hn(J_{w_i}\sm J_{v_i})\leq {C}\sigma_i$ (by \eqref{01042233}--\eqref{0104262246}), as well as $|[w_i]-[v_i]|\leq 4 \sigma_i$ on $J_{v_i} \cap J_{w_i} $ and $|[w_i]|\leq 2\sigma_i$ on $J_{w_i}\sm J_{v_i}$. By (g1) \ZZZZ and \EEE  (g4)  we get 
\begin{equation}\label{0104262328}
\begin{split}
\int_{Q_1^\nu \cap J_{w_i}} g_{n_i}(x + \varrho_i y, [w_i](y), \nu_{w_i}(y)) \,\dhn(y) &\leq \int_{Q_1^\nu \cap J_{v_i}} g_{n_i}(x + \varrho_i y, [v_i](y), \nu_{v_i}(y)) \,\dhn(y) \\&
\hspace{1em}+ \BBB \gamma_i \EEE \hn(J_{v_i}) + \beta (1+   2 \sigma_i   ) {C}\sigma_i,
\end{split}
\end{equation}
\BBB where for shorthand we set $\gamma_i := \omega( 4 \sigma_i)2\beta(1 + |[u]|  )$.  Eventually, we apply Lemma~\ref{le:fundestGn}   for $u$, $v$, \ZZZZ $A'$ $A$, $B$, \EEE and $\sigma$ therein equal to $w_i$, $\ZZZZ \ol u_{0,[u],\nu} \EEE$, \ZZZZ $(1-2\varrho_i) Q^\nu_1$, $(1-\varrho_i) Q^\nu_1$, $Q^\nu_1\sm (1-2\varrho_i) Q^\nu_1$,  \EEE  and $\sqrt{\tau_i}$, where \EEE $\tau_i:=\|w_i - \ZZZZ \ol u_{0,[u],\nu} \EEE \ZZZZ \|_{L^1(Q_1^\nu)} \EEE$.  Accordingly, \EEE  $\eta=\eta_i$ \BBB is   chosen such that $\eta_i \to 0$ slowly enough to guarantee \EEE  $\Lambda_i \tau_i^{1/2}\to 0$,  where  $\Lambda_i$ is the constant associated to $\eta_i$: \EEE  this allows us to find a modification $\widetilde{w}_i$ of $w_i$ such that $\widetilde{w}_i= \ZZZZ \ol u_{0,[u],\nu} \EEE$ \BBB in \EEE a neighborhood of $\partial Q^\nu_1$  and \EEE $\widetilde{w}_i \to \ZZZZ \ol u_{0,[u],\nu} \EEE$ in $L^1(Q^\nu_1)$,   where we use \eqref{12041849} as well as $\|w_i-v_i\|_\infty\leq 2 \sigma_i$. Moreover,  \EEE  recalling  (g1) we get 
\begin{equation*}\label{12042357}
\begin{split}
\int_{Q_1^\nu \cap J_{\widetilde{w}_i}}& \hspace{-0.15cm} g_{n_i}(x + \varrho_i y, [\widetilde{w}_i](y), \nu_{\widetilde{w}_i}(y)) \,\dhn(y)\leq \BBB (1+2\omega(2\sqrt{\tau_i}))^2 \EEE (1 + \eta_i)\int_{Q_1^\nu \cap J_{w_i}}  \hspace{-0.15cm} g_{n_i}(x + \varrho_i  \cdot, \EEE [w_i], \nu_{w_i}) \,\dhn
\\&
+  \BBB (1+2\omega(2\sqrt{\tau_i}))^2 \EEE(1 + \eta_i) \beta (1+ |[u]|) \EEE  \, (1-(1-\BBB 2 \EEE \varrho_i)^{d{-}1})  + \Lambda_i \tau_i^{1/2}.
\end{split}
\end{equation*}
By rescaling, the functions $z_i(\zeta):= \widetilde{w}_i((\zeta-x)/{\varrho_i})$ are in $\BBB \PC \EEE (Q_{\varrho_i}^\nu(x))$, $z_i= \ol u_{x,[u],\nu_u}$ on $\partial Q_{\varrho_i}^\nu(x)$,   and   
\begin{equation*}\label{0104262321}
\begin{split}
\int_{Q_1^\nu \cap J_{ \widetilde{w}_i } }&  \Big(g_{n_i}(x + \varrho_i y, [\widetilde{w}_i](y), \nu_{\widetilde{w}_i}(y)) - g_{n_i}(x+\varrho_i y, \widetilde{\varphi}_i\BBB (y), \EEE \nu_{\widetilde{w}_i}(y))\Big)^+  \,\dhn(y)= \frac{\G_{n_i}^-(z_{i}, Q_{\varrho_i}^\nu(x))}{\varrho_i^{d{-}1}} 
\\&\geq \frac{\mathbf{m}^{\PC}_{\G^-_{n_i}}(\ol u_{x,[u],\nu_u}, Q_{\varrho_i}^\nu(x))}{\varrho_i^{d{-}1}}.
\end{split}
\end{equation*}
We conclude $g_\infty(x,u^-(x),u^+(x),\nu_u(x))\geq g^-(x,[u](x),\nu_u(x))$ by collecting \eqref{01041949}--\eqref{0104262326}, and the estimates from  \eqref{0104262328} on,  \EEE
passing to the limit in $i$, \ZZZZ and \EEE  using that $\sigma_i,\tau_i,  \Lambda_i \tau_i^{1/2} \to 0$. \EEE Here, we also used $H(x) <  \infty$ to ensure that $\sup_{i \in \N} \hn(J_{v_i})  <   \infty$ by a change of variables. 
\end{proof}  
\begin{remark}\label{rem:continuityf} The previous analysis could be generalized to a vector valued setting, i.e., for functionals defined on $u\colon \Omega'\to \R^m$, $m>1$, by following the strategy in \cite{Sto lavoro GSBV}.  This is \EEE based on suitable truncations and an application of the Coarea Formula on components $u^i$  (see also \cite{BraDeFVit97}). Further, everything could be accomplished in the generalized setting of $\GSBV(\Omega'; \R^m)$ functions. The assumptions on $f_n$ and $g_n$ are   analogous, \EEE  but generalized to the vector-valued case ($\xi \in \R^m$);  the only condition that is not readily generalized is the monotonicity (g2), which is replaced either by
\begin{itemize}
\item[(g2)$^m$] $g_n(x, \xi_1,\nu)\leq g_n(x, \xi_2, \nu)$ for every $x \in \Omega'$, $\nu \in \Sn$,  $\xi_1, \, \xi_2 \in \R^m \sm \{0\}$ with $|\xi_1|\leq |\xi_2|$
\end{itemize}
or, more generally, by the combination of the following, for some $c>0$:
\begin{itemize}
\item[(g2a)$^m$] $g_n(x, \xi_1,\nu)\leq c\, g_n(x, \xi_2, \nu)$ for every $x \in \Omega'$, $\nu \in \Sn$,  $\xi_1, \, \xi_2 \in \R^m \sm \{0\}$ with $|\xi_1|\leq |\xi_2|$;
\item[(g2b)$^m$] $g_n(x, \xi_1,\nu)\leq g_n(x, \xi_2, \nu)$ for every $x \in \Omega'$, $\nu \in \Sn$,  $\xi_1, \, \xi_2 \in \R^m \sm \{0\}$ with $c\, |\xi_1|\leq |\xi_2|$.
\end{itemize}
\end{remark}
\end{appendices}

\end{document}